\documentclass[12pt,twoside]{amsart}

\usepackage{tikz-cd}
\usepackage[colorlinks,linkcolor=blue,bookmarksnumbered,pagebackref,linktocpage,final]{hyperref}
\usepackage{enumitem}
\usepackage[olditem,oldenum]{paralist}

\usepackage[initials,msc-links]{amsrefs} \usepackage{derivative}
\usepackage{amssymb}

\usepackage{marginnote}
\usepackage{mathtools}

\newcommand{\kh}{\mathfrak{h}}
\newcommand{\kk}{\mathfrak{k}}
\newcommand{\kg}{\mathfrak{g}}
\newcommand{\kt}{\mathfrak{t}}

\newcommand{\kq}{\mathfrak{q}}
\newcommand{\rss}{\mathrm{ss}}
\newcommand{\rs}{\mathrm{s}}
\newcommand{\Ad}[2]{\operatorname{Ad}_{#1}(#2)}
\newcommand{\Ads}[2]{\operatorname{Ad}^*_{#1}(#2)}
\DeclareMathOperator{\diff}{d}
\DeclareMathOperator{\diffc}{\diff^{c}}
\newcommand{\bfa}{\mathbf{a}}
\newcommand{\bfb}{\mathbf{b}}

\newcommand{\bfZ}{\mathbf{Z}}
\newcommand{\bR}{\mathbb{R}}
\newcommand{\bC}{\mathbb{C}}

\newcommand{\sJ}{\mathsf{J}}

\newcommand{\T}{\operatorname{T}}
\newcommand{\grad}{\operatorname{grad}}

\newcommand{\id}{\operatorname{id}}

\newcommand{\pr}{\operatorname{pr}}
\newcommand{\Lie}{\operatorname{Lie}}
\newcommand{\Hess}{\operatorname{Hess}}
\newcommand{\ind}{\operatorname{Ind}}
\newcommand{\kn}{\mathsf{f}}
\newcommand{\cat}{\mathsf{CAT}(0)}
\newcommand{\tits}{\angle_{\mathrm{T}}}
\newcommand{\minp}{\mathbf{a}_{\mathrm{min}}}

\newcommand{\wm}{w_{\mathrm{min}}}
\newcommand{\wh}{w_{\mathrm{H}}}

\newcommand{\CE}{\mathsf{M}_{n,m,\tau}}

\newfont{\Bfmit}{eufm10 scaled\magstep1}

\def\qed {\nobreak$\quad$\lower 1pt\vbox{
    \hrule
    \hbox to 8pt{\vrule height 8pt\hfil\vrule height 8pt}
      \hrule}\ifmmode\relax\else\par\medbreak \fi}

\newtheorem{thm}{Theorem}[section]

\newtheorem{lem}[thm]{Lemma}
\newtheorem{cor}[thm]{Corollary}

\newtheorem{prop}[thm]{Proposition}

\theoremstyle{definition}

\newtheorem{defn}[thm]{Definition}

\newtheorem{exmp}[thm]{Example}

\newtheorem{defnprop}[thm]{Definition-Proposition}

\theoremstyle{remark}
\newtheorem{rem}[thm]{Remark}

\newcommand{\norm}[1]{\left\Vert#1\right\Vert}

\numberwithin{equation}{section}

\usepackage{perpage} \MakePerPage{footnote}

\definecolor{ta}{HTML}{002FA7}

\makeatletter
\def\l@subsection{\@tocline{2}{0pt}{2.5pc}{5pc}{}}
\makeatother

\begin{document}
\title[Reduction and Complex Quotients]
{Generalized Moment Maps, Reduction\\ and Complex Quotients}
\author{Yi Hu and Xiangsheng Wang}
\address{Department of Mathematics, The University of Arizona, Tucson, AZ 85721, USA}
\email{yhu@arizona.edu}
\address{School of Mathematics, Shandong University, Jinan, Shandong 250100, P.R. China}
\email{xiangsheng@sdu.edu.cn}

\maketitle

\begin{abstract} 
  In this note, we introduce the concept of momentumly closed forms.
  A non-degenerate momentumly closed two-form and its generalized moment map are the generalization of two well-known notions, symplectic forms and moment maps, in the almost Hermitian setting.
  We then generalize the classical theory of moment maps to this broader framework.
  As a first step, we prove a variant of the Darboux--Weinstein theorem for non-degenerate momentumly closed two-forms.
  Based on this, we further establish the convexity property of the generalized moment map, construct the corresponding reduction space and investigate the properties of the Kirwan--Ness stratification.
\end{abstract}

\tableofcontents

\section{Introduction}
\label{sec:introduction}
The concept of a moment map originates from mathematical physics, where it plays a central role in the definition of Hamiltonian manifolds.
At its core, the moment map is a notion defined on real manifolds.
However, in K\"ahler geometry, it serves as a bridge between symplectic and complex geometry, giving rise to many deep and beautiful results.
Notably, it features prominently in a series of celebrated works, see, for example, ~\cite{Atiyah, Atiyah_1983ya, Guillemin_1982aa, Kempf_1979aa, Kirwan84}.

In this note, we discuss a generalization of the moment map for certain almost Hermitian manifolds.
In particular, for Hermitian manifolds, we extend the notions of symplectic moment map, symplectic reduction, and K\"ahler quotient to their Hermitian counterparts: the generalized moment map, reduction space, and Hermitian quotient.

Let $M$ be a almost complex manifold acted upon by a compact group $K$.
Let $g$ be a $K$-invariant almost Hermitian metric on $M$ and $\omega$ be its fundamental two-form.
Let $\xi_M$ be the fundamental vector field generated by $\xi \in {\frak k}=\Lie K$.
The two-form $\omega$ is called \emph{momentumly closed} if the contraction $\iota_{\xi_M} \omega$ is closed for all $\xi \in {\frak k}$.
Adopting a term in the equivariant de Rham cohomology theory, $\omega$ is momentumly closed if and only if $\diff \omega$ is a {\it basic} form.
Solving the equation $\diff \langle\Psi, \xi\rangle = \iota_{\xi_M} \omega,$ one is to obtain a map $\Psi: M \rightarrow {\frak k}^*$ which we call a {\it generalized moment map}, see Definition~\ref{def:gmp}.

The main purpose of this note is to demonstrate that, aside from those statements which inherently require the closedness of a symplectic form $\omega$, many results concerning the symplectic moment map continue to hold for the generalized moment map.
Along the way, we also establish several new results.

For example, we will show that many fundamental theorems in symplectic geometry, such as the Atiyah--Guillemin--Sternberg abelian convexity theorem, Kirwan’s non-abelian convexity theorem, results concerning the norm-square of the moment map, the correspondence between complex and symplectic quotients, and the Duistermaat--Heckman theorem, continue to hold in this generalized setting.
In many cases, these results are straightforward extensions of their symplectic counterparts, though not always.
Notably, the generalization of the Duistermaat--Heckman theorem requires more careful consideration: additional conditions must be introduced to compensate for the lack of closedness; see Theorem~\ref{DH}.
In Theorem \ref{thm:darboux}, we prove
a variant of the Darboux--Weinstein theorem for the momentumly closed two-form.
Another new contribution is presented in Section~\ref{dynastable}, where we prove that complex quotients always arise as reduction spaces; see Theorem~\ref{thm:quotient}.

As for the proofs of these results, we find that, in many cases, the arguments used in the symplectic setting still apply here, if one checks the details carefully. However, verifying these proofs in the new setting line by line can be tedious, and we aim to avoid such repetition where possible. Instead, we adopt an approach that highlights a common structural pattern underlying these arguments,  an approach motivated by new insights we discovered while preparing this note. We hope this perspective not only streamlines the exposition, but also helps the reader better understand why the closedness condition is often nonessential, and, in some cases, sheds light on potential issues in the classical setting.

To elaborate this approach, we recall that a typical proof of a result concerning moment maps is often divided into two parts:
\begin{enumerate}
\item establishing certain local properties of the moment map, which usually rely on the closedness of $\omega$;
\item evoking additional techniques that are independent of the closedness condition.
\end{enumerate}

To understand why many results about moment maps continue to hold in the generalized setting, a key observation is that there exists a unified answer to (1) for generalized moment maps. Indeed, in Section~\ref{sec:gmm}, we show that 
\emph{the local properties of the generalized moment map are just as good as those of the classical one}, by establishing new variants of classical results, 
namely, the Darboux--Weinstein theorem and the isotropic embedding theorem (Theorems~\ref{thm:darboux} and~\ref{thm:iso-emb}).

As for (2), the techniques involved naturally vary depending on the specific problem.
In some cases, such as the convexity problem, well-developed tools are readily available and can be applied directly.
However, in other cases, such as the Kirwan--Ness stratification, the relevant techniques are less developed.
To address this, we include in Section~\ref{sec:kn-fct} a detailed discussion of the properties of the Kempf--Ness function, with the aim of also clarifying certain aspects in the classical setting.
As noted earlier, this part is independent of whether $\omega$ is closed or not.
Given the important role of the Kempf--Ness function in geometric invariant theory (GIT), we hope the material in this section will serve as a useful reference even in the K\"ahler setting.

There are many more extensions of symplectic theorems than can be covered in this note. For example, we omit discussion of the Hermitian cut, the analogue of Lerman's symplectic cut.

By definition, momentumly closed forms are far more abundant than symplectic forms.
A paradox, however, is the apparent scarcity of significant examples.
This may be attributed to the fact that, on the one hand, the space of K\"ahler manifolds forms a ``measure zero'' subset within the space of all complex manifolds; yet, on the other hand, most complex manifolds we commonly encounter happen to be K\"ahler.

Despite this, it is worth emphasizing that the new theory remains meaningful even when $M$ is K\"ahler.
A key point here is that certain quotients of a K\"ahler manifold may fail to be K\"ahler themselves, yet naturally carry a Hermitian structure (see Theorem~\ref{thm:quotient}).
In such cases, these new quotients are nevertheless very close to being K\"ahler, for instance, they are Moishezon when $M$ is projective.
See also~\cite{Kollar_2006no}.

Thus, it remains an interesting task to compare the results of this note with the two foundational papers by Kollár~\cite{kollar97} and Keel--Mori~\cite{KM}, where quotients are constructed as algebraic spaces.

Finally, as mentioned earlier, when proving new theorems in the Hermitian setting, we try to avoid repeating the proofs from the symplectic or Kähler cases line by line. 
Nonetheless, in many instances, some repetition is inevitable. When this occurs, to keep the note concise, we usually present a brief proof here and provide only an outline along with precise references for longer arguments.
In a sense, our main contribution lies in initiating a broader framework for the moment map and in identifying and formulating analogous results within this new context.

\medskip
\noindent {\sl Acknowledgments.}
We are very grateful to the anonymous referee for reading the paper carefully and the inspiring comments.
While this paper was being prepared, YH was visiting Great Bay University, whose hospitality and support are gratefully acknowledged.
XW was partially supported by NSFC grant 12471049 and 12101361, the Project of Young Scholars of Shandong University.

\bigskip
\section{Momentumly closed forms}

Let $K$ be a connected compact Lie group acting smoothly on a differentiable manifold $M$.
Throughout the note, when $M$ is also a complex manifold, we will always assume that the  $K$-action is holomorphic. Furthermore,
 denoting the complexification of $K$ by $G$, we assume that the $K$-action on $M$ extends to a holomorphic $G$-action.  This extension always holds if $M$ is compact.

Given any $\xi \in {\frak k}= \Lie K$, it generates a vector field $\xi_M \in \T M$ defined by
\begin{equation*}
  \xi_{M,m} = \odv{}{t} (\exp (t \xi) \cdot m)|_{t=0}.
\end{equation*}
Let  $\Omega(M)$ denote the algebra of smooth exterior
differential forms on $M$.

\begin{defn}
  \label{def:mc}
A differential form $\omega \in \Omega(M)$ is said to be
momentumly closed if
\begin{equation*}
  \diff \iota_{\xi_M} \omega = 0\;\; \text{for all} \;\; \xi \in {\frak k},
\end{equation*}
where $\diff$ is the exterior differential and $\iota_{\xi_M}$ is the
contraction along the direction $\xi_M$.
\end{defn}

\begin{rem}
  By averaging over $K$, we can always assume $\omega$ to be $K$-invariant.
  As a result, although Definition~\ref{def:mc} works for any differential form $\omega$, in this paper, we will always assume that a momentumly closed form is $K$-invariant.
\end{rem}

Let $L_X = \diff \iota_X + \iota_X \diff$ be the Lie derivative along the vector field $X$ and {$\omega$ be a momentumly closed form}.
Since $\omega$ is $K$-invariant, $L_{\xi_M} \omega =0$, for $\xi\in \kk$.
Therefore $\diff \iota_{\xi_M} \omega = 0$ is equivalent to $ \iota_{\xi_M} \diff \omega = 0$.
In other words, the form $\diff \omega $ is horizontal.
Since $\diff \omega $ is also {$K$-invariant}, it is {\it basic}.
In this term, an invariant form $\omega$ is momentumly closed if and only if $\diff \omega $ is a basic form.

Our momentumly closed forms are related to Henri Cartan's model of the equivariant cohomology as follows \footnote{We thank Hans Duistermaat for pointing out and providing the following.}.

Consider the algebra $A$ of polynomial mappings $\alpha:{\kk}\to \Omega(M)$.
An element $\alpha$ of $A$ is called equivariant if for any $X\in \kk$ and $k\in K$,
\begin{equation*}
  \alpha(X) = k \cdot (\alpha(\Ad{k^{-1}}{X})).
\end{equation*}
Denote the subalgebra of $A$ consisting of equivariant elements by $A_K$.
Define $D:A\to A$ by
\begin{equation*}
  (D\alpha)(X) := \diff(\alpha(X)) - \iota_{X_M}(\alpha(X)), X \in {\kk}, \alpha\in A.
\end{equation*}
On $A_K$, one verifies that $D \circ D = 0$.
Therefore, $(A_K,D)$ is a chain complex and the equivariant cohomology ring $H^*_K(M)$ is defined to be $\ker D/ {\rm Im} D$ using this complex.
The elements of $\ker D$ and ${\rm Im} D$ of $D$ are called the equivariantly closed and exact forms, respectively.
If $\alpha$ is homogeneous of degree $p$ as a differential form and of degree $q$ as a polynomial on ${\frak k}$, then the (total) degree of $\omega$ is defined as $p + 2q$.
With this definition, $D$ raises the degree by one.

Now, if $\omega$ is a $K$-invariant form on $M$ (which does not depend on $X\in {\frak k}$), then as an element in $A_K$, $\omega$ is equivariantly closed if and only if $0 = \diff \omega - \iota_{X_M}\omega$ for all $X\in{\frak k}$, which is equivalent to the condition that $\diff \omega = 0$ and $\iota_{X_M}\omega = 0$ for all $X\in{\frak k}$.
That is, $\omega$ is closed in the ordinary sense and basic.

Therefore, our condition for momentumly closed forms, $\diff( \iota_{X_M}\omega) = 0$ for all $X\in{\frak k}$, is obviously much weaker than the condition for equivariantly closed forms, which is also justified by the following result.

\begin{prop}[Hans Duistermaat]
  Assume that $\omega$ is a $K$-invariant form.
  Define $\nu_{\omega}\in A_K$ by $\nu_{\omega}(X) = \iota_{X_M}\omega$, $X\in \kk$.
  Then $\omega$ is momentumly closed if and only $\nu_\omega$ is equivariantly closed.
\end{prop}

\begin{proof}
  By the definition of the operator $D$,
  \begin{equation*}
    (D\nu_\omega)(X) = \diff(\nu_{\omega}(X)) - \iota_{X_m}(\nu_\omega (X)) = \diff(\iota_{X_M}\omega) - \iota_{X_M}(\iota_{X_M}\omega) = \diff(\iota_{X_M}\omega).
  \end{equation*}
\end{proof}

\bigskip
\section{Hamiltonian $K$-action}

We generalize the concept of the moment map to the momentumly closed case as follows.
The same definition also appears in~\cite[Definition~2.1]{DiTerlizzi_2007re}.

\begin{defn}
  \label{def:gmp}
  Let $\omega$ be any momentumly closed two-form on $M$.
  A map
\begin{equation*}
  \Psi: M \rightarrow \kk^*
\end{equation*}
is called a generalized moment map with respect to $\omega$ if
\begin{equation}
  \label{eq:def-mp}
  \diff \Psi^\xi = \iota_{\xi_M} \omega
\end{equation}
for all $\xi \in {\frak k}$, where $\Psi^\xi = \langle\Psi , \xi\rangle$ using the pairing $\langle \cdot, \cdot \rangle$ between ${\kk}^*$ and ${\kk}$.
By averaging over $K$, we may and will always assume that $\Psi$ is $K$-equivariant with respect to the given action on $M$ and the coadjoint action on $\kk^*$.
\end{defn}

Note that in the above definition, the assumption that $\omega$ is momentumly closed is not necessary because (\ref{eq:def-mp}) implies that $\omega$ is momentumly closed automatically.

Since Karshon and Tolman~\cite{Karshon_1993mo}
has proved that the moment maps of pre-symplectic forms enjoy no convexity in general, \emph{we will only consider the non-degenerate momentumly closed two-form in the rest of the note}.

The non-degenerate two-form has the following well-known characterization.
For a proof, see \cite[p.~114,~Proposition 2]{Fre-Uh}.

\begin{prop}
  \label{prop:ac}
  A manifold $M$ carries a non-degenerate smooth two-form $\omega$ if and only if it has an almost complex structure $J$.
  In fact, we can and will choose the two-form $\omega$ so that $\omega$ is $J$-invariant and $g(-,-) = \omega(-,J-)$ is a Riemannian metric on $M$.
  In other words, $\omega$ is the Hermitian form of the almost Hermitian manifold $(M,J,g)$.
\end{prop}

Based on Definition~\ref{def:gmp}, we can generalize the concept of the Hamiltonian action to our settings.

\begin{defn}
  \label{def:hh}
Assume that $M$ carries a non-degenerate momentumly closed two-form $\omega$.
  The $K$-action on $M$ is said to be Hamiltonian if there exists a generalized moment map $\Psi$ for $(M,\omega)$.
\end{defn}

\begin{rem}

If $(M,J,g)$ is an almost Hermitian $K$-manifold, by taking $\omega$ to be the Hermitian form of $(M,J,g)$, we say the $K$-action is Hamiltonian if only if the $K$-action is Hamiltonian with respect to $(M,\omega)$ according to Definition~\ref{def:hh}.
  Due to Proposition~\ref{prop:ac}, when discussing the properties of the generalized moment map in the following sections, we can always assume that $M$ is an almost Hermitian manifold.

\end{rem}

It follows immediately from the definition that when an almost Hermitian manifold $(M,J,g)$ carries a Hamiltonian $K$-action and let $\Psi$ be the generalized moment map, we have
\begin{equation}
  \label{eq:grad-psixi}
  \grad \Psi^\xi = J \xi_M,\quad \xi \in \kk.
\end{equation}

\begin{rem}
Let $(M,J)$ be an almost complex $K$-manifold.
  For each $\xi \in \kk$, we can consider the dynamical system $(M, J \xi_M)$.
In his celebrated paper~\cite{Smale}, Smale proved that every dynamical system on a manifold can be represented by the gradient vector field of a function with respect to some Riemannian metric provided it satisfies a few generic condition (called the Smale conditions).
From this viewpoint, by (\ref{eq:grad-psixi}), we can say that when the $K$-action on an almost Hermitian manifold is Hamiltonian, the generalized moment map $\Psi$ provides such a function for each element in a family of dynamical systems $(M, J \xi_M)$, $\xi\in \kk$, in a simultaneous way.
\end{rem}

As one can see, one reason to introduce momentumly closed symplectic forms and their associated generalized moment maps is that they are quite abundant compared to their symplectic counterparts.

In the following examples, $\pr_1$ (resp.\ $\pr_2$) denotes the projection map of a product manifold onto the first (resp.\ second) component.

\begin{exmp}
  To start off, we consider a symplectic manifold $(M, \omega)$ with a Hamiltonian $K$-action and an {\it arbitrary} almost Hermitian manifold $(N,J_N, \sigma)$.
  Let $K$ act on the first factor of $M \times N$ only.
  Clearly, $\pr_1^*\omega+ \pr_2^*\sigma$ is a momentumly closed form, and it is not closed if and only if $\sigma$ is not.
  Moreover, denoting the moment map on $M$ by $\Phi$, one checks that $\Psi = \Phi \circ \pr_1: M \times N \rightarrow \kk^*$ is a generalized moment map with respect to $\pr_1^*\omega+ \pr_2^*\sigma$.
\end{exmp}

This example is straightforward.
For the examples that are more interesting and involved, we now present a basic model for momentumly closed two-forms and the associated generalized moment maps.
The model in the symplectic case is due to Sternberg and Weinstein.
We adopt the approach from Sjamaar and Lerman \cite{SL}.

Assume that $K$ acts on $M$ {quasi-freely} (i.e.\ all isotropy subgroups are connected).
Then $P= \{ x \in M | K_x = \id \}$ is an open subset of $M$.
We assume that $P\neq \emptyset$.
Since slices for a compact group action always exist, $P$ is a principal $K$-bundle over the base $B= P/K$.
Having this construction in mind, we now turn to the general situation of principal $K$-bundles.
So, let $\pi: P \rightarrow B$ be an arbitrary principal $K$-bundle.
Define a $K$-action on $P \times \kk$ by
\begin{equation*}
  k \cdot (p, a) = (p k^{-1}, \Ads{k}{a}),\quad k\in K, (p,a)\in P\times \kk^*.
\end{equation*}

Take a connection one-form $\theta$ on $P$.
As before, let $\langle-, -\rangle$ be the pairing between ${\frak k}^*$ and ${\frak k}$.
Take $X_1,\cdots, X_n$ to be a basis of $\kk$ and write $\theta = \sum_{i} \theta^i X_i$ with respect to such a basis, where $\theta^i \in \Omega^1(P)$.
We define a real-valued one-form $\langle\pr_2, \theta \rangle$ on $P\times \kk^*$ as follows.
For $(p, a)\in P\times \kk^*$, the value of $\langle \pr_2, \theta \rangle$ at $(p,a)$ is
\begin{equation}
  \label{eq:pr2-theta}
  \langle \pr_2, \theta \rangle(p,a) = \sum_{i=1}^n \pr_1^*(\theta^i(p)) \langle X_i, a \rangle \in \T^*_{(p,a)}(P \times \kk).
\end{equation}
$\langle\pr_2, \theta \rangle$ is invariant with respect to the $K$-action on $P\times \kk^*$.

\begin{prop}
  \label{prop:min-cp}
Let $\sigma$ be any two-form on $B$ and
\begin{equation*}
  \omega = (\pi\circ \pr_1)^* \sigma - \diff \langle \pr_2, \theta \rangle.
\end{equation*}
Then $\omega$ is a momentumly closed two-form on $P \times {\frak k}^*$ with respect to the $K$-action.
Furthermore, if $\sigma$ is non-degenerate, then $\omega$ is non-degenerate near $P \times \{0\}$.
Moreover, the $K$-action on $P\times \kk^*$ is Hamiltonian and the projection
\begin{equation*}
  -\pr_2: P \times {\frak k}^* \rightarrow {\frak k}^*
\end{equation*}
is a generalized moment map for the $K$-action.
\end{prop}

\begin{proof}
When $\sigma$ is symplectic, this is the minimal coupling form of Sternberg~\cite{Sternberg_1977mi} and Weinstein~\cite{Weinstein_1978un}.
For a proof, one can copy the one given by Sjamaar and Lerman in~\cite[Theorem~8.1]{SL}.
\end{proof}

\begin{rem}
In \S \ref{dynastable}, we will see that such a coupling frequently arises in the context
of complex geometric quotients.
\end{rem}

In view of Proposition~\ref{prop:min-cp}, a momentumly closed two-form is partially closed along the group action orbit in some sense, which is also justified by the following result.
\begin{thm}
  \label{thm:ob-close}
Let $\omega$ be any momentumly closed two-form on $M$.
  Then $\diff \omega$ vanishes along any $K$-orbit.
  Moreover, if $(M,J)$ is a complex manifold and $\omega$ is $J$-invariant, then $\diff \omega$ also vanishes along any $G$-orbit.
  
\end{thm}
\begin{proof}
  To show the real case, we only need to show that for any $\xi, \eta, \zeta \in \kk$,
  \begin{equation}
    \label{eq:domega}
    \diff \omega(\xi_M, \eta_M, \zeta_M) = 0.
  \end{equation}
  However, since $\omega$ is a $K$-invariant momentumly closed two-form, by Cartan's formula,
  \begin{equation}
    \label{eq:cf}
    \iota_{\xi_M}\diff \omega = - \diff \iota_{\xi_M} \omega = 0,
  \end{equation}
  from which, (\ref{eq:domega}) follows immediately.

Based on the real case, to show the complex case, for $\xi,\eta,\zeta\in \kk$, we only need to verify that
  \begin{subequations}
    \begin{align}
      \diff \omega(\xi_M, \eta_M, (i\zeta)_M) &= \diff \omega(\xi_M, \eta_M, J\zeta_M)=0,\label{eq:cd1}\\
      \diff \omega(\xi_M, (i\eta)_M, (i\zeta)_M) &= \diff \omega(\xi_M, J\eta_M, J\zeta_M)=0,\label{eq:cd2}\\
      \diff \omega((i\xi)_M, (i\eta)_M, (i\zeta)_M) &= \diff \omega(J\xi_M, J\eta_M, J\zeta_M)=0.\label{eq:cd3}
    \end{align}
  \end{subequations}
  Among them, (\ref{eq:cd1}) and (\ref{eq:cd2}) also follow from (\ref{eq:cf}).

  For the proof of (\ref{eq:cd3}), we note that since the $G$-action is holomorphic, the following equalities always hold.
  \begin{equation}
    \label{eq:jlie}
    [\xi_M, J\zeta_M] = J[\xi_M, \zeta_M],\quad[J\xi_M, J\zeta_M] = -[\xi_M,\zeta_M].
  \end{equation}
  
For any vector field $X$ on $M$, by using the global formula for the exterior derivative and the definition of a momentumly closed form, we have
  \begin{equation*}
    \begin{aligned}
      0 = (\diff \iota_{\eta_M}\omega)(X, \zeta_M) &= X((\iota_{\eta_M}\omega)(\zeta_M)) - \zeta_M((\iota_{\eta_M}\omega)(X)) - \iota_{\eta_M}\omega([X,\zeta_M])\\
                              &= X(\omega(\eta_M, \zeta_M)) - \zeta_M(\omega(\eta_M,X)) - \omega(\eta_M,[X,\zeta_M]).
    \end{aligned}
  \end{equation*}
  Since $\omega$ is $K$-invariant, we also have
  \begin{equation*}
    \zeta_M(\omega(\eta_M,X)) = \omega([\zeta_M,\eta_M],X) + \omega(\eta_M, [\zeta_M, X]).
  \end{equation*}
  Combining the above two equalities, we get
  \begin{equation}
    \label{eq:ld}
    X(\omega(\eta_M, \zeta_M)) = -\omega([\eta_M,\zeta_M],X)
  \end{equation}
  
  Now, by (\ref{eq:jlie}), (\ref{eq:ld}) and the $J$-invariance of $\omega$, we note that
  \begin{equation*}
    \begin{aligned}
      J\xi_M(\omega(J\eta_M, J\zeta_M)) &= J\xi_M(\omega(\eta_M, \zeta_M))\\
&= -\omega([\eta_M, \zeta_M],J\xi_M) = \omega([J\eta_M, J\zeta_M],J\xi_M).
    \end{aligned}
  \end{equation*}
  By cycling the variable $\xi,\eta,\zeta$ in the above equality, (\ref{eq:cd3}) can be deduced by using the global formula for the exterior derivative again.
\end{proof}

\begin{rem}
In view of Theorem~\ref{thm:ob-close}, if the complex $K$-manifold $M$ admits an open and dense $G$-orbit, e.g.\ $M$ is a toric manifold, any momentumly closed two-form $\omega$ on $M$ is closed actually.
\end{rem}

\section{Basics of generalized moment maps}\label{sec:gmm}

In this section, we will assume that $M$ carries a non-degenerate momentumly closed two-form $\omega$ and the $K$-action on $M$ is Hamiltonian.
Denote the generalized moment map for the $K$-action by $\Psi$.

\subsection{Differentials of the generalized moment maps}
To begin with, we discuss some properties of the differential of the generalized moment maps, which is parallel to the corresponding results in the symplectic settings.\begin{lem}
\label{differimage}
Let $m \in M$ and $K_m$ be the isotropic subgroup of $m$ in $K$.
Then the image of $(\diff \Psi)_m : \T_m M \rightarrow {\frak k}^*$ equals $\kk_m^{\circ}$, the annihilator space in ${\frak k}^*$ of the Lie subalgebra ${\frak k}_m =\Lie(K_m)$.
\end{lem}
\begin{proof}
  The image of $(\diff \Psi)_m$ is a linear subspace of ${\frak k}^*$.
  For any $0 \ne \xi \in {\frak k}$, it is contained in the hyperplane $(\mathrm{span}_{\bR} \{\xi\})^\circ$ if and only if
\begin{equation*}
  0 = \langle (\diff \Psi)_m(v),\xi\rangle = (\diff \langle\Psi , \xi\rangle)_m(v) = \omega_m (\xi_{M,m}, v)
\end{equation*}
for all $v \in \T_m M$, where we use the definition of the generalized moment map (\ref{eq:def-mp}) for the last equality.
Since $\omega$ is non-degenerate, the above equation holds if and only if $\xi_{M,m} = 0$.
In other words, $\xi \in {\kk}_m$.
This implies that $\diff \Psi_m (\T_m M ) = {\kk}_m^\circ$.
\end{proof}

As a consequence, $m$ is a regular (resp.\ critical) point of $\Psi$ if and only if $K_m$ is (resp.\ is not) discrete.
In addition, this simple lemma governs the flat nature of the images of invariant submanifolds under the generalized moment map.
The following result is just a restatement of Proposition~3.6 in~\cite{Guillemin_1982aa} in our setting.

\begin{prop}
\label{isotropy}
Let ${\frak h}$ be the Lie algebra of a subgroup $H \subseteq K$ and ${\frak h}^\circ$ its annihilator in ${\frak k}^*$.
Then the map $\Psi$ sends each connected component of $M_H = \{m \in M | K_m =H\}$ into an affine subspace of ${\kk}^*$ of the form $p + {\frak h}^\circ$, $p\in \kk^*$.
Moreover, if $H$ is a normal subgroup of $K$, $\Psi:M_H\rightarrow p + {\frak h}^\circ$ is a submersion.
\end{prop}

\begin{proof}
  The first statement follows from Lemma \ref{differimage}.
  For the second, notice that since $H$ is a normal subgroup of $K$, $M_H$ is $K$-invariant.
  As a result, the restriction of $\Psi$ on $M_H$, denoted by $\Psi_1$, is the moment map of the $K$-action on $M_H$.
  Then applying Lemma \ref{differimage} again, we know that $(\diff \Psi_{1})_m (\T_m M_H ) = {\kh}^\circ$ for any $m\in M_H$, that is $(\diff \Psi_{1})_m$ is surjective.
\end{proof}

\subsection{A variant of the Darboux--Weinstein theorem}
The most decisive local property of the symplectic form is Darboux's theorem, which asserts that all symplectic forms are the same up to a diffeomorphism locally.
Certainly, Darboux's theorem does not hold for the general non-degenerate two-form.
Meanwhile, note that when the group action is trivial, any non-degenerate two-form is momentumly closed.
As a result, Darboux's theorem does not hold for general momentumly closed two-forms either.

Given this fact, it may be a little surprising that the generalized moment map $\Psi$ still has a local normal form.
In other words, Marle--Guillemin--Sternberg's theorem still holds in our settings.
It is based on a variant of the Darboux--Weinstein theorem for the momentumly closed two-form.
\begin{thm}
  \label{thm:darboux}
  Let $X$ be a smooth manifold carrying a $K$-action and $W$ be a $K$-invariant compact submanifold of $X$.
  Suppose that $\omega_0,\omega_1$ are two momentumly closed non-degenerate two-forms on $X$ with respect to the $K$-action.
  For any point $p\in W$, we assume that $\omega_0|_{\T_p X} = \omega_1|_{\T_p X}$.
  Then, there exist two $K$-invariant neighborhoods $U_0, U_1$ of $W$ and a $K$-equivariant diffeomorphism $\varphi:U_0 \rightarrow U_1$ which fixes $W$ and satisfies
  \begin{equation*}
    \iota_{\xi_X}\omega_0 = \iota_{\xi_X}\varphi^*\omega_1,
  \end{equation*}
  where $\xi\in \kk$.
\end{thm}

To prove this result, via the exponential map, we will identify a $K$-invariant neighborhood $U$ of $W$ in $X$ with a neighborhood of zero section in the normal vector bundle $\mathrm{N}W$ of $W$ in $X$.
Note that since $X$ is $K$-invariant, the $K$-action also induces a $K$-action on $\mathrm{N}W$.
With respect to such an induced group action, the identification is also $K$-equivariant.

As a preparation, we need a lemma.
\begin{lem}
  \label{lm:beta}
  There exists a $K$-invariant one-form $\beta\in \Omega^1(U)$ such that for any $\xi\in \kk$,
  \begin{equation}
    \label{eq:beta}
    \iota_{\xi_X}(\omega_1 - \omega_0) = \iota_{\xi_X}\diff \beta\quad\text{and}\quad\beta|_{\T_p X} =0\text{ for any $p\in W$}.
  \end{equation}
\end{lem}
\begin{proof}
  Since we have identified $U$ with a neighborhood of zero section in $\mathrm{N}W$, for $0 \le s \le 1$, we can set
  \begin{equation*}
    \psi_s(x) = sx: U \rightarrow U.
  \end{equation*}
  Using the induced $K$-action on $\mathrm{N}W$, it is straightforward to verify that $\psi_s$ is a $K$-equivariant map of $U$.
  Let $Y_s$ be the vector field
  \begin{equation*}
    Y_{s,\psi_s(x)} = \pdv{}{s} (\psi_s(x)).
  \end{equation*}
  By this definition, one can check that $Y_s$ is $K$-invariant.

  Denote $\omega_1 - \omega_0$ by $\alpha\in \Omega^2(U)$.
  Then $\alpha$ is momentumly closed and $\alpha|_{\T_p X} = 0$ for any $p\in W$.
  Define the one-form on $U$
  \begin{equation*}
    \beta = \int^1_0 \psi_s^* (\iota_{Y_s}\alpha) \diff s.
  \end{equation*}
  By the definition of $\psi_s$ and $Y_s$, one can check that $\beta$ is a smooth, $K$ invariant one-form and $\beta|_{\T_p X} =0$ for any $p\in W$.

Since the image set of $\psi_0$ is contained in $W$ and $\alpha|_{\T_p X} = 0$ for any $p\in W$, we have $\psi^*_0\alpha = 0$.
  Besides, $\psi_1 = \mathrm{id}$.
  Therefore, by Cartan's formula,
  \begin{align*}
    \alpha = \psi_1^*\alpha - \psi^*_0 \alpha &= \int_0^1 \psi_s^* (\mathcal{L}_{Y_s}\alpha) \diff s \\
    &= \diff (\int_0^1 \psi^*_s(\iota_{Y_s} \alpha) \diff s)
    + \int_0^1 \psi^*_s(\iota_{Y_s} \diff \alpha) \diff s \\
    &= \diff \beta + \int_0^1 \psi^*_s(\iota_{Y_s} \diff \alpha) \diff s.
  \end{align*}
  Since $\alpha$ is momentumly closed, by the above equality, for any $\xi\in \kk$, we have
  \begin{multline*}
    \iota_{\xi_X} \alpha = \iota_{\xi_X} \diff \beta + \iota_{\xi_X}(\int_0^1 \psi^*_s(\iota_{Y_s} \diff \alpha) \diff s)\\
    = \iota_{\xi_X} \diff \beta - \int_0^1 \psi^*_s(\iota_{Y_s} (\iota_{\xi_X}\diff \alpha)) \diff s) = \iota_{\xi_X} \diff \beta.
  \end{multline*}
  The proof of Lemma~\ref{lm:beta} is finished.
\end{proof}

With Lemma~\ref{lm:beta}, we use Moser's trick to show Theorem~\ref{thm:darboux}.

\begin{proof}
  Let $\omega_t = (1-t)\omega_0 + t\omega_1$, $0\le t \le 1$.
  Since $\omega_0$ and $\omega_1$ coincide at $W$, by choosing the $K$-invariant neighborhood $U$ of $W$, we can and will assume that $\omega_t$ is non-degenerate on $U$ for $0\le t \le 1$.
  Moreover, $\omega_t$ is also momentumly closed.

  The key of Moser's trick is to find a family of $K$-invariant vector fields $R_t$ on $U$ such that
  \begin{equation}
    \label{eq:xt}
    \iota_{\xi_X}(\mathcal{L}_{R_t}\omega_t + \omega_1 - \omega_0) = 0\quad\text{and}\quad R_{t,p} =0\text{ for any }p\in W
  \end{equation}
  hold for any $\xi\in \kk$.
  Once $R_t$ are found, we can use $R_t$ to construct a family of local $K$-equivariant diffeomorphisms $\varphi_t$ such that
  \begin{equation*}
    \varphi_0 = \mathrm{id},\quad \iota_{\xi_X}(\varphi_t^*\omega_t) = \iota_{\xi_X}(\omega_0),\quad \varphi_t|_W = \mathrm{id}.
  \end{equation*}
  Therefore, by choosing $\varphi = \varphi_1$, the proof of Theorem~\ref{thm:darboux} completes.

  To find $R_t$, we use the one-form $\beta$ given by Lemma~\ref{lm:beta}.
  With such a one-form, due to the non-degeneracy of $\omega_t$, we can choose $R_t$ satisfying
  \begin{equation*}
    \iota_{R_t}\omega_t + \beta =0.
  \end{equation*}

  To verify $R_t$ satisfying (\ref{eq:xt}), we first note that $R_t$ is $K$-invariant because both $\omega_t$ and $\beta$ are $K$-invariant.
  And the fact that $\beta|_{\T_p X} =0$ for any $p\in W$ implies that $R_t$ vanishes on $W$.
  Meanwhile, by (\ref{eq:beta}) and Cartan's formula, for any $\xi\in \kk$, we have
  \begin{multline*}
    \iota_{\xi_X}(\mathcal{L}_{R_t}\omega_t + \omega_1 - \omega_0) = \iota_{\xi_X}(\iota_{R_t}\diff \omega_t + \diff \iota_{R_t} \omega_t) + \iota_{\xi_X}(\diff \beta)\\
    = - \iota_{R_t}(\iota_{\xi_X}\diff \omega_t) + \iota_{\xi_X}\diff(\iota_{R_t}\omega_t + \beta) = 0,
  \end{multline*}
  where for the last equality, we use that $\omega_t$ is momentumly closed.
  As a result, (\ref{eq:xt}) is proved for $R_t$.
\end{proof}

Let $i:W \rightarrow M$ be an isotropic embedding submanifold, which means that $i$ is an embedding and at each $x\in W$, $(\diff i)_x(\T_x W)$ is an isotropic subspace of $\T_{i(x)}M$, that is,
\begin{equation*}
  (\diff i)_x(\T_x W) \subseteq ((\diff i)_x(\T_x W))^{\omega} \subseteq \T_{i(x)}M,
\end{equation*}
where $(-)^\omega$ is the symplectic orthogonal complement.
Following the convention in symplectic geometry, we define a bundle $E$ over $W$,
\begin{equation*}
  E_x = ((\diff i)_x(\T_x W))^{\omega} / (\diff i)_x(\T_x W),
\end{equation*}
and we call $E$ the symplectic normal bundle of $W$.
Note that the fiber of $E$ has a symplectic form induced from $\omega$.
Then, due to Theorem~\ref{thm:darboux}, one can also prove a version of the isotropic embedding theorem as follows.

\begin{thm}
  \label{thm:iso-emb}
  Let $(M_1,\omega_1)$ and $(M_2,\omega_2)$ be two smooth $K$-manifolds with momentumly closed non-degenerate two-forms.
  Suppose that $W_1$ and $W_2$ are compact $K$-invariant isotropic embedding submanifolds of $M_1$ and $M_2$ respectively.
  Let $E_1$ and $E_2$ be the symplectic normal bundles of $W_1$ and $W_2$ respectively.
  If there is a $K$-equivariant bundle isomorphism $L:E_1\rightarrow E_2$ which preserves the fiberwise symplectic form, then there exist $K$-invariant neighborhoods $U_0, U_1$ of $W_1$ and $W_2$ respectively and a $K$-equivariant diffeomorphism $\varphi:U_0 \rightarrow U_1$ which maps $W_1$ onto $W_2$ and satisfies
  \begin{equation*}
    \iota_{\xi_X}\omega_0 = \iota_{\xi_X}\varphi^*\omega_1,
  \end{equation*}
  where $\xi\in \kk$.
  Moreover, the induced bundle map of $\varphi$ from $E_1$ into $E_2$ coincides with $L$.
\end{thm}

\begin{proof}
  In this proof, the subscript $i$ takes value $1$ or $2$.

  By choosing a Riemannian metric on $M_i$ compatible with $\omega_i$, we have an orthogonal decomposition of $\T M_i|_{W_i}$ with respect to this metric,
  \begin{equation}
    \label{eq:tm-decomp}
    \T M_i|_{W_i} = \hat{E}_i \oplus F_i \oplus \T W_i.
  \end{equation}
  About this decomposition, the following facts hold.
  \begin{enumerate}[label=(\roman*)]
  \item $\hat{E}_i$ is orthogonal $F_i \oplus \T W_i$ with respect to $\omega_i$.
  \item For any $x\in W_i$, $\hat{E}_{i,x}$ is a symplectic subspace of $\T_x M_{i}$.
    And $\hat{E}_i$ is isomorphic to $E_i$ as symplectic bundles.
  \item For any $x\in W_i$, $F_{i,x}$ is also an isotropic subspace of $\T_x M_i$.
    As a result, $F_i$ is isomorphic to $\T^* W_i$ via $\omega_i$.
  \end{enumerate}
  Therefore, up to isomorphisms, we have
  \begin{equation*}
    \T M_i|_{W_i} \simeq E_i \oplus \T^* W_i \oplus \T W_i.
  \end{equation*}

  Note that when restricting to the zero section, the bundle isomorphism $L$ induces a diffeomorphism from $W_1$ to $W_2$, which, in turn, gives an isomorphism between $\T W_1$ and $\T W_2$, as well as an isomorphism between $\T^* W_1$ and $\T^* W_2$.
  In summary, we have a bundle isomorphism $\rho:\hat{E}_1 \oplus F_1\rightarrow\hat{E}_2 \oplus F_2$.

  On the other hand, $\hat{E}_i \oplus F_i$ is the normal bundle of $W_i$.
  By using the exponential map and $\rho$, we can find $K$-invariant neighborhoods $\widetilde{U}_0, \widetilde{U}_1$ of $W_1$ and $W_2$ respectively and a $K$-equivariant diffeomorphism $\widetilde{\varphi}:\widetilde{U}_0 \rightarrow \widetilde{U}_1$ which maps $W_1$ onto $W_2$.
  Moreover, by the construction of $\rho$, one can check that
  \begin{equation*}
    \omega_1|_{\T_x W_1} = \widetilde{\varphi}^*(\omega_2)|_{\T_x W_1},\quad\text{for each }x\in W_1.
  \end{equation*}
  Based on $\widetilde{\varphi}$, we can apply Theorem~\ref{thm:darboux} to find the desired $\varphi$.
\end{proof}

\subsection{A local normal form for the generalized moment map}
As in the symplectic case, Theorem~\ref{thm:iso-emb} enables us to show the existence of a local normal form for the generalized moment map, that is, Marle--Guillemin--Sternberg's theorem,~\cite{Guillemin_1984no},~\cite{Marle_1984le}.

We follow Lerman's statement of this theorem~\cite[Theorem~2.1]{Lerman_2005aa}.
Recall that we have assumed that $\Psi$ is the generalized moment map for the Hamiltonian $K$-action on $(M,\omega)$.
For any $m\in M$, as before, we denote the isotropy subgroup (resp.\ subalgebra) at $m$ by $K_m$ (resp.\ $\kk_m$).
Let $\alpha = \Psi(m)$.
Similarly, denote the isotropy subgroup (resp.\ subalgebra) at $\alpha$ by $K_{\alpha}$ (resp.\ $\kk_{\alpha}$).
Since $\Psi$ is equivariant, $\kk_{m} \subseteq \kk_{\alpha}$.
With a $K$-invariant metric on $\kk$, we can choose a $K_m$-equivariant splitting
\begin{equation}
  \label{eq:kk-split}
  \kk^* = \kk_m^* \times (\kk_{\alpha}/ \kk_m)^* \times (\kk/ \kk_{\alpha})^*,
\end{equation}
which gives the embedding $\kk_m^* \hookrightarrow \kk^*$ and $(\kk_{\alpha}/ \kk_m)^* \hookrightarrow \kk^*$.
Moreover, let
\begin{equation*}
  E_m = (\kk \cdot m)^{\omega}/ ((\kk \cdot m)^{\omega} \cap \kk \cdot m).
\end{equation*}
The symplectic form at $\T_m M$ induces a symplectic linear space structure on $E_m$ and the induced linear $K_m$-action on $E_m$ preserves the symplectic structure on $E_m$.

\begin{thm}
  \label{thm:nf}
  With respect to the linear $K_m$-action on $E_m$, let $\mu:E_m \rightarrow \kk^*_m$ be the associated quadratic homogeneous moment map.
  There exists a $K$-invariant neighborhood $U$ of the orbit $K \cdot m \subseteq M$, a $K$-invariant neighborhood $U_0$ of the zero section of the vector bundle
  \begin{equation*}
    K\times_{K_m} ((\kk_{\alpha}/ \kk_{m})^* \times E_m) \rightarrow K/K_m
  \end{equation*}
  and a $K$-equivariant diffeomorphism $\varphi: U_0 \rightarrow U$ such that
  \begin{equation}
    \label{eq:two-mp}
    \Psi \circ \varphi([k, p, v]) = \Ads{k}{\alpha + p + \mu(v)}
  \end{equation}
  for all $[k,p,v]\in U_0$.
  Here, $[k,p,v]$ denotes the orbit of $(k,p,v)\in K \times ((\kk_{\alpha}/ \kk_{m})^* \times E_m)$ in the associated bundle $K \times_{K_m} ((\kk_{\alpha}/ \kk_{m})^* \times E_m)$ and $\Ads{k}{\alpha + p + \mu(v)}$ is just the moment map for the symplectic structure on this associated bundle.
\end{thm}

Although the statement of Theorem~\ref{thm:nf} resembles the classical Marle--Guillemin--Sternberg's theorem closely, there is a crucial difference.
Namely, we start from a generalized moment map but obtain a non-generalized one (i.e.\ a moment map associated with a symplectic form) as a local model in the end.
Nevertheless, in view of the minimal coupling construction given in Proposition~\ref{prop:min-cp}, such a phenomenon is forseeable because the generalized moment map $\Psi$ is independent on the two-form on the base $B$ there.

\begin{proof}
  The method to prove the symplectic counterpart of Theorem~\ref{thm:nf} can also be applied here.
  Certainly, at some point, we need to replace the symplectic isotropic embedding theorem with Theorem~\ref{thm:iso-emb}.
  More precisely, the proof consists of two steps.
    The first step is to construct a suitable symplectic form on the vector bundle $K\times_{K_m} ((\kk_{\alpha}/ \kk_{m})^* \times E_m)$.
    The second step is to use Theorem~\ref{thm:iso-emb} to construct the diffeomorphism $\varphi$.
    For an important special case: $\kk_{\alpha} = \kk$, the whole proof is particularly clear.
    We will work it out in details and comment on how to adapt the proof for the general case.

  In the case of $\kk_{\alpha} = \kk$, by the equivariance of $\Psi$, one knows that the orbit $K \cdot m$ is an isotropic embedding submanifold of $M$.
  Since $K\cdot m$ is a homogeneous manifold, the symplectic normal of $K \cdot m$ is also homogeneous, whose fiber at $m$ is $E_m$ exactly.
  On the other hand, by identify $K \times \kk^*$ with $\T^* K$ via the left translation, $K \times \kk^* \times E_m$ is a Hamiltonian $K \times K_m$-manifold.
  More specifically, for $(k,l)\in K \times K_m$, $(g, q, v)\in K \times \kk^* \times E_m$, the $K\times K_m$-action is defined by
  \begin{equation*}
    (k, l) \cdot (g, q, v) = (kgl^{-1}, \Ads{l}{q}, l\cdot v).
  \end{equation*}
  The moment map associated with the $K \times K_m$-action is
  \begin{equation*}
    \Phi(g,q,v) = (\Phi_K, \Phi_{K_m}) = (\Ads{g}{q}, -q|_{\kk_m} + \mu(v)) \in \kk^* \times \kk_m^*,
  \end{equation*}
  where $q|_{\kk_m}\in \kk^*_m$ is the restriction on $\kk_m$.

  Note that $K\times ((\kk/ \kk_{m})^* \times E_m)$ is isomorphic to $\Phi^{-1}_{K_m}(0)$ via the map
  \begin{equation*}
    \begin{aligned}
      K\times ((\kk/ \kk_{m})^* \times E_m) & \rightarrow && \Phi^{-1}_{K_m}(0) \\
      (k, p, v) & \mapsto && (k, p + \mu(v), v),
    \end{aligned}
  \end{equation*}
  where we have used the splitting (\ref{eq:kk-split}).
  As a result, the symplectic reduction gives a symplectic structure on $K\times_{K_m} ((\kk/ \kk_{m})^* \times E_m) \simeq \Phi^{-1}_{K_m}(0)/K_m$ and the $K$-moment map for such a symplectic structure is
  \begin{equation}
    \label{eq:phik}
    \overline{\Phi}_{K}([k,p,v]) = \Ads{k}{p+ \mu(v)}.
  \end{equation}
  Moreover, at $[e,0,0]$, the fiber of symplectic normal bundle of $K\cdot[e,0,0]$ is $E_m$.
  Then, by the homogeneity of $K\cdot m$ and $K\cdot[e,0,0]$, we can apply Theorem~\ref{thm:iso-emb} to these two isotropic embedding submanifolds.
  Therefore, we can find a $K$-invariant neighborhood $U$ of $K\cdot m$ and a $K$-invariant neighborhood $U_0$ of $K\cdot [e,0,0]$, as well as a $K$-equivariant diffeomorphism $\varphi: U_0 \rightarrow U$, such that $\Psi \circ \varphi$ and $\overline{\Phi}_K$ coincide up to a constant.
  Due to
  \begin{equation*}
    \Psi\circ \varphi([e,0,0]) = \Psi(m) = \alpha,
  \end{equation*}
  the equality (\ref{eq:two-mp}) follows from (\ref{eq:phik}).

  For the general case without the assumption $\kk_{\alpha} = \kk$, both the steps in the proof need some modifications.
  We first dicuss the second step.
  Unlike the $\kk_{\alpha} = \kk$ case, in general, $K \cdot m$ is no longer an isotropic embedding submanifold.
  In fact, such a problem also appears in the proof of the symplectic version of Theorem~\ref{thm:nf}.
  As in the symplectic case, we can resolve this problem by substituting Theorem~\ref{thm:iso-emb} with another Darboux type theorem, i.e.\ the $G$-relative Darboux theorem given in~\cite[Theorem~7.3.1]{Ortega_2004aa}.
  Certainly, one should verify that the $G$-relative Darboux theorem also holds for the non-degenerate momentumly closed two-forms.
  But it is an almost routine task and we leave it to readers.

  Back to the first step.
  If we can construct a symplectic form on $K \times (\kk_{\alpha}^* \times E_m)$, a similar symplectic reduction construction on $K \times (\kk_{\alpha}^* \times E_m)$, as we have used before, leads to the normal form on $K\times_{K_m} ((\kk_{\alpha}/ \kk_{m})^* \times E_m)$ (near the zero section).
  However, comparing to the $\kk_{\alpha} = \kk$ case, it is a little more complex to construct the symplectic form on $K \times (\kk_{\alpha}^* \times E_m)$, because the construction involves the two-form $\omega$ on $M$ and as a result, why the two-form obtained on $K \times (\kk_{\alpha}^* \times E_m)$ is symplectic is not obvious a priori.
  We will check it directly in the below.

  Clearly, to define a symplectic form on $K \times (\kk_{\alpha}^* \times E_m)$ near $K \times \{0\} \times E_m$, we only need to show how to define the symplectic form on $K \times \kk_{\alpha}^*$ near $K \times \{0\}$.
Fix a $K$-invariant inner product on $\kk$.
  Then, we have an embedding $\kk_{\alpha}^* \hookrightarrow \kk^*$, which induces an embedding $i: K \times \kk_{\alpha}^* \hookrightarrow K \times \kk^*$.
  Via the pull-back by $i$, the standard symplectic form on $K \times \kk^*$ gives a closed two-form $\omega_1$ on $K\times \kk_{\alpha}^*$.
  On the other hand, one can define another two-form $\omega_2$ on $K\times \kk_{\alpha}^*$ as follows.
  At $(k,q)\in K\times \kk_{\alpha}^*$,
  \begin{equation}
    \label{eq:omega2}
    \omega_{2,(k, q)}((k\cdot \xi, v), (k\cdot \eta, w)) = \omega_m (\xi_{M,m}, \eta_{M,m}) = -\langle \Psi(m),[\xi,\eta] \rangle,
  \end{equation}
  where $\xi,\eta\in \kk$, $v,w\in \kk_{\alpha}^*$ and $k\cdot \xi, k\cdot \eta \in \T_{k}K$ via the left translation on $K$.
  The two-form $\Omega$ on $K\times \kk_{\alpha}^*$ is defined to be
  \begin{equation*}
    \Omega = \omega_1 + \omega_2.
  \end{equation*}
  Clearly, to show $\Omega$ is closed, we only need to check the closedness of $\omega_2$.

  Note that one can calculate the Lie bracket on the smooth vector fields on $K\times \kk_{\alpha}^*$ in following way,
  \begin{equation}
    \label{eq:klb}
    [(k\cdot \xi, v), (k\cdot \eta, w)] = -[k\cdot[\xi,\eta], 0].
  \end{equation}
  By the global formula for the exterior differential, for $\zeta,\xi,\eta \in \kk$ and $u,v,w\in \kk_{\alpha}^*$,
  \begin{multline}
    \label{eq:domega2}
    \begin{aligned}
      &(\diff \omega_2)((k\cdot \zeta, u), (k\cdot \xi, v), (k\cdot \eta, w)) \\
      = &(k\cdot \zeta, u) \omega_2((k\cdot \xi, v), (k\cdot \eta, w)) - (k\cdot \xi, v) \omega_2((k\cdot \zeta, u), (k\cdot \eta, w))\\
      &
        \begin{multlined}
          +  (k\cdot \eta, w) \omega_2((k\cdot \xi, v), (k\cdot \zeta, u)) 
          - \omega_2([(k\cdot \zeta, u), (k\cdot \xi, v)], (k\cdot \eta, w))
        \end{multlined}\\
      &
        \begin{multlined}
          + \omega_2([(k\cdot \zeta, u), (k\cdot \eta, w)], (k\cdot \xi, v)) 
          - \omega_2([(k\cdot \xi, v), (k\cdot \eta, w)], (k\cdot \zeta, u)).
        \end{multlined}
    \end{aligned}
  \end{multline}
  By the definition of $\omega_2$,
  \begin{equation*}
    (k\cdot \zeta, u) \omega_2((k\cdot \xi, v), (k\cdot \eta, w)) = - (k\cdot \zeta, u) \langle \Psi(m), [\xi,\eta]\rangle = 0,
  \end{equation*}
  together with (\ref{eq:klb}),
  \begin{multline*}
    \omega_2([(k\cdot \zeta, u), (k\cdot \xi, v)], (k\cdot \eta, w)) \\
    = -\omega_2((k\cdot[\zeta,\xi], 0), (k\cdot \eta, w)) = \langle \Psi(m), [[\zeta,\xi], \eta]\rangle.
  \end{multline*}
  Plugging the above two equalities into (\ref{eq:domega2}), we have
  \begin{multline*}
    (\diff \omega_2)((k\cdot \zeta, u), (k\cdot \xi, v), (k\cdot \eta, w)) = \\
    \langle \Psi(m), -[[\zeta,\xi], \eta] + [[\zeta,\eta], \xi] - [[\xi, \eta], \zeta] \rangle = 0.
  \end{multline*}
  As a result, $\Omega$ is closed.

  Next, we show that $\Omega$ is non-degenerate on $K\times \{0\}$.
  Let $\kk = \kk_{\alpha} \oplus \kq$ be an orthogonal decomposition.
  Choose $(\xi, v)\in \kk \times \kk^*_{\alpha}$.
  Suppose that
  \begin{equation}
    \label{eq:oxe}
    \Omega_{(k,0)}((k\cdot \xi, v), (k\cdot \eta, w)) = 0
  \end{equation}
  holds for any $(\eta, w) \in \kk \times \kk^*_{\alpha}$.
  Let $\xi = \xi_1 + \xi_2$ and $\eta = \eta_1 + \eta_2$, where $\xi_1, \eta_1 \in \kk_{\alpha}$ and $\xi_2, \eta_2 \in \kq$.
  Then, by (\ref{eq:oxe}) and the definition of $\Omega$, $\omega_1$ and $\omega_2$, we have
  \begin{equation*}
    0 = \langle w, \xi \rangle - \langle v, \eta \rangle - \langle\Psi(m), [\xi, \eta]\rangle = \langle w, \xi_1 \rangle - \langle v, \eta_1 \rangle - \langle\Psi(m), [\xi_2, \eta_2]\rangle.
  \end{equation*}
  By taking $\eta = 0$, the above equality implies that $\xi_1 = 0$.
  By taking $\eta_2 = 0$ and $w = 0$, the above equality implies $v = 0$.
  Now, we have
  \begin{equation*}
0 = - \langle\Psi(m), [\xi_2, \eta]\rangle = -\langle\alpha, [\xi_2, \eta]\rangle. \end{equation*}
  holds for any $\eta \in \kk$, which implies that $\xi_2 \in \kk_{\alpha}$.
  But by definition $\xi_2 \in \kq$ and $\kk_{\alpha} \perp \kq$.
  Thus, $\xi_2 = 0$.

  Let $U$ be a sufficient small $K_m$-invariant neighborhood of the origin in $\kk_{\alpha}^*$.
  Since $\Omega$ is non-degenerate on $K \times \{0\}$, $\Omega$ is also non-degenerate on $K \times U$.
  As we have said, similar to the $\kk_{\alpha} = \kk$ case, now we can construct a symplectic form near the zero section of $K\times_{K_m} ((\kk_{\alpha}/ \kk_{m})^* \times E_m)$ by invoking the symplectic reduction on $K \times U \times E_m$.

\end{proof}

\begin{rem}
  Recently, in~\cite{Diez_2024sy}, the authors found another approach to prove the existence of a different local normal form for the moment map.
  As their method depends more on the moment map itself rather the symplectic form, probably, their method may also work in our settings.
\end{rem}

\section{Convexity for image of generalized moment maps}

In this section, we discuss the generalization of two classical convexity results for the image of moment maps: the Atiyah--Guillemin--Sternberg--Kirwan convexity theorem and Atiyah's convexity theorem for the orbit-closure.

Throughout this section, we always assume that the compact group $K$ is connected.
Let $V$ be a finite dimensional real vector space.
Several kinds of convex subsets of $V$ will be used in this section.
A closed affine halfspace is a subset of $V$ defined by an inequality $\lambda(v) \ge c$ with $\lambda\in V^*$, $c \in \mathbb{R}$.
A convex polyhedral set is the intersection of a locally finite collection of closed affine halfspaces in $V$.
A convex polyhedron is the intersection of finitely many closed affine halfspaces.
A convex polytope is a bounded convex polyhedron.

\subsection{Convexity for the moment body}

Choose $T$ to be a maximal torus of $K$.
Let ${\kt}^*_+\subseteq \kt^*$ be a fixed positive closed Weyl chamber in ${\kk}^*$.
Recall that each coadjoint orbit intersects the chamber $\kt^*_+$ in exactly one point and the composition $\kt^*_+ \hookrightarrow \kk^* \rightarrow \kk^*/\Ads{}{K}$ induces a homeomorphism,~\cite[p.~294,~Corollary]{Bourbaki_2005li}.
By identifying $\kt^*_+$ with $\kk^*/\Ads{}{K}$, we denote the quotient map from $\kk^*$ to $\kt^*_+$ by $q$.

\begin{thm}[Local Convexity]
  \label{thm:lconv}
  Let $M$ be a manifold carrying a non-degenerate momentumly closed two-form $\omega$ and assume that the $K$-action on $M$ is Hamiltonian.
  Denote the generalized moment map for the $K$-action by $\Psi$.
For any point $m\in M$, there exists a closed convex cone $C_m$ in $\kt^*$ with apex at $q(\Psi(m))$ and a basis of $K$-invariant neighborhoods\footnote{It means that for any open subset $A$ containing $K\cdot m$, there exists an element $U_0$ in this basis such that $U_0\subseteq A$.}
  $U$ of $m$ such that
  \begin{enumerate}[label=\normalfont(\arabic*)]
  \item the fibers of $q \circ \Psi|_U$ are connected;
  \item $q \circ \Psi: U \rightarrow C_m$ is an open map;
  \item for any point $y\in K\cdot m$, $C_y = C_m$.
  \end{enumerate}
\end{thm}

\begin{proof}
  We use Theorem~\ref{thm:nf} to reduce Theorem~\ref{thm:lconv} to corresponding symplectic case.

  By Theorem~\ref{thm:nf}, we can find a $K$-invariant neighborhood $U$ of the orbit $K \cdot m \subseteq M$, a $K$-invariant neighborhood $U_0$ of a $K$-orbit in $K\times_{K_m} ((\kk_{\alpha}/ \kk_{m})^* \times E_m)$ and a $K$-equivariant diffeomorphism $\varphi: U_0 \rightarrow U$ so that
  \begin{equation*}
    \Psi \circ \varphi = \Phi,
  \end{equation*}
  where $\Phi$ is the moment map on $K\times_{K_m} ((\kk_{\alpha}/ \kk_{m})^* \times E_m)$.
  Therefore, to show Theorem~\ref{thm:lconv}, we only need to show that (1) and (2) hold for $\Phi$.
  Compared to $M$, $K\times_{K_m} ((\kk_{\alpha}/ \kk_{m})^* \times E_m)$ is a symplectic manifold.
  Now, by~\cite[Theorem~5.1]{Knop_2002co}, (1) holds for $\Phi$ and by~\cite[Theorem~6.5]{Sjamaar_1998aa}, (2) also holds for $\Phi$.
  The proof of Theorem~\ref{thm:lconv} completes.
\end{proof}

To enhance the local convexity of the generalized moment map to a global property, we need to investigate the topological information of the generalized moment map.
A traditional way is to use the Morse theory.
In~\cite{Condevaux_1988ge}, the authors found another method via the ``Local-to-Global Principle'', which can be applied to many convexity problems.
Here, we use a version of the principle following~\cite{Hilgert_1994aa}.

\begin{defn}
  \label{def:loc-conv-data}
  Let $X$ be a connected Hausdorff topological space and $V$ be a finite dimensional real vector space.
  A continuous map $f:X \rightarrow V$ is said to be locally fiber connected, if every point $x$ in $X$ admits a basis of neighborhoods $U_x$ of $x$ such that $f^{-1}(f(u))\cap U_x$ is connected for all $u \in U_x$.
  We say that a map $x \rightarrow C_x$ assigning to each point $x \in X$ a closed convex cone $C_x \subseteq V$ with apex $f(x)$ is a system of local convexity data if for each $x\in X$ there exists a basis of neighborhood $U_x$ of $x$ such that the following conditions hold:
  \begin{enumerate}
  \item the fibers of $f|_{U_x}$ are connected, that is, $f^{-1}(f(u)) \cap U_x$ is connected for all $u \in U_x$;
  \item $f|_{U_x}: U_x \rightarrow C_x$ is an open map.
  \end{enumerate}
\end{defn}

\begin{thm}[Local-to-Global Principle, {\cite[Theorem~3.10]{Hilgert_1994aa}}]
  \label{thm:l2g}
  Suppose that $f:X \rightarrow V$ is a proper, locally fiber connected map with the local convexity data $(C_x)_{x\in X}$.
  Then the fibers of $f$ are connected and $f:X \rightarrow f(X)$ is an open map.
  Moreover, $f(X)$ is a closed convex polyhedral subset of $V$.
\end{thm}

We now formulate a convexity theorem, which is due to Atiyah~\cite{Atiyah}, Guillemin and Sternberg~\cite{Guillemin_1982aa} and Kirwan~\cite{Kirwan84b} in the symplectic case.

\begin{thm}[Convexity]
  \label{thm:convex}
Take the same assumption for $M$, $\omega$ and $\Psi$ as in Theorem~\ref{thm:lconv}.
  Besides, suppose that $M$ is compact.
  Then the following properties of $\Psi$ hold.
  \begin{enumerate}[label=\normalfont(\arabic*)]
  \item The fibers of $\Psi$ are connected and $q\circ \Psi:M \rightarrow q\circ \Psi(M)$ is an open map.
  \item The moment body of $M$ defined by
    \begin{equation*}
      \Delta(M) = q \circ \Psi(M) = \Psi(M) \cap \kt^*_+,
    \end{equation*}
    is a closed convex polytope in $\kt^*$.
  \end{enumerate}
\end{thm}

\begin{rem}
  In~\cite[p.~626]{Birtea_2009op}, if $q\circ \Psi:M \rightarrow q\circ \Psi(M)$ is an open map, the authors call that $\Psi$ is $K$-open.
  As pointed out in~\cite{Birtea_2009op}, $\Psi$ itself is not open in general.
\end{rem}

\begin{proof}
  We will check that the local convexity of $q\circ \Psi$, namely Theorem~\ref{thm:lconv}, provides a system of local convexity data for $q\circ \Psi$ essentially.

  For any $m\in M$, by Theorem~\ref{thm:lconv}, we can find the $K$-invariant neighborhood $U$ and the closed convex cone $C_m$ satisfying the requirement in Definition~\ref{def:loc-conv-data}.
  However, a subtle point is that the collection of all such neighborhoods $U$ is only a neighborhood basis of $K\cdot m$ rather than a neighborhood basis of $m$ as needed in Theorem~\ref{thm:l2g}.\footnote{In fact, if we check the proof of~\cite[Thereom~3.10]{Hilgert_1994aa}, it turns out that the neighborhoods given by Theorem~\ref{thm:lconv} are also sufficient for the proof.
    Considering this fact, the argument given below may be not so necessary.}
  We can bypass this question as follows.
  
  Let $X = M/K$ be the quotient space of $K$-action.
  Since $K$ is a compact group, $X$ is a Hausdorff space and the projection map $\pi: M \rightarrow X$ is an open map.
  At the same time, because the $q \circ \Psi$ is $K$-invariant, $q \circ \Psi$ induces a continuous map $\Phi: X \rightarrow \kt^*$ such that the following diagram commutes.
  \begin{equation}
    \label{eq:def-xf}
    \begin{tikzcd}
      M \arrow[r, "\Psi"] \arrow[d, "\pi"'] &  \kk^* \arrow[d, "q"]\\
      X \arrow[r, "\Phi"] & \kt^*
    \end{tikzcd}.
  \end{equation}
  For any $m\in M$, choose $U$ and $C_m$ as given by Theorem~\ref{thm:lconv}.
  Since $U$ is a $K$-invariant neighborhood of $m$ and $\pi$ is open, $\pi(U)$ is a neighborhood of $\pi(m)$.
  Moreover, the openness $q\circ \Psi|_U:U \rightarrow C_m$ is implies that $\Phi|_{\pi(U)} \rightarrow C_m$ is also open.
  For any $u\in U$, due to (\ref{eq:def-xf}), we have
  \begin{equation*}
    \Phi^{-1}(\Phi(\pi(u))) \cap \pi(U) = \pi\Big((q\circ \Psi)^{-1}\big(q\circ\Psi(u)\big) \cap U\Big).
  \end{equation*}
  Since the fibers of $q\circ \Psi|_U$ are connected, the above equality implies that the fibers of $\Phi|_{\pi(U)}$ are also connected.
  At last, the collection of all $\pi(U)$ is a basis of neighborhood of $\pi(m)$ because all such $U$ forms a basis of $K$-invariant neighborhood of $m$.
  In summary, we have checked that the system $\pi(m)\rightarrow C_m$ (which is well defined due to Theorem~\ref{thm:lconv} (3)) is a system of local convexity data for the map $\Phi: X \rightarrow \kt^*$.

  Now, by Theorem~\ref{thm:l2g} and the compactness of $M$, we conclude that
  \begin{enumerate}[label=(\roman*)]
  \item the fibers of $\Phi$ is connected;
  \item $\Phi: X \rightarrow \Phi(X)$ is an open map;
  \item $\Phi(X)$ is a closed convex polytope in $\kt^*$.
  \end{enumerate}
  Then, since $K$ is connected, (i) implies that the fibers of $q \circ \Psi$ is also connected by an argument given in~\cite[Proof of Theorem~3.19]{Birtea_2009op}.
  Due to the openness of $\pi$, (ii) implies that $q\circ \Psi: M \rightarrow q\circ \Psi(M)$ is also open.
  Finally, since $\Phi(X) = q\circ \Psi(M) = \Delta(M)$, (iii) implies the $\Delta(M)$ is a closed convex polytope in $\kt^*$.
\end{proof}

\subsection{Convexity for complex orbit-closures}
In this subsection, we generalize Atiyah's convexity theorem for complex orbit-closures from the K\"ahler manifolds to the Hermitian manifolds.

As a preparation, we state a Morse theoretic property of the generalized moment map.

\begin{lem}
  \label{morse-bott}
  Let $(M,J,g)$ be an almost Hermitian $K$-manifold.
  Suppose that the $K$-action is Hamiltonian and denote the associated generalized moment map by $\Psi$.
  Then, for any $\xi\in \kk$, $\Psi^\xi$ is a Morse--Bott function and has only critical manifolds of even index.
\end{lem}

\begin{proof}
  Let $\omega$ be the Hermitian form for $M$.
  By Lemma \ref{differimage}, the critical set $Z$ of $\Psi^\xi$ is identical to the zero set of the vector field $\xi_M$, or the fixed point set of $T=\overline{\exp ({\Bbb R} \xi)}$.
This implies that each connected component of the critical set is a manifold,~\cite{Kobayashi_1958aa}.
It remains to check that they are non-degenerate and have even indexes.
We follow Atiyah's arguments.
If $V$ is the tangent space to $M$ at $z \in Z$, it has an almost complex structure $J$ and decomposes under the action of torus $T$ as
\begin{equation*}
  V= V_0 \oplus V_1 \oplus \ldots \oplus V_p,
\end{equation*}
where $V_0$ is fixed by $T$ and is the tangent space to $Z$ at $z$, while each $V_j$, for $j >0$, corresponds to a non-trivial character of $T$.
As a result, for $j> 0$, there exist real $\lambda_j \neq 0$, $\lambda_i \ne \lambda_j$ if $i\ne j$, such that for $v_j \in V_j$, the induced action of $\xi$ on $V_j$ is
\begin{equation}
  \label{eq:xi-act}
  \xi\cdot v_j = \lambda_j Jv_j.
\end{equation}

For any $v\in V$, write $v = \sum_i v_i$, $v_i\in V_i$ and extend $v$ to be a vector field $X$ near $z$.
Then the Hessian of $\Psi^\xi$ at $z$ is
\begin{multline*}
  \Hess(\Psi^{\xi})(v,v) = (X(X(\Psi^{\xi})))_z = (X(\omega(\xi_M, X)))_z \\
  = -\omega([\xi_M,X]_z, v) = \omega(\xi\cdot v, v).
\end{multline*}
By (\ref{eq:xi-act}), the above equality implies
\begin{equation*}
  \Hess(\Psi^{\xi})(v,v) = \omega(\sum_{j> 0} \lambda_j Jv_j, v) = -\sum_{j>0} \lambda_jg(v_j,v_j).
\end{equation*}
which is non-degenerate and necessarily of even index.
\end{proof}

\begin{rem}
Under the same assumption of Lemma~\ref{morse-bott}, the critical submanifold $Z$ of $\Psi^{\xi}$ is almost complex and $\omega|_Z$ is non-degenerate consequently.
\end{rem}

In the rest of this subsection, we will assume that $(M,J,g)$ is a complex Hermitian $K$-manifold and denote the associated generalized moment map by $\Psi$.

\begin{thm}[Convexity of complex orbit-closures]
  \label{thm:orbit}
Suppose that $M$ is compact and the $K$-action on $M$ is abelian.
Let $Y$ be a $G$-orbit and $\overline{Y}$ be its closure.
  Set $C_j = \Psi(Z_j \cap \overline{Y})$ if $Z_j \cap \overline{Y} \ne \emptyset$, where $Z_j$ is a connected component of the fixed points set of the $K$-action.
  Then following assertions about $\overline{Y}$ hold.
\begin{enumerate}[label=\normalfont(\arabic*)]
\item $\Psi(\overline{Y})$ is the convex polytope with vertices $\{C_j\}$.
\item For each open face $\sigma$ of $\Psi(\overline{Y})$, $\Psi^{-1}(\sigma) \cap \overline{Y}$ consists of a single $G$-orbit.
\item $\Psi$ induces a homeomorphism of $\overline{Y}/K$ onto $\Psi(\overline{Y})$.
\item For any $y\in Y$, $\xi\in \kk$, the limit $y_{\infty} = \lim_{t\rightarrow +\infty} \exp(-it\xi)\cdot y \in \overline{Y}$ exists.
  And if $z\in \overline{Y}$ satisfying $\Psi^{\xi}(z) = \Psi^{\xi}(y_{\infty})$, then $z\in \overline{G \cdot y_{\infty}}$.
\end{enumerate}
\end{thm}

\begin{rem}
  In the setting of Theorem~\ref{thm:orbit}, the moment body of $M$, $\Psi(M)$, is also the convex hull generated by $\{\Psi(Z_i)\}$.
  However, there is a significant difference between the two polytopes: $\Psi(M)$ and $\Psi(\overline{Y})$.
  That is, the set of vertices (or equivalently extreme points) of $\Psi(M)$ is only a proper subset of $\{\Psi(Z_i)\}$ in general.
  In other words, $\Psi(Z_i)$ can be an interior point of a face of $\Psi(M)$, which can never happen for $\Psi(Z_j \cap \overline{Y}), Z_j \cap \overline{Y} \ne \emptyset$ and $\Psi(\overline{Y})$.\footnote{However, if $Z_j \cap \overline{Y} = \emptyset$, it is possible that $\Psi(Z_i)$ lies on the interior of a face of $\Psi(\overline{Y})$.}
\end{rem}

To show Theorem~\ref{thm:orbit}, we use the Kempf--Ness function associated with the generalized moment map.
Although not fully necessary for the proof of Theorem~\ref{thm:orbit}, we think it may be appropriate to introduce the Kempf--Ness function here, since it also plays a key role in the later part of this note.

Recall that since $\kg = \kk \oplus i \kk$, for an element $\xi\in \kg$, we can define its real and imaginary part using such a splitting.
The following Definition-Proposition holds for any compact group $K$ and its complexification $G$.

\begin{defnprop}
  Fix an element $m \in M$, there exists a unique function $\phi_m:G \rightarrow \mathbb{R}$ such that
  \begin{equation}
    \label{eq:def-kn}
    (\diff \phi_m)_g (g \cdot \xi) = - \langle \Psi(g^{-1}\cdot m), \Im(\xi)\rangle,\quad \phi_m|_K \equiv 0,
  \end{equation}
  where $g\in G$, $\xi\in \kg = \T_{e} G$, $g\cdot \xi\in \T_{g}G$ and $\Im(\xi)$ is the imaginary part of $\xi$.
  We call such a function \emph{the lifted Kempf--Ness function}.
\end{defnprop}

\begin{proof}
  Let $\omega$ be the Hermitian form of $M$.
  The proof of the existence of $\phi_m$ for the K\"ahler case,~\cite[Theorem~4.1]{Georgoulas_2021mo}, also works for our setting without any change.
  In fact, as pointed by~\cite[Remark~5.2.7]{Woodward_2010aa}, such a proof only uses the anti-symmetry of $\omega$.

  Nevertheless, due to Theorem~\ref{thm:ob-close}, a more straightforward proof is also possible.
  Let
  \begin{equation}
    \label{eq:def-lg}
    \Lambda(g) = g^{-1}\cdot m: G \rightarrow M.
  \end{equation}
  Then $\Lambda$ is a holomorphic map.
  By Theorem~\ref{thm:ob-close}, $\Lambda^*\omega$ is a closed $(1,1)$-form on $G$.
  For $k\in K$, denote $R_k$ to be the right translation defined by $k$ on $G$ and $F_k$ to be the diffeomorphism defined by $k$ on $M$.
  Since $\omega$ is $K$-invariant, we can check that
  \begin{equation*}
    R_k^*(\Lambda^*\omega) = (\Lambda\circ R_k)^* \omega = \Lambda^*\circ F^*_{k^{-1}}(\omega) = \Lambda^*(\omega),
  \end{equation*}
  that is, $\Lambda^*\omega$ is invariant under the right translation.
  Moreover, by (\ref{eq:def-lg}), for any $\zeta\in \kk$, we have
  \begin{equation*}
    (\diff \Lambda)_g(g\cdot \zeta) = - \zeta_{M,g^{-1}\cdot m}.
  \end{equation*}
  As a result,
  \begin{equation*}
\iota_{g \cdot \zeta} (\Lambda^*\omega) = - \Lambda^*(\iota_{\zeta_M}\omega)_{g^{-1}\cdot m} = - \Lambda^* (\langle (\diff \Psi)_{g^{-1}\cdot m}, \zeta \rangle) \\
    = - \langle (\diff \Lambda^*\Psi)_g, \zeta\rangle.
  \end{equation*}
  Therefore, $- \Lambda^*\Psi$ is a moment map for $\Lambda^*\omega$ with respect to the right $K$-action on $G$.
  
  Now, by~\cite[Lemma~4.2.1, Theorem~4.2.2 and Theorem~4.2.4]{Guillemin_2005ab}, there is a unique right $K$-invariant function $\phi_m:G \rightarrow \mathbb{R}$ such that
  \begin{gather}
\diff \diffc \phi_m = \Lambda^*\omega,\quad\phi_m|_K \equiv 0,\label{eq:ddc}\\
    \iota_{g \cdot \zeta} \diffc \phi_m = \langle(\Lambda^* \Psi)(g), \zeta\rangle.\label{eq:dc-phi}
  \end{gather}
  where $\diffc = i(\bar{\partial}-\partial)$.
  By the general property of $\diffc$ operator,~\cite[p. 64,~(4.4)]{Guillemin_2005ab},
  \begin{equation*}
    \iota_{g \cdot \zeta} \diffc \phi_m = - (\diff \phi_m)_g (g\cdot (i \zeta)).
  \end{equation*}
  Then (\ref{eq:ddc}), (\ref{eq:dc-phi}) and the above equality implies that (\ref{eq:def-kn}) holds for $\xi\in i\kk$.
  On the other hand, if $\xi\in \kk$, both sides of (\ref{eq:def-kn}) vanish due to the right $K$-invariance of $\phi_m$.

\end{proof}

Note that the above proof implies $\phi_m$ is actually a right $K$-invariant function on $G$.
In other words, $\phi_m$ can descend to a function defined on $G/K$.

\begin{defn}
  \label{def:kn-fct}
  The descended function of $\phi_m$ on $G/K$ is called \emph{the Kempf--Ness function}, denoted by $\kn_m$.
\end{defn}

\begin{lem}
  \label{lm:kn}
  Fix $m\in M$.
  The lifted Kempf--Ness function $\phi_m$ has the following properties.
  \begin{enumerate}[label=\normalfont(\arabic*)]
  \item For any $\xi\in \kk$ and $t\in \mathbb{R}$,
    \begin{equation*}
      \odv[order={2}]{}{t}\phi_m(\exp(it\xi)) \ge 0;
    \end{equation*}
    besides,
    \begin{equation}
      \label{eq:kmeq}
      \odv[order={2}]{}{t}\Big|_{t=0}\phi_m(\exp(it\xi)) = 0,
    \end{equation}
    if and only if $\xi \in \kk_m$.
  \item For any $g, h\in G$,
    \begin{equation}
      \label{eq:equi-kn}
      \phi_m(g) + \phi_{g^{-1}\cdot m}(h) = \phi_m(gh).
    \end{equation}
  \end{enumerate}
  The Kempf--Ness function $\kn_m$ has the same properties.
\end{lem}
\begin{proof}
  \begin{inparaenum}
  \item By (\ref{eq:def-kn}), the derivative of $\phi_m(\exp(it\xi))$ with respect to $t$ is
    \begin{equation*}
      \odv{}{t}\phi_m(\exp(it\xi)) = (\diff \phi_m)_{\exp{it\xi}}(\exp(it\xi)\cdot(i\xi)) 
      = -\langle\Psi(\exp(-it\xi)\cdot m), \xi\rangle.
    \end{equation*}
    Then,
    \begin{multline}
      \label{eq:km-dd}
      \odv[order={2}]{}{t}\phi_m(\exp(it\xi)) = -\odv{}{t}\langle\Psi(\exp(-it\xi)\cdot m), \xi\rangle\\
      = \langle (\diff \Psi)_{\exp(it\xi)\cdot m}(J\xi_M), \xi \rangle = \omega_{\exp(it\xi)\cdot m}(\xi_M, J\xi_M) \ge 0.
    \end{multline}
    Moreover, (\ref{eq:km-dd}) implies that (\ref{eq:kmeq}) holds if and only if $\xi_M(m) =0$, which is equivalent to $\xi\in \kk_m$.

  \item Fix $g$, let $\widetilde{\phi}_{g^{-1}\cdot m}(h) = \phi_m(gh) - \phi_m(g)$.
    For any $k\in K$, by the right $K$-invariance of $\phi_m$,
    \begin{equation*}
      \widetilde{\phi}_{g^{-1}\cdot m}(k) = \phi_m(gk) - \phi_m(g) =0.
    \end{equation*}
    Moreover, by (\ref{eq:def-kn}), we also have
    \begin{align*}
      (\diff \widetilde{\phi}_{g^{-1}\cdot m})_h (h \cdot \xi)
      &= \odv{}{t}\Big|_{t=0} \widetilde{\phi}_{g^{-1}\cdot m}(h\exp(t\xi))\\
      &= \odv{}{t}\Big|_{t=0} \phi_m(gh\exp(t\xi))\\
      &= - \langle \Psi(h^{-1}g^{-1}\cdot m), \Im(\xi)\rangle.
    \end{align*}
    As a result, both $\widetilde{\phi}_{g^{-1}\cdot m}$ and ${\phi}_{g^{-1}\cdot m}$ satisfy (\ref{eq:def-kn}).
    By the uniqueness of the lifted Kempf--Ness function, we have $\widetilde{\phi}_{g^{-1}\cdot m} = {\phi}_{g^{-1}\cdot m}$ and (\ref{eq:equi-kn}) follows.
  \end{inparaenum}
\end{proof}

Lemma~\ref{lm:kn} says that the (lifted) Kempf--Ness function is convex in a suitable sense.
It is a key fact behind many results in the complex geometry, for example, in the following proof of Theorem~\ref{thm:orbit}.
One can find more application of the Kempf--Ness function in Section~\ref{sec:kn-fct}.

\begin{proof}[Proof of Theorem~\ref{thm:orbit}]
  \begin{inparaenum}
  \item We use an argument in~\cite{Biliotti_2018re} to show the convexity.
    Fix an invariant metric on $\kk$ and choose $y\in Y$.
    By using the lifted Kempf--Ness function associated to $y$, we define a function
    \begin{equation*}
      f(\xi) = \phi_y(\exp(i\xi)):\kk^{\perp}_y \rightarrow \mathbb{R}.
    \end{equation*}
    For any $\eta\in \kk^{\perp}_y$, by Lemma~\ref{lm:kn},
    \begin{equation}
      \label{eq:fxe}
      \phi_{y}(\exp(i\xi)\exp(it\eta)) = \phi_{\exp(-i\xi)\cdot y}(\exp(it\eta)) - \phi_y(\exp(i\xi)),
    \end{equation}
    and
    \begin{align*}
      \odv[order={2}]{}{t}\Big|_{t=0} f(\xi + t\eta)
      &= \odv[order={2}]{}{t}\Big|_{t=0} \phi_{y}(\exp(i\xi)\exp(it\eta)) \\
&= \odv[order={2}]{}{t}\Big|_{t=0}\phi_{\exp(-i\xi)\cdot y}(\exp(it\eta)) \ge 0.
    \end{align*}
    And the equality holds if and only if $\eta\in \kk_{\exp(-i\xi)\cdot y} = \kk_y$.
    But $\eta\in \kk^{\perp}_y$, which means that $f(\xi)$ is a strictly convex function on $\kk^{\perp}_y$.

    Then, by the property of strictly convex function,~\cite[p.~122, Theorem~3.5]{Guillemin_1994aa}, the image for the differential of $f$, as a function
    \begin{equation*}
      \diff f: \kk_y^{\perp} \rightarrow \kk_y^{\perp,*}\simeq \kk_y^{\circ},
    \end{equation*}
    is an open convex subset.
    At the same time, by (\ref{eq:def-kn}) and (\ref{eq:fxe}), for $\xi,\eta\in \kk_y^{\perp}$, we have
    \begin{multline*}
      (\diff f)_{\xi}(\eta) = \odv{}{t}\Big|_{t=0} f(\xi + t\eta) = \odv{}{t}\Big|_{t=0} \phi_{\exp(-i\xi)\cdot y}(\exp(it\eta))\\
      = - \langle \Psi(\exp(-i\xi)\cdot y), \eta\rangle.
    \end{multline*}
    Due to Proposition~\ref{isotropy}, $-\Psi(\exp(-i\xi)\cdot y) + \Psi(y) \in \kk_y^{\circ}$.
    Hence, the above equality implies
    \begin{equation*}
      (\diff f)_{\xi} = -\Psi(\exp(-i\xi)\cdot y) + \Psi_1(y),
    \end{equation*}
    where $\Psi_1(y)$ is the projection of $\Psi(y)$ onto $(\kk^{\perp}_y)^{\circ}$.
    As a result, $\Psi(G\cdot y)$ is an open convex subset of $\Psi(y) + \kk_y^{\circ}$.
    Therefore, by the compactness of $M$,
    \begin{equation*}
      \Psi(\overline{Y}) = \overline{\Psi(G\cdot y)}
    \end{equation*}
    is a convex set.

    Now, let $\alpha\in \Psi(\overline{Y})$ be an extreme point of $\Psi(\overline{Y})$ and choose $z\in \overline{Y}$ such that $\Psi(z) = \alpha$.
    By what we have proved, $\Psi(G\cdot z)$ must be an open convex subset of $\alpha + \kk_z^{\circ}$.
    Since $\alpha$ is an extreme point, $\kk_z$ must be $K$, that is, $z \in \cup_j Z_j$.
    Therefore, $\Psi(\overline{Y})$ is the convex polytope generated by $\{C_j\}$.
    
  \item\hspace{-0.7ex}, \item and the characterization of the vertices in (1).
    One can show these results by the exactly same argument used in~\cite[Theorem~2]{Atiyah}.
    Moreover, the technical result (4), which will be used later, follows from (1) and (2).
    However, to give more details about Atiyah's arguments, we prove (4) here.

  \item In Atiyah's proof, only two properties of the moment map are needed.
    \begin{inparaenum}[(a)]
    \item For any $\xi\in \kk$, $\Psi^{\xi}$ is a Morse--Bott function.
    \item The gradient flow associated with $\Psi^{\xi}$ is $G$-equivariant.
    \end{inparaenum}
    Both of them also hold for the generalized moment map.
    (a) follows from Lemma~\ref{morse-bott}.
    For (b), note that due to (\ref{eq:grad-psixi}), for any $m\in M$, $\exp(-it\xi)\cdot m$ is the trajectory of the gradient flow of $-\Psi^{\xi}$, from which (b) follows.

    To see the existence of the limit $\lim_{t\rightarrow +\infty} \exp(-it\xi)\cdot y$, we note that, by (a) and the proof of (b), $\exp(-it\xi)\cdot y$ is a trajectory of the gradient flow of a Morse-Bott function on a compact manifold, which implies that the limit exists by the general property of Morse-Bott functions.

    For the remaining part of (4), we argue as follows.
    By (a), one can define the stable manifold of the gradient flow of $-\Psi^{\xi}$.
    Let $N^s$ be one of such stable manifold satisfying $y\in N^s$ and $N$ be the critical manifold of $N^s$.
By (b) and the $G$-invariance of $N$, $N^s$ is also $G$-invariant.
    Hence, $Y\subseteq N^s$.
    Then, by the general property of the gradient flow,~\cite[(3.7)]{Atiyah} and~\cite[Lemma~10.7]{Kirwan84}, we have
    \begin{gather}
      \overline{Y} \cap N_1 = \emptyset, \label{eq:prop-gf}\\
      \overline{Y} \cap N = \overline{G \cdot y_{\infty}}.\label{eq:prop-gf2}
    \end{gather}
    where $N_1$ is any other critical manifold of $\Psi^{\xi}$ with $\Psi^{\xi}(N_1) \le \Psi^{\xi}(N) = \Psi^{\xi}(y_{\infty})$.

    Now, if $z\notin N^s$, then $z' = \lim_{t\rightarrow + \infty}\exp(-it\xi) \cdot z$ must lie in some critical manifold of $\Psi^{\xi}$ and
    \begin{equation*}
      \Psi^{\xi}(z') \le \Psi^{\xi}(z) = \Psi^{\xi}(y_{\infty}).
    \end{equation*}
    However, since $z'\in \overline{Y}$, this contradicts with (\ref{eq:prop-gf}).
    Therefore, we have $z\in N^s$.
    Together with $\Psi^{\xi}(z) = \Psi^{\xi}(y_{\infty})$, we further know that $z\in N$.
    By (\ref{eq:prop-gf2}), $z\in \overline{G \cdot y_{\infty}}$.
  \end{inparaenum}

\end{proof}

\begin{rem}
In general, the vertices of polytope $P=\Psi({\overline{Y}})$ may not be rational.
However, by Lemma \ref{differimage} and Proposition
\ref{isotropy}, the normal fan $N_P$ of $P$ is rational.  $N_P$
may carry piecewise linear convex (real) function, making the
toric variety defined by $N_P$ a K\"ahler toric variety, therefore
a projective toric variety because it is Moishezon. From this
perspective, $\overline{Y}$, after normalization, may still be a
projective toric variety despite the fact that its ambient space
may not. Here, it is interesting to propose the following: an
arbitrary polytope $P$ whose normal fun is rational is normally
equivalent to a rational polytope. We believe that a geometric
argument along the above line will prove this combinatorial
statement. The notes in this remark grew out of a question of
Allen Knutson and subsequent correspondences with the first author.
\end{rem}

\section{Reduction construction}

Let $(M,J,g)$ be an almost Hermitian $K$-manifold carrying a Hamiltonian $K$-action, whose generalized moment map is $\Psi$ and Hermitian form is $\omega$.
In this section, we discuss the reduction at $p\in \kk^*$ for $M$.
For simplicity, we assume that $p$ is a regular value of $\Psi$.
To avoid the issue with orbifolds, we also assume that $K_p$ acts freely on $\Psi^{-1}(p)$, which in fact implies that $p$ is a regular value of $\Psi$ automatically by Lemma~\ref{differimage}.

\begin{prop}
  \label{prop:red}
  The Hermitian form $\omega$ of $M$ descends to a non-degenerate two-form $\omega_p$ on $M_p= \Psi^{-1}(p)/K_p$ such that
  \begin{equation}
    \label{eq:red-sf}
    i^* \omega = \pi^* \omega_p
  \end{equation}
  where $i: \Psi^{-1}(p) \hookrightarrow M$ is the inclusion and $\pi: \Psi^{-1}(p) \rightarrow M_p$ is the quotient map.
  The almost complex structure $J$ of $M$ also descends to an almost complex structure $J_p$ on $M_p$ and $\omega_p$ is an Hermitian form with respect to $J_p$.
  Furthermore, $J_p$ is integrable if $J$ is so.
\end{prop}

\begin{proof}
  Let $m \in \Psi^{-1}(p)$ be a regular point.
  For any $\xi \in \kk$ and $v\in \T_m \Psi^{-1}(p)$, since
  \begin{equation*}
    \omega_m(v,\xi_{M,m}) = \langle(\diff  \Psi)_m(v) , \xi\rangle = 0,
  \end{equation*}
  we have
  \begin{equation}
    \label{eq:ts-inc}
    \T_m(K\cdot m) \subseteq (\T_m \Psi^{-1}(p))^{\omega}.
  \end{equation}
  On the other hand, due to $K_m \subseteq K_p$ and $K_p$ acting on $\Psi^{-1}(p)$ freely, $K_m$ must be trivial, that is, $\dim \T_m(K\cdot m) = \dim K$.
  Since $p$ is a regular point, we also know that $\dim \T_m \Psi^{-1}(p) = \dim M - \dim K$.
  Now the non-degeneracy of $\omega$ implies that $\dim (\T_m \Psi^{-1}(p))^{\omega} = \dim K$.
  Combining with (\ref{eq:ts-inc}), we have
  \begin{equation}
    \label{eq:ts-eq}
    \T_m(K\cdot m) = (\T_m \Psi^{-1}(p))^{\omega}.
  \end{equation}

  By (\ref{eq:ts-eq}), the null space of $i^*  \omega$ at $m$ is precisely
  \begin{equation*}
    \T_m \Psi^{-1}(p) \cap (\T_m \Psi^{-1}(p))^{\omega} = \T_m \Psi^{-1}(p) \cap \T_m (K \cdot m) = \T_m (K_p \cdot m).
  \end{equation*}
  where the last equality is due to the equivariance of $\Psi$.
  Recall that $\omega$ is $K$-invariant.
  The inclusion (\ref{eq:ts-inc}) already implies that there exists a unique ``push down'' two-form $\omega_p$ on
  \begin{equation}
    \label{eq:tmp}
    \T_{\pi(m)} M_p \simeq  \T_m \Psi^{-1}(p)/\T_m (K_p \cdot m)
  \end{equation}
  as described in the proposition.
  And the equality (\ref{eq:ts-eq}) further implies that $\omega_p$ is non-degenerate.

  It remains to show that the almost complex structure $J$ descends and $\omega_p$ is an Hermitian form with respect to $J_p$.
  In fact, there is a well-defined orthogonal splitting with respect to the metric $g$,
  \begin{equation}
    \label{eq:os}
    \T_m \Psi^{-1}(p) = \T_m (K_p \cdot m) \oplus H_m.
  \end{equation}
  Let $w\in H_m$.
  For any $u\in \T_m (K_p \cdot m)$, the orthogonal splitting (\ref{eq:os}) gives
  \begin{equation*}
    \omega(u, Jw) = g(u, w) = 0.
  \end{equation*}
  Therefore, (\ref{eq:ts-eq}) implies that $Jw \in (\T_m (K_p \cdot m))^{\omega} = \T_m \Psi^{-1}(p)$.
  Meanwhile, since $H_m \in \T_m \Psi^{-1}(p)$, by (\ref{eq:ts-inc}), for any $u\in \T_m (K_p \cdot m)$, the following equality holds,
  \begin{equation*}
    g(u,Jw) = -\omega(u,w) = 0.
  \end{equation*}
  In other words, $Jw$ is orthogonal to $\T_m (K_p \cdot m)$.
By the definition of $H_m$, these mean exactly that $Jw \in H_m$.
  That is, $H_m$ is $J$-invariant.
  Since $H_m$ is isomorphic to $\T_{\pi(m)} M_p$ due to (\ref{eq:tmp}) and (\ref{eq:os}), this shows that $J$ descends to almost complex structure $J_p$ on $\T M_p$ and $\omega_p$ is an Hermitian form with respect to $J_p$.
  
  That $J$ is integrable implies that $J_p$ is also integrable can be proved by using the Newlander--Nirenberg theorem in the same way as in the K\"ahler reduction case.
  The details are left to readers.
\end{proof}

\begin{defn}
  The orbit space $M_p = \Psi^{-1}(p)/K_p$ together with the induced Hermitian two-form and almost complex structure is called the reduction space of $M$ at the point $p$ with respect to the generalized moment map $\Psi$.\end{defn}

The symplectic ``shifting trick'' works equally well in our context.
Let $\mathcal {O}_p^- \subseteq {\frak k}^*$ be the coadjoint orbit passing $p$ and carrying the minus of the canonical symplectic form $\sigma_p$.
For the product space
\begin{equation*}
  M \times \mathcal{O}_p^-
\end{equation*}
the two-form $\omega - \sigma_p$ is a momentumly closed fundamental two-form.
The generalized moment map for the diagonal $K$-action is $\widehat{\Psi}(m, \eta) = \Psi(m) - \eta$, where $(m,\eta) \in M \times \mathcal{O}_p^{-}$.
Then one checks that $\widehat{\Psi}^{-1}(0) /K = \Psi^{-1}(p)/K_p$.
Hence, one can ``shift'' the reduction of $\Psi$ at $p$ to the reduction of $\widehat{\Psi}$ at $0$.

\begin{rem} We ask whether there exists a $K$-action on a symplectic manifold $X$ such that reductions $\Psi_\omega^{-1}(0)/K$ are singular for all Hamiltonian symplectic forms $\omega$ but there exists a momentumly closed Hermitian form $\omega_0$ (necessarily non-symplectic) such that the corresponding reduction $\Psi_{\omega_0}^{-1}(0)/K$ is smooth.
\end{rem}

For the reduction space of a compact symplectic manifold $(N, \omega_N)$, if the symplectic structure comes from a prequantum line bundle $(L,\nabla^L)$ over $N$, that is, $\frac{i}{2\pi}(\nabla^L)^2 = \omega_N$, certain characteristic number defined by the reduction space of $N$ and $N$ itself respectively satisfies an important property: ``quantization commutes with reduction'', often denoted by $[\mathrm{Q},\mathrm{R}] = 0$.
Such a property cannot be generalized to the general almost Hermitian manifold $(M,\omega)$ discussed here because such a line bundle $(L,\nabla^L)$ never exists if $\omega$ is not closed.
However, a variant of the $[\mathrm{Q},\mathrm{R}] = 0$ property proved by~\cite{Meinrenken_1999ab} and~\cite{Tian_1998aa} still holds in our settings.
\begin{thm}
  \label{thm:qr}
Suppose that $M$ is compact and $\Psi^{-1}(0)$ is not empty.
  Let $M_K = \Psi^{-1}(0)/K$ be the reduction space of $M$ at $0$.
  Then, an equality of Todd genus holds

\begin{equation}
    \label{eq:todd-eq}
    \int_M \mathrm{Td}(\T M) = \int_{M_K} \mathrm{Td}(\T M_K).
  \end{equation}
\end{thm}
Recall that by our assumption in this section, the reduction space $M_K$ is smooth.
Therefore, the right hand side of (\ref{eq:todd-eq}) is well defined.

\begin{proof}
  Before the proof, we would like to give a brief explanation about the relation between above result and the classical $[\mathrm{Q},\mathrm{R}] = 0$ principle.

  In $[\mathrm{Q},\mathrm{R}] = 0$, $\mathrm{Q}$ often refers to the index of an elliptic operator.
  In our case, since each almost complex manifold has a canonical spin$^c$ structure, the elliptic operator that we used is the Dirac operator\footnote{Note that $D$, in general, is different from $\sqrt{2}(\bar{\partial} + \bar{\partial}^*)$ by a zeroth order term.}
  \begin{equation*}
    D=
    \begin{bmatrix}
      0 & D_{-} \\
      D_{+} & 0
    \end{bmatrix}
    : \wedge^{0,\mathrm{even}/\mathrm{odd}}(\T^* M) \rightarrow \wedge^{0,\mathrm{odd}/\mathrm{even}}(\T^* M)
  \end{equation*}
  associated with this spin$^c$ structure, where $\wedge^{0,p}(\T^* M)$ is the $(0,p)$-form bundle defined by the almost complex structure,~\cite[\S~3.4]{Friedrich_2000di}.
  By the Atiyah--Singer index theorem,
  \begin{equation*}
    \ind{D_+} = \int_M \mathrm{Td}(\T M).
  \end{equation*}
  Let $(\ker{D_+})^K$ (resp.\ $(\ker{D_-})^K$) be the $K$-invariant subspace of $\ker{D_+}$ (resp.\ $\ker{D_-}$).
  Define
  \begin{equation*}
    (\ind{D_+})^K = \dim (\ker{D_+})^K - \dim (\ker{D_-})^K.
  \end{equation*}
  Then, due to the rigidity of spin$^c$ Dirac operator,~\cite{Hirzebruch_1988el}, we have
  \begin{equation*}
    \ind{D_+} = (\ind{D_+})^K.
  \end{equation*}
  As a result, the equality (\ref{eq:todd-eq}) is equivalent to
  \begin{equation}
    \label{eq:qr-comm}
    (\ind{D_+})^K = \ind{D_{K,+}},
  \end{equation}
  where $D_{K}$ is the Dirac operator defined on $M_K$.
  An equality like (\ref{eq:qr-comm}) is the usual form in which various versions of $[\mathrm{Q},\mathrm{R}] = 0$ principle are stated in the literature.

  We can use the same method in~\cite{Tian_1998aa} to show (\ref{eq:qr-comm}) in our almost Hermitian settings.
  In~\cite{Tian_1998aa}, to calculate $(\ind{D_+})^K$, the authors use a deformed Dirac operator $D + T V$, where $T\ge 0$ is a parameter and $V$ is a zeroth order term defined by $J\diff \norm{\Psi}^2$.
  Let $U$ be an open neighborhood of $\Psi^{-1}(0)$ and $\Omega^{0,*}_{K,c}(M-U)$ be the space consisting of smooth $K$-invariant $(0,*)$-forms with support lying in $M - U$.
  The key point of the whole argument is to show that if $T$ is sufficiently large, $D + TV$ is invertible when restricted to $\Omega^{0,*}_{K,c}(M-U)$,~\cite[Theorem~2.1]{Tian_1998aa}.
  Moreover, since $V$ is invertible outside the critical points sets of $\norm{\Psi}^2$, the problem is further reduced to show that
  \begin{equation*}
    \parbox{25em}{if $T$ large enough, $D+ TV$ is invertible near the critical point $m$ of $\norm{\Psi}^2$ satisfying $m\notin \Psi^{-1}(0)$.}\tag{$\ast$}
  \end{equation*}

  To prove this analytic result, in the symplectic case, a crucical geometric input is provided by~\cite{Kirwan84}, which asserts that for the such a point, the symmetric matrix $\Hess_m{\norm{\Psi}^2}$ always has a strictly negative eigenvalue.
  In our almost Hermitian settings, due to Proposition~\ref{prop:ind-f} proved later, such a property about $\Psi$ still holds.
  Therefore, the fact ($\ast$) is also true without the closedness of $\omega$.
  Once ($\ast$) is proved, (\ref{eq:qr-comm}) can be proved in the same way as in~\cite{Tian_1998aa} without any additional changes.
\end{proof}

\begin{rem}
  If the $K$-action on the regular level set $\Psi^{-1}(0)$ is not free, $M_K$ is a symplectic orbifold in general.
  By using the index theorem for orbifolds~\cite{Kawasaki_1981in}, we believe that an equality similar to (\ref{eq:todd-eq}) also holds in this case.
  Compared to the free action case, since $M_K$ is not smooth in general, the integration on the right hand side of (\ref{eq:todd-eq}) only defines on the regular part.
  Moreover, there will be extra terms on the right hand side of (\ref{eq:todd-eq}) coming from the contribution of orbifold points.
\end{rem}

\section{Variation of reduced two-forms}

We use the same symbol convention as in the previous section.
As before, we assume that $K$ acts on the level set of $\Psi$ freely.
Besides, we also assume that $K$ is an abelian group in this section.
Let $p$ and $q$ be two regular values of $\Psi$ in the same connected component of the set of regular values.
We are going to compare $\omega_p$ and $\omega_q$.

In the symplectic case, the Duistermaat--Heckman formula\footnote{Compared to~\cite{Duistermaat_1982th}, the formula given here has an extra minus sign because we use a different sign convention for the moment map.} states that the de Rham class of $\omega_p$ and $\omega_q$ satisfies
\begin{equation*}
  [\omega_p] - [\omega_q] = - \langle c_1(P), p-q\rangle \in \mathrm{H}^2(M_p, \bR),
\end{equation*}
where $P$ is the principal bundle
\begin{equation*}
  \Psi^{-1}(p) \rightarrow M_p.
\end{equation*}
This formula, as it stands, does not even make sense in our case because $\omega_p$ is not closed in general.
More assumptions on $\omega$ are necessary.

Let $\Delta$ be an open convex subset of $\kk^*$ containing $p$ and $q$.
Although $M_p$ and $M_q$ are diffeomorphic to each other, there is no canonical diffeomorphism between them in general.
For the symplectic case, this problem makes no trouble.
But for our settings, we have to stipulate the diffeomorphism explicitly.

Recall that $\Psi: \Psi^{-1}(\Delta) \rightarrow \Delta$ is a fibration over $\Delta$.
Fix a horizontal distribution $H$ for this fibration.
In other words, $H$ is a subbundle of $\T \Psi^{-1}(\Delta)$ and $\T \Psi^{-1}(\Delta) = H \oplus \T^{V} \Psi^{-1}(\Delta)$, where $\T^{V} \Psi^{-1}(\Delta)$ is the vertical subbundle for the fibration.
As usual, we will assume that $H$ is $K$-invariant.

With $H$, for each straight line passing through $p$ and $\Psi(m) = p$, there is a unique horizontal curve passing through $m$.
Therefore, there exists a $K$-equivariant projection $\eta: \Psi^{-1}(\Delta) \rightarrow \Psi^{-1}(p)$.
And $\Psi^{-1}(\Delta)$ has a $K$-equivariant trivialization
\begin{equation*}
  \Psi \times \eta: \Psi^{-1}(\Delta) \rightarrow \Delta \times \Psi^{-1}(p),
\end{equation*}
which induces identifications
\begin{equation*}
  \Psi^{-1}(p) \cong \Psi^{-1}(q)\text{ and } M_p \cong M_q
\end{equation*}
for any $q \in \Delta$.
We will compare $\omega_p$ and $\omega_q$ based on the above identifications.

\begin{rem}
  Since we assume that $M$ is an almost Hermitian manifold, there is a natural candidate for $H$.
  Namely, for any $m\in M$, $H$ at $m$ is $J\cdot \kk_m$.
  By (\ref{eq:grad-psixi}), $H$ is orthogonal to $\T^{V} \Psi^{-1}(\Delta)$, which implies that $H$ is a horizontal distribution.
  It seems to be a natural question to ask when the trivialization $\Psi \times \eta$ defined this way satisfies the following definition.
\end{rem}

\begin{defn}
The trivialization $\Psi \times \eta$ is called good for $\omega$ if for any vector field $A$ on $\Delta$, $\iota_{\widetilde{A}}\diff \omega$ is a horizontal form, where $\widetilde{A} = (A, 0)$ is the vector field\footnote{Caution! $\widetilde{A}$ is not the horizontal lift-up of $A$ with respect to $H$ in general.} on $\Psi^{-1}(\Delta)$ induced by $A$ using the product structure on $\Delta \times \Psi^{-1}(p)$.
\end{defn}

Recall that a differential form $\alpha$ on $\Psi^{-1}(\Delta)$ is called horizontal if and only if $\iota_X\alpha = 0$ for any $X \in \Gamma(\T^{V} \Psi^{-1}(\Delta))$.

\begin{thm}[Generalized Duistermaat--Heckman's Theorem]
  \label{DH} Assume that $\Psi \times \eta$ is a good trivialization for $\omega$.
  Then
  \begin{equation*}
    [\omega_p - \omega_q] = -\langle c_1(P), p-q \rangle.
  \end{equation*}
\end{thm}

\begin{proof}
  We follow Duistermaat--Heckman's original arguments,~\cite{Duistermaat_1982th}.
  Identify $\Psi^{-1}(\Delta)$ with $\Delta \times \Psi^{-1}(p)$ using $\Psi \times \eta$.
  Let $\pi: \Psi^{-1}(p) \rightarrow M_p$ be the projection.
For any $\lambda \in \kk^*$, let $\partial_{\lambda}$ be the directional derivative along the direction $\lambda$ on $\Delta$ and let $\widetilde{\lambda} = (\lambda, 0)$ be the vector field on $\Delta \times \Psi^{-1}(p)$ induced by $\lambda$.

  Note that due to the isomorphism $\Psi \times \eta$, by varying $\xi$ in $\Delta$, the reduced forms $\omega_{\xi}$ (resp.\ $\pi^*\omega_{\xi}$) can be viewed as a map from $\Delta$ to $\Omega^2(M_p)$ (resp.\ $\Omega^2(\Psi^{-1}(p))$).
  Let $i_{\xi}$ be the inclusion $\Psi^{-1}(p) \hookrightarrow \{\xi\} \times \Psi^{-1}(p)  \subseteq \Delta \times \Psi^{-1}(p)$.
  Then, (\ref{eq:red-sf}) implies that $\pi^* \omega_{\xi} = i^*_{\xi} \omega$.
  We have
  \begin{equation}
    \label{eq:diff-wxi}
    \begin{aligned}
      \pi^* (\partial_\lambda \omega_\xi) & = \partial_\lambda (\pi^* \omega_\xi) = \partial_\lambda (i^*_{\xi} \omega)\\
                    & = i^*_{\xi} (\mathcal{L}_{\widetilde{\lambda}} \omega) = i^*_{\xi} (\diff (\iota_{\widetilde{\lambda}} \omega)) + i^*_{\xi} ( \iota_{\widetilde{\lambda}} \diff \omega) \\
                    & = i^*_{\xi} (\diff (\iota_{\widetilde{\lambda}} \omega)) = \diff ( i^*_{\xi}(\iota_{\widetilde{\lambda}} \omega)),
    \end{aligned}
  \end{equation}
  where for the third equality is due to the definition of $\widetilde{\lambda}$ and for the fifth equality holds because $\Psi \times \eta$ is good for $\omega$.

  For each $\xi\in \Delta$, the map
  \begin{equation*}
    \theta_\xi : \lambda \to -i^*_\xi (\iota_{\widetilde{\lambda}} \omega)
  \end{equation*}
  defines a ${\kk}$-valued one-form on $\Psi^{-1}(p)$.
  For any $X \in {\frak k}$ and $m\in \Psi^{-1}(p)$,
  \begin{multline*}
    \langle \iota_{X_{\Psi^{-1}(p)}} \theta_{\xi}, \lambda \rangle_m = -(\iota_{X_{\Psi^{-1}(p)}} i^*_\xi (\iota_{\widetilde{\lambda}} \omega) )_m = -(\iota_{X_{M}} (\iota_{\widetilde{\lambda}} \omega) )_m\\
    = (\iota_{\widetilde{\lambda}}(\iota_{X_{M}} \omega) )_m = \langle(\diff \Psi)_m (\widetilde{\lambda}), X\rangle = \langle\lambda, X \rangle,
  \end{multline*}
  where the last equality is due to the definition of $\widetilde{\lambda}$.
One sees that $\theta_\xi$ is a connection one-form of the principal $K$-bundle $P:\Psi^{-1}(p) \rightarrow M_p$.
  Since $K$ is an abelian group, there is a $\kk$-valued curvature two-form $\Omega_\xi$ on $M_p$ such that
  \begin{equation*}
    \diff \theta_\xi = \pi^* \Omega_\xi.
  \end{equation*}
  Conbining the above equality and (\ref{eq:diff-wxi}), we have
  \begin{equation*}
    \pi^* (\partial_\lambda \omega_\xi) = - \pi^* \langle\Omega_{\xi}, \lambda \rangle.
  \end{equation*}
  Since $\pi$ is a submersion, the above equality implies that
  \begin{equation*}
    \partial_\lambda \omega_\xi = - \langle\Omega_{\xi}, \lambda \rangle.
  \end{equation*}
Therefore, at level of the de Rham cohomology, we have
  \begin{equation*}
    [\partial_\lambda \omega_\xi] = -\langle c_1(P), \lambda \rangle
  \end{equation*}
  where $c_1(P)$ is the first Chern class of the bundle $P$.
  This proves the formula in the theorem.
\end{proof}

\begin{rem}
When $\omega$ is symplectic, it follows from the Darboux--Weinstein theorem that the two-form $\omega$ is unique near $\Psi^{-1}(p)$.
  With Theorem~\ref{thm:darboux}, it is reasonable to expect a similar result remains valid in a realm outside the symplectic territory.
  However, even if $\Psi \times \eta$ is good $\omega$, we suspect that this is not true in general except the case that $\omega_p$ is closed.
  The reason is that $\omega$ may contain terms like $\gamma \wedge \diff \Psi$, where $\gamma$ comes from $M_p$.
  If $\omega_p$ is not closed, such terms can't be absorbed by diffeomorphisms in general.
\end{rem}

\section{Kirwan--Ness stratification}\label{sec:kn}

Let $(M,J,g)$ be an almost Hermitian $K$-manifold carrying a Hamiltonian $K$-action, whose generalized moment map is $\Psi$ and Hermitian form is $\omega$.
In this section, we also assume that $M$ is compact.
Besides, we will fix a $K$-invariant inner product on $\kk$ and identify $\kk$ with $\kk^*$ via such an inner product when necessary.
We will discuss a stratification of $M$ naturally associated with $\norm{\Psi}^2$ in this section.

\subsection{Almost Hermitian case}
Recall that, by (\ref{eq:grad-psixi}),  with respect to $g$, the gradient of $\Psi^{\xi}$ is $J\xi_M$.
Likewise, the gradient of the norm-square of the generalized moment map $\norm{\Psi}^2: M \rightarrow \mathbb{R}$ also has a clean expression:
\begin{equation}
  \label{eq:grad-m2}
  (\grad \norm{\Psi}^2)_m = 2 J (\Psi (m))_{M,m}\text{ for any }m\in M.
\end{equation}
To check this, let $\{ \xi_1, \ldots, \xi_N \}$ be an orthonormal basis for ${\kk} \cong {\kk}^*$.
Then
\begin{equation*}
  \Psi = \sum_{i=1}^N \Psi^{\xi_i} \xi_i\quad\text{and}\quad\norm{\Psi}^2 = \sum_{i=1}^N |\Psi^{\xi_i}|^2.
\end{equation*}
Therefore,
\begin{align*}
  (\grad \norm{\Psi}^2)_m &= \sum_{i=1}^N 2 \Psi^{\xi_i}(m) (\grad \Psi^{\xi_i})_m\\
                       &= \sum_{i=1}^N 2 \Psi^{\xi_i} (m) J (\xi_i)_{M,m} = 2 J (\Psi (m))_{M,m}.
\end{align*}

For the simplicity, in this section, we denote $\norm{\Psi}^2/2$ by $f$.
Unlike $\Psi^{\xi}$, $f$ is not a Morse--Bott function.
But $f$ does hehave like a Morse--Bott function in many ways.
Especially, the gradient flow of $f$ also gives a stratification of $M$ consisting of smooth submanifolds.

As before, choose $T\subseteq K$ to be a maximal torus and $\kt^*_+$ to be a closed positive Weyl chamber.
\begin{lem}
  \label{lm:finite-cv}
  Let $\mathrm{crit}(f)$ be the critical points set of $f$.
  Then $\Psi(\mathrm{crit}(f)) \cap \kt^*_+$ is a finite set.
  Especially, the $f$ has only finitely many critical values.
\end{lem}

\begin{proof}
  We will identify $\kk$ and $\kk^*$ using the $K$-invariant inner product on $\kk$.
  We first assume that $K$ is abelian.
  As in Proposition~\ref{isotropy}, for a subgroup $H\subseteq K$, let $M_H = \{m \in M | K_m =H\}$ and take $M_H^i$ to be a connected component of $M_H$.
  Then, by the slice theorem of the $K$-action, $M^i_H$ is a smooth submanifold of $M$ and we have a finite decomposition
  \begin{equation}
    \label{eq:orb-decomp}
    M = \bigcup_{H\subseteq K, i} M_H^i.
  \end{equation}

  Choose $m\in M_H^i \cap \mathrm{crit}(f)$.
  By Proposition~\ref{isotropy}, $m\in M_H^i$ implies that
  \begin{equation*}
    \Psi(m) \in p + \kh^{\perp}
  \end{equation*}
  for some $p\in \kk$ independent of $m$.
  On the other hand, by (\ref{eq:grad-m2}), $m\in \mathrm{crit}(f)$ implies that
  \begin{equation*}
    \Psi(m) \in \kk_m = \kh.
  \end{equation*}
  The above two equations imply that for any $m\in M_H^i \cap \mathrm{crit}(f)$, $\Psi(m)$ takes the same value (the orthogonal projection of $p$ on $\kh$).
  Now, since the decomposition (\ref{eq:orb-decomp}) is finite, $\Psi(\mathrm{crit}(f)) = \cup_{H,i} \Psi(M_H^i \cap \mathrm{crit}(f))$ is also finite.

  For the general compact group $K$, let $\Psi_T$ be the induced generalized moment map of the $T$-action and $f_T = \norm{\Psi_T}^2/2$.
  Take $m \in \mathrm{crit}(f)$ with $\Psi(m)\in \kt^* \simeq \kt$.
  By (\ref{eq:grad-m2}), $\Psi(m) \in \kk_m$.
  Moreover, since $\Psi(m)\in \kt$, we further know that
  \begin{equation*}
    \Psi(m)\in \kk_m\cap \kt = \kt_m.
  \end{equation*}
  Meanwhile, $\Psi(m)\in \kt$ also implies that $\Psi_T(m) = \Psi(m)$.
  Hence, by (\ref{eq:grad-m2}) again, $m$ is also a critical point of $f_T$.
  In other words,
  \begin{equation*}
    \Psi(\mathrm{crit}(f)) \cap \kt^*_+ \subseteq \Psi_T(\mathrm{crit}(f_T))
  \end{equation*}
  Therefore, the finiteness of $\Psi(\mathrm{crit}(f)) \cap \kt^*_+$ follows from the finiteness of the abelian case.
\end{proof}

For each $\lambda\in \Psi(\mathrm{crit}(f)) \cap \kt^*_+$, define $C_{\lambda}$ to be $\Psi^{-1}(K\cdot \lambda) \cap \mathrm{crit}(f)$.
Then, $\cup_{\lambda} C_{\lambda}$ is a decomposition of the critical points set of $f$.

\begin{lem}
  \label{lm:conv-mp}
  Let $\varphi_t: M \rightarrow M$, $t \ge 0$, be the gradient flow associated with $-f$.
  Then, for any $m\in M$, the limit $\lim_{t\rightarrow +\infty} \varphi_t(m)$ exists.
\end{lem}

\begin{proof}
  For any $m\in M$, by Theorem~\ref{thm:nf}, there exists a coordinate neighborhood $U$ of $m$ such that $f|_U$ is a real analytic function with respect to this coordinates.\footnote{But this does not means that $f$ is analytic with respect to the analytic structure on $M$ provided by Whitney's theorem.}
  Then one can apply the {\L}ojasiewicz gradient inequality for $f$ as in~\cite[Theorem~3.3]{Georgoulas_2021mo} to conclude the existence of the limit.
\end{proof}

Due to Lemma~\ref{lm:conv-mp}, for any $C_{\lambda}$, we can define
\begin{equation*}
  W^s_\lambda = \{m \in M | \lim_{t\rightarrow +\infty} \varphi_t(m) \in C_{\lambda}\}.
\end{equation*}
Since the limit in Lemma~\ref{lm:conv-mp} must be a critical point of $f$, similar to the Morse--Bott function, $M$ has a stratification associated with $f$,
\begin{equation}
  \label{eq:kn-decomp}
  M = \bigcup_{\lambda} W^s_\lambda.
\end{equation}

\begin{rem}
  Let $g'$ be another $K$-invariant metric on $M$ (not necessarily compatible with $\omega$)
and $\varphi'_t$ be the gradient flow associated with $-f$ defined by $g'$.
  Note that in the proof of Lemma~\ref{lm:conv-mp}, we only use the local analyticity of $f$, which is independent of the choice of metric.
  Therefore, the result of Lemma~\ref{lm:conv-mp} also holds for $\varphi_t'$.
  Consequently, for each $C_{\lambda}$, we can define a subset $W^s_{\lambda}(g')$ similar to $W^s_{\lambda}$ by using $g'$.
\end{rem}

In the symplectic case, Kirwan~\cite{Kirwan84} shows that when the metric is suitably chosen, the strata in above stratification are smooth, which remains true in our settings.

\begin{thm}[{\cite[Theorem~4.16 \& 5.4]{Kirwan84}}]
  \label{thm:st}
  There exists a $K$-invariant Riemannian metric $g_0$ on $M$, such that for each $\lambda\in \Psi(\mathrm{crit}(f)) \cap \kt^*_+$, $W^s_\lambda(g_0)$ is a smooth submanifold of $M$.
  The spectral sequence for the equivariant stratification $M = \cup_{\lambda} W^s_{\lambda}(g_0)$ collapses at the second page.
\end{thm}

\begin{proof}
  The proof given by Kirwan for the symplectic case also works here without any changes.
  In fact, Kirwan's proof uses solely the non-degeneracy of the form $\omega$ but not the closedness.
  See~\cite[\S\S~4 \& 5]{Kirwan84} for the details.
For the readers' convenience, we summarize the gists of the proof briefly.

  Compared to the Morse--Bott function, the main difficulty about $f$ is that the Morse Lemma no longer holds near the critical point of $f$.
  As a result, for the Riemannian metric $g$ we have chosen, the smoothness of $W^s_{\lambda}$ (defined by $g$) near $C_{\lambda}$ does not follows from the ODE theory directly.
  Instead, Kirwan constructs a smooth submanifold near the $C_{\lambda}$ directly, denoted by $\Sigma_{\lambda}$, to replace $W^s_{\lambda}$.
  
  Let $m\in \mathrm{crit}{(f)} \cap \Psi^{-1}(\lambda)$.
  Due to the equivariance of $\Psi$, we only need to construct $\Sigma_{\lambda}$ near $m$.
  Let $Z_{\lambda} = \cup_{i} Z_{\lambda,i}$ be the union of connected components of fixed points set of the subgroup generated by $\lambda$ satisfying $Z_{\lambda,i} \cap C_{\lambda} \neq \emptyset$.
  Clearly, $m\in Z_{\lambda}$.
  And let $Y_{\lambda}$ be the stable submanifold of $Z_{\lambda}$ with respect to the gradient flow of $-\Psi^{\lambda}$.
  Now, in a sufficiently small open neighborhood $V$ of $m$, $\Sigma_{\lambda}$ is defined to be the subset $KY_{\lambda} \cap V$.
  It turns out that $\Sigma_{\lambda}$ is a smooth submanifold of $M$, which is, in fact, diffeomorphic to an open subset of $K \times_{K_{\lambda}} Y_{\lambda}$,~\cite[Corollary~4.11]{Kirwan84}.

  The property of $\Sigma_{\lambda}$ is similar to the stable submanifold of a Morse--Bott function,~\cite[Proposition~4.15]{Kirwan84}.
  For example, the codimension of $\Sigma_{\lambda}$ is equal to the index of $\Hess_m{f}$ and $\Hess_m{f}$ is semi-positive when restricting on $\T_m \Sigma_{\lambda}$.
  Moreover, $f|_{\Sigma_{\lambda}}$ takes the minimum value at $m$.
  
  Then, based on $f$ and $\Sigma_{\lambda}$, Kirwan~\cite[Lemma~10.5]{Kirwan84} shows that there exists another $K$-invariant Riemannian metric $g_0$ on $M$ such that for each $\lambda$, $W^s_{\lambda}(g_0)$ coincides with $\Sigma_{\lambda}$ near $C_{\lambda}$, which implies the smoothness of $W^s_{\lambda}(g_0)$ for all $\lambda$.
  Note that $g_0$ is not compatible with $\omega$ in general.
\end{proof}

\begin{prop}
  \label{prop:ind-f}
  About the stratification $M = \cup_{\lambda} W^s_{\lambda}(g_0)$, the following properties hold.
  \begin{enumerate}[label=\normalfont(\arabic*)]
  \item For each $\lambda\in \Psi(\mathrm{crit}(f)) \cap \kt^*_+$, $\dim W^s_{\lambda}(g_0)$ is even.
  \item There is a unique open stratum, which is dense and connected.
  \item There is a unique $\lambda_0 \in \Psi(\mathrm{crit}(f)) \cap \kt^*_+$ such that $\norm{\lambda_0}^2/2$ is the minimum of $f$.
    Moreover, $W^s_{\lambda_0}(g_0)$ is the unique open stratum.
  \item For any critical point of $f$ lying outside the open stratum, the index of $f$ at the point is a strictly positive even number.
    Hence, any local minimum point of $f$ is a global minimum point of $f$.
  \end{enumerate}
\end{prop}

\begin{proof}
  We use the same notation as in the proof of Theorem~\ref{thm:st}.

  \begin{inparaenum}
  \item Take $m\in \mathrm{crit}(f) \cap \Psi^{-1}(\kt^*_+) \cap W^s_{\lambda}(g_0)$.
    By definition, near $m$, $W^s_{\lambda}(g_0)$ coincides with $\Sigma_{\lambda}$.
    As a result, $\dim W^s_{\lambda}(g_0) = \dim \Sigma_{\lambda}$.
    Moreover, since $\Sigma_{\lambda}$ is diffeomorphic to an open subset of $K \times_{K_{\lambda}} Y_{\lambda}$, we have
    \begin{equation*}
      \dim W^s_{\lambda}(g_0) = \dim K - \dim K_{\lambda} + \dim Y_{\lambda}.
    \end{equation*}
    Recall that $Y_{\lambda}$ is the stable manifold of $Z_{\lambda}$ for the gradient flow of $-\Psi^{\lambda}$.
    As a result, the codimension of $Y_{\lambda}$ is equal to the index of $\Psi^{\lambda}$ at $Z_{\lambda}$.
    By Lemma~\ref{morse-bott}, the dimension of $Y_{\lambda}$ is even.
    Note that $\dim K - \dim K_{\lambda}$ is always an even number.
    Then $\dim W^s_{\lambda}(g_0)$ is also even.

  \item Now, all the strata are of even dimensional.
    As a consequence,
    \begin{equation*}
      M - \bigcup_{\dim W^s_{\lambda}(g_0) < \dim M} W^s_{\lambda}(g_0)
    \end{equation*}
    is a connected set.
    In particular, there is only one open stratum, which is dense and connected.

  \item Let $\lambda_0 \in \Psi(\mathrm{crit}(f)) \cap \kt^*_+$ such that $\norm{\lambda_0}^2/2$ is the minimum of $f$.
    Take $m\in \mathrm{crit}(f) \cap \Psi^{-1}(\lambda_0)$.
    Since $f(m) = \norm{\lambda_0}^2/2$ is the minimum of $f$, the index of $\Hess_m f$ is zero.
    Recall that the codimension of $W^s_{\lambda_0}(g_0)$ is equal to the index of $\Hess_m f$, we know that $\dim W^s_{\lambda_0}(g_0) = \dim M$, that is, $W^s_{\lambda_0}(g_0)$ is an open stratum.
    If there is another $\lambda'_0\in \Psi(\mathrm{crit}(f)) \cap \kt^*_+$ also satisfying $\norm{\lambda'_0}^2/2$ is the minimum of $f$, $W^s_{\lambda'_0}(g_0)$ must be also an open stratum.
    By the uniqueness of the open stratum, $\lambda_0$ and $\lambda'_0$ must be the same.

  \item Let $m\in \mathrm{crit}(f) \cap W^s_{\lambda}(g_0)$.
    Again, we have that the index of $f$ at $m$ is equal to $\dim M - \dim W^s_{\lambda}(g_0)$.
Therefore, since $\dim W^s_{\lambda}(g_0)$ is even, the index of $f$ at $m$ is even.
    Besides, if $W^s_{\lambda}(g_0)$ is not open, the index of $f$ at $m$ is also strictly positive.
  \end{inparaenum}
\end{proof}

\subsection{Complex case}
As in~\cite[Theorem~6.18]{Kirwan84}, if $M$ is a complex manifold, the result of Theorem~\ref{thm:st} can be strengthened as follows.\begin{thm}
  \label{thm:cpx-kn}
  Assume that $g$ is a $K$-invariant Hermitian metric on $M$ with the generalized moment map $\Psi$.
  Then,
  \begin{enumerate}[label=\normalfont(\arabic*)]
  \item each stratum in the stratification (\ref{eq:kn-decomp}) associated with $g$ is a locally closed submanifold;
  \item and each stratum is complex and $G$-invariant.
  \end{enumerate}
\end{thm}

Note that compared to Theorem~\ref{thm:st}, there is no need to choose another metric here.
In the K\"ahler case, the stratification (\ref{eq:kn-decomp}) is called the Kirwan--Ness stratification sometimes,~\cite[\S~7]{Woodward_2010aa}.

We will prove two parts of Theorem~\ref{thm:cpx-kn} separately.
Each parts need some auxiliary results that we will prove in \S~\ref{sec:kn-fct}.

For the proof of the first part of Theorem~\ref{thm:cpx-kn}, we begin with a lemma.

\begin{lem}
  \label{lm:s0}
  Assume that $\Psi^{-1}(0) \neq \emptyset$.
  Then stratum $W^{s}_0$ in (\ref{eq:kn-decomp}) is open and $G$-invariant.
  In fact,
  \begin{equation}
    \label{eq:w0-gcl}
    W^s_0 = \{m\in M | \overline{G\cdot m} \cap \Psi^{-1}(0) \ne \emptyset\}.
  \end{equation}
\end{lem}

From our perspective, a priori, the $G$-invariance of $W^s_0$ is not trivial, which, however, is necessary for the proof of Theorem~\ref{thm:cpx-kn}.
Therefore, we decide to incorporate a proof of this fact, which is left to \S~\ref{sec:kn-fct} (after Proposition~\ref{prop:ws0}).

\begin{proof}[Proof of {Theorem~\ref{thm:cpx-kn} (1)}]
  Let $\lambda\in \Psi(\mathrm{crit}(f)) \cap \kt^*_+$ and $m\in \mathrm{crit}{(f)} \cap \Psi^{-1}(\lambda)$.
  Recall that in the proof of Theorem~\ref{thm:st}, we use the submanifolds $Z_{\lambda}$ and $Y_{\lambda}$ associated with the Morse--Bott function $\Psi^{\lambda}$ and $m$.
  Since the group action is holomorphic, $Z_{\lambda}$ and $Y_{\lambda}$ are also holomorphic.
  Note that $Z_{\lambda}$ is $K_{\lambda}$-invariant, which means that we can define $\Psi_{K_{\lambda}}$, the generalized moment map of the $K_{\lambda}$-action, on $Z_{\lambda}$.
  Meanwhile, for $x \in Z_{\lambda}$ and identifying $\kk$ with $\kk^*$,
  \begin{equation*}
    [\lambda, \Psi(x)] = (\diff \Psi)_x (\lambda_{M,x}) = 0.
  \end{equation*}
  Therefore, $\Psi(x) \in \kk_{\lambda}$, which implies
  \begin{equation}
    \label{eq:psi-l-psi}
    \Psi_{K_{\lambda}}(x) = \Psi(x),\text{ for }x\in Z_{\lambda}.
  \end{equation}
Moreover, since $\Psi^{\lambda}|_{Z_{\lambda}} = \Psi^{\lambda}(m) = \norm{\lambda}^2$, for $x\in Z_{\lambda}$
  \begin{multline}
    \label{eq:min-zl}
    \norm{\Psi_{K_{\lambda}}(x)}^2 = \norm{\Psi(x)}^2 = \norm{\Psi(x) - \lambda}^2 + 2 \langle \Psi(x) - \lambda, \lambda \rangle + \norm{\lambda}^2\\
    = \norm{\Psi(x) - \lambda}^2 + 2 \Psi^{\lambda}(x) - \norm{\lambda}^2 \ge \norm{\lambda}^2.
  \end{multline}
  Define a subset of $Z_{\lambda}$ as follows,
  \begin{equation}
    \label{eq:def-zss}
    Z_{\lambda}^{\rss} = \{x \in Z_{\lambda} | \lim_{t\rightarrow +\infty} \varphi_t(x) \in \Psi_{K_{\lambda}}^{-1}(\lambda) = \Psi^{-1}(\lambda)\cap Z_{\lambda}\}.
  \end{equation}
  By (\ref{eq:grad-m2}) and (\ref{eq:psi-l-psi}), the gradient flow of $-\norm{\Psi_{K_{\lambda}}}^2/2$ and $f$ coincide on $Z_{\lambda}$.
  And (\ref{eq:min-zl}) implies that $Z_{\lambda}^{\rss}$ is the stratum corresponding to the minimum of $\norm{\Psi_{K_{\lambda}}}^2/2$ with respect to such a flow.
  Moreover, since $\lambda$ lies in the center of $\kk_{\lambda}$, $\Psi_{K_{\lambda}} - \lambda$ is also the generalized moment map of the $K_{\lambda}$ on $Z_{\lambda}$ and $Z^{\rss}_{\lambda}$ coincides with the $W^{s}_{0}$ stratum defined by $\Psi_{K_{\lambda}} - \lambda$.
  Hence, by Lemma~\ref{lm:s0}, $Z_{\lambda}^{\rss}$ is an open and $K^{\mathbb{C}}_{\lambda}$-invariant subset of $Z_{\lambda}$.
  Let $\pi: Y_{\lambda} \rightarrow Z_{\lambda}$ be the canonical map induced from the gradient flow of $-\Psi^{\lambda}$.
  We define $Y_{\lambda}^{\rss}$ to be $\pi^{-1}(Z_{\lambda}^{\rss})$.
  Figure~\ref{fig:yss} is an illustration of $Y_{\lambda}^{\rss}$ when $Z_{\lambda}$ has only one component.
  \begin{figure}[htbp]
    \centering
    \includegraphics[height=250pt]{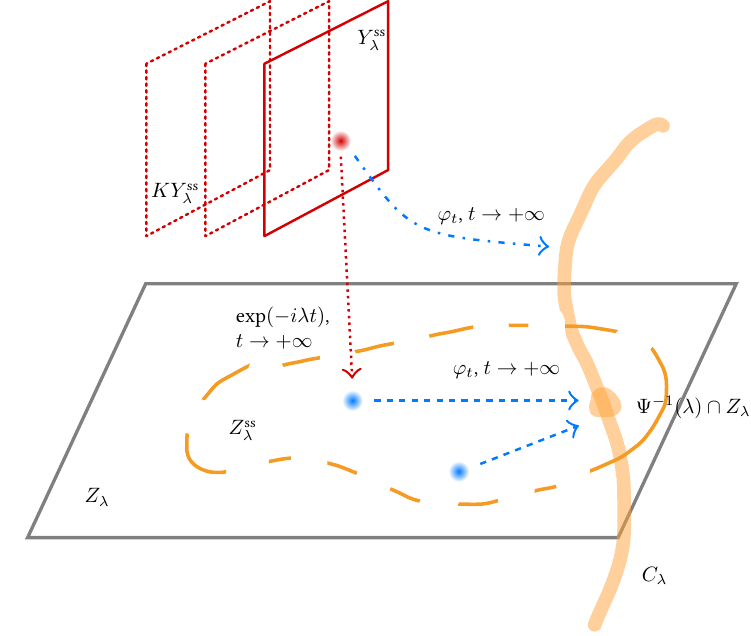}
    \caption[Yss]{An illustration of $Y_{\lambda}^{\rss}$.}
    \label{fig:yss}
  \end{figure}

For any $\xi\in \kt_+$, let
  \begin{equation*}
    P_{\xi} = \{g\in G| \lim_{t\rightarrow -\infty}\exp(i t \xi) g \exp(-it\xi) \text{ exists}\}.
  \end{equation*}
  to be the parabolic subgroup associated with $\xi$.
  Since $Z_{\lambda}^{\rss}$ is $K^{\mathbb{C}}_{\lambda}$-invariant, $Y_{\lambda}^{\rss}$ is invariant under the $P_{\lambda}$-action,~\cite[Lemma~6.10]{Kirwan84}.
  Note that $KY_{\lambda}^{\rss}$ is an open subset of $KY_{\lambda}$ containing $C_{\lambda}$.
  Therefore, as in the proof of Theorem~\ref{thm:st}, a small neighborhood of $C_{\lambda}$ in $KY_{\lambda}^{\rss}$, denoted also by $\Sigma_{\lambda}$, is also a smooth submanifold.
  On the other hand, the parabolic group $P_{\lambda}$ satisfies $KP_{\lambda} = G$, which implies that $KY_{\lambda}^{\rss} = KP_{\lambda}Y_{\lambda}^{\rss} = GY_{\lambda}^{\rss}$, that is, $KY_{\lambda}^{\rss}$ is $G$-invariant.
  Then, by (\ref{eq:grad-m2}), for any $x\in \Sigma_{\lambda}$,
  \begin{equation*}
    (\grad f)_x \in \kg \cdot x \subseteq \T_x \Sigma_{\lambda}.
  \end{equation*}
  As a result, unlike Theorem~\ref{thm:st}, it is not necessary to introduce another metric on $M$.
  Instead, the stratum $W^s_{\lambda}$ associated with the Hermitian metric $g$ coincides with $\Sigma_{\lambda}$ in a neighborhood of $C_{\lambda}$ automatically.
  Then, $\{W^s_{\lambda}\}$ forms a smooth stratification by the properties of minimally degenerate Morse function,~\cite[Theorem~10.4]{Kirwan84}.
\end{proof}

For the proof of the second part of Theorem~\ref{thm:cpx-kn}, the key is to show that
\begin{equation}
  \label{eq:no-int}
  KY_{\lambda}^{\rss} \cap KY_{\lambda'}^{\rss} = \emptyset, \text{ if }\lambda \neq \lambda'.
\end{equation}
In~\cite{Kirwan84}, Kirwan deduces (\ref{eq:no-int}) from the fact: if $x\in KY_{\lambda}^{\rss}$, then $\lambda$ is the unique point closest to $0$ of $\Psi(\overline{G\cdot x}) \cap \kt_+$,~\cite[Corollary~6.12]{Kirwan84}.
As before, the same proof also works in our Hermitian setting.
However, rather than repeating Kirwan's argument, we decide to deduce this result from another characterization of $Y^{\rss}_{\lambda}$, which itself is another application of the Kempf-Ness function.
Such a characterization will be proved in \S~\ref{sec:kn-fct}.
Here, we summarize the properties of $Y^{\rss}_{\lambda}$ needed for the proof of Theorem~\ref{thm:cpx-kn} (2).
\begin{lem}[{=Thereom~\ref{thm:yss}}]
  \label{lm:yss}
  Suppose $\lambda\in \Psi(\mathrm{crit}(f)) \cap \kt^*_+$, $\lambda\neq 0$.
  Then,
  \begin{inparaenum}[(i)]\item $Y^{\rss}_{\lambda} \subseteq \cup_{\rho\neq 0} W^s_{\rho}$;
  \item there exists a $K$-equivariant function $\wh: \cup_{\lambda\neq 0} W^s_{\lambda} \rightarrow \kk$ such that $m\in Y^{\rss}_{\lambda}$ if and only if $\wh(m) = \lambda$.
  \end{inparaenum}
\end{lem}

\begin{proof}[Proof of {Theorem~\ref{thm:cpx-kn} (2)}]    
  To show that $W^s_{\lambda}$ is complex and $G$-invariant, we need to prove that
  \begin{align}
    W^s_{\lambda} &= KY_{\lambda}^{\rss} (= GY_{\lambda}^{\rss}),\label{eq:w-ky}
  \end{align}
  
  To begin with, we note that due to (\ref{eq:grad-m2}), a trajectory of the gradient flow $\varphi_t$ associated with $-f$ satisfies
  \begin{equation}
    \label{eq:t-o}
    \varphi_t(x) \subseteq G \cdot x.
  \end{equation}
  Since $\Sigma_{\lambda} \subseteq W^s_{\lambda}\cap KY_{\lambda}^{\rss}$, the $G$-invariance of $KY_{\lambda}^{\rss}$ and (\ref{eq:t-o}) imply that
  \begin{equation}
    \label{eq:w-i-ky}
    W^s_{\lambda} \subseteq KY_{\lambda}^{\rss}.
  \end{equation}
  For the reverse inclusion, we should prove $Y^{\rss} \subseteq W^{s}_{\lambda}$, that is, the convergence relation displayed as dash dot line in Figure~\ref{fig:yss}.
  Such an inclusion follows from (\ref{eq:no-int}).
In fact, once (\ref{eq:no-int}) is proved, then (\ref{eq:kn-decomp}) and (\ref{eq:w-i-ky}) force $KY_{\lambda}^{\rss} \subseteq W^s_{\lambda}$ to be true.

  To show (\ref{eq:no-int}), firstly, we note that by definition, when $\lambda = 0$, $Y^{\rss}_{0} = W^s_{0}$.
  At the same time, if $\lambda \ne 0$, by Lemma~\ref{lm:yss}, $KY^{\rss}_{\lambda} \subseteq \cup_{\rho \ne 0} W^s_{\rho}$.
  Therefore,
  \begin{equation*}
    KY^{\rss}_{0} \cap KY^{\rss}_{\lambda} = \emptyset,\text{ if } \lambda \neq 0.
  \end{equation*}

  For both $\lambda$ and $\lambda'$ being nonzero and $\lambda \neq \lambda'$, by Lemma~\ref{lm:yss}, if $m\in KY^{\rss}_{\lambda}$ and $m' \in KY^{\rss}_{\lambda'}$, then
  \begin{equation*}
    \wh(m) = \Ad{k}{\lambda} \neq \Ad{k'}{\lambda'} = \wh(m'),
  \end{equation*}
  where $k, k'\in K$ and the inequality is due to $\lambda, \lambda' \in \kt^*_+$ and $\lambda \neq \lambda'$.
  As a result,
  \begin{equation*}
    KY^{\rss}_{\lambda} \cap KY^{\rss}_{\lambda'} = \emptyset
  \end{equation*}
  still holds and the proof of (\ref{eq:no-int}) completes.
\end{proof}

\begin{rem}
By our proof of Theorem~\ref{thm:cpx-kn}, we know that
    \begin{equation}
      \label{eq:ws-ky}
      W^s_{\lambda} = K Y^{\rss}_{\lambda} = G Y^{\rss}_{\lambda},
    \end{equation}
    which can be further refined to
    \begin{equation}
      \label{eq:ws-gy}
      W^s_{\lambda} = G \times_{P_{\lambda}} Y^{\rss}_{\lambda}.
    \end{equation}
    To see this, we need to show that for $g \in G$ and $m\in Y^{\rss}_{\lambda}$, $g\cdot m \in Y^{\rss}_{\lambda}$ implies that $g\in P_{\lambda}$.
    Since $G = K P_{\lambda}$, in fact, we can assume that $g \in K$.
    Then, due to Lemma~\ref{lm:yss},
    \begin{equation*}
      \lambda = \wh(g \cdot m) = \Ad{g}{\wh(m)} = \Ad{g}{\lambda}.
    \end{equation*}
    That is, $g\in K_{\lambda} \subseteq P_{\lambda}$.

    Note that in~\cite[Theorem~6.18]{Kirwan84}, Kirwan's proof of the existence of the stratification, i.e.\ Theorem~\ref{thm:cpx-kn} here, depends on (\ref{eq:ws-gy}), which is different from our arguments.
    Although the proof of Lemma~\ref{lm:yss} (i.e.\ Theorem~\ref{thm:yss}) itself is a little involved, once it is proved, we hope that the proof for (\ref{eq:ws-gy}) given here may be easier to understand compared to Kirwan's original proof.
\end{rem}

\begin{rem}
Perhaps quite surprisingly, it was never observed before that few of the proofs
in Atiyah \cite{Atiyah} and Kirwan \cite{Kirwan84} truly uses the
closedness of the form $\omega$. They are more Morse-theoretic than symplectic.
The sole role played by
the symplectic form is to generate a moment map and set off
a journey --- on the way to the destinations, oftentimes,
only non-degeneracy becomes relevant.
\end{rem}

From here, one sees that Kirwan's inductive cohomological formulas for the
Betti numbers of orbifold $\Psi^{-1}(0)/K$ also remain valid in the new setting (\cite{Kirwan84}).

\bigskip

\section{Properties and applications of Kempf--Ness function}\label{sec:kn-fct}
In this section, we use the same notation and assumption as in Section~\ref{sec:kn}.
And we always assume that $M$ is a compact complex manifold.

In the proof of Theorem~\ref{thm:cpx-kn}, we mention that there is another description, perhaps a more intrinsic one, of the submanifolds $Z^{\rss}_{\lambda}$ and $Y^{\rss}_{\lambda}$ following~\cite{Hesselink_1978un} and~\cite{Woodward_2010aa}.
To formulate such a description, we need to discuss more properties of the Kempf--Ness function $\kn_m$, see Definition~\ref{def:kn-fct}.

\subsection{Behaviors of $\kn_m$ near $\partial(G/K)$}
As we have done implicitly in the definition of the (lifted) Kempf--Ness function, we choose a metric on $\kg$ via the splitting $\kg = \kk \oplus i\kk$ and the metric on $\kk$.
And then the metric on $G$ is defined using the left translation action.

With respect to this metric on $G$, the induced metric on $G/K$ is complete and nonpositive curved, i.e.\ a $\cat$ space.
Denote the distance function on $G/K$ by $d(-,-)$.
Let $\pi: G \rightarrow G/K$.
The geodesic on $G/K$ is of the form $\pi(g\exp(it \xi))$, $g\in G$, $\xi\in \kk$,~\cite[Appendix~C]{Georgoulas_2021mo}.

\begin{lem}
  \label{lm:kn-lc}
  The Kempf--Ness function $\kn_m$ is a Lipschitz and convex function on $G/K$.
\end{lem}
\begin{proof}
  Since $M$ is compact, by the definition of Kempf--Ness function $\kn_m$, at any $\pi(g)\in G/K$, we have
  \begin{equation*}
    \norm{(\grad \kn_m)_{\pi(g)}} = \norm{\Psi(g^{-1}\cdot m)} \le C,
  \end{equation*}
  where $C$ is a constant independent of $g$ and $m$.
  Therefore, for any $q_1,q_2\in G/K$,
  \begin{equation*}
    |\kn_m(q_1) - \kn_m(q_2)| \le (\sup_{q\in G/K} \norm{(\grad \kn_m)_{q}})\cdot d(q_1,q_2)
    \le C d(q_1,q_2),
  \end{equation*}
  which implies that $\kn_m$ is Lipschitz on $G/K$.

  By the definition of convexity of functions on a metric space, we only need to show that $\kn_m$ is convex along any geodesic.
  Take $\xi\in \kk$ and $g\in G$.
  By Lemma~\ref{lm:kn},
  \begin{equation*}
    \odv[order={2}]{}{t}\kn_m(\pi(g\exp(it\xi))) = \odv[order={2}]{}{t}\phi_m(g\exp(it\xi)) \ge \odv[order={2}]{}{t}\phi_{g^{-1}\cdot m}(\exp(it\xi)) \ge 0.
  \end{equation*}
  As a result, $\kn_m$ is convex along the geodesic $\pi(g\exp(it\xi))$.
\end{proof}

We need some asymptotic properties of $\kn_m$.
As a preparation, we give a quick review about the geometry of $G/K$ near the infinity.

As a $\cat$ space, we can define the boundary at infinity $\partial(G/K)$ for $G/K$.
Let $c_i(t): [0,+\infty)\rightarrow G/K$, $i=1,2$, be two unit-speed geodesic rays on $G/K$.
$c_1(t)$ and $c_2(t)$ are said to be equivalent if and only if there exists a constant $C$ such that
\begin{equation*}
  d(c_1(t),c_2(t)) \le C,\quad t\ge 0.
\end{equation*}
$\partial(G/K)$ is defined to be \emph{the collection of equivalence classes of unit-speed geodesic rays on $G/K$}.
For any $\bfa\in \partial(G/K)$, if a unit-speed geodesic ray $c(t)$ is a representative element of $\bfa$, we say that $c(t)$ is asymptotic to $\bfa$.

\begin{rem}
  \label{rk:equi-geod}
  In general, for two different elements $g_1,g_2 \in G$ and $\xi\in \kk$ with the unit norm, the unit-speed geodesics $\pi(g_1 \exp(it \xi))$ and $\pi(g_2 \exp(it \xi))$ are not equivalent in general.
  In fact, such two unit-speed geodesics are equivalent if and only if $g_1^{-1}g_2 \in P_{\xi}$ (the parabolic subgroup of $G$ associated with $\xi$).
To see this, by the left-invariance of the metric on $G/K$,
  \begin{equation*}
    d(\pi(g_1 \exp(it \xi)), \pi(g_2 \exp(it \xi))) = d(\pi(e), \pi (\exp(-it \xi)g_1^{-1}g_2 \exp(it \xi))).
  \end{equation*}
  Then, note that the right hand side of the above equality is bounded for $t \ge 0$ if and only if $\exp(-it \xi)g_1^{-1}g_2 \exp(it \xi)$ is bounded in $G$ for $t \ge 0$, which, in turn, is equivalent to $g_1^{-1}g_2 \in P_{\xi}$.
\end{rem}

Note that for any $q\in G/K$ and $\bfa\in \partial(G/K)$, there is one and only one unit-speed geodesic ray $c(t)$ starting from $q$ such that $c(t)$ is asymptotic to $\bfa$,~\cite[p.~261,~Proposition~8.2]{Bridson_1999aa}.
As a result, there is a natural bijection between $\{\xi\in \kk| \norm{\xi} = 1\}$ and $\partial(G/K)$, which is even a homeomorphism if $\partial(G/K)$ carries the so-called cone topology.
As another consequence of this observation, for $\bfa_1, \bfa_2\in \partial(G/K)$ and $q\in G/K$, one can define an angle as follows.
\begin{equation*}
  \angle_{q}(\bfa_1, \bfa_2) = \angle_q(c_1,c_2),
\end{equation*}
where $c_i$ is the unit-speed geodesic ray starting at $q$ and asymptotic to $\bfa_i$.\footnote{In general, $\angle_q(c_1,c_2)$ is defined by the Alexandrov angle.
  But since $G/K$ is a Riemannian manifold, we can simply define $\angle_q(c_1,c_2)$ to be angle in $\T_q(G/K)$ between the velocity vectors of $c_i$ at $q$.
  For $p_1, p_2, q\in G/K$, we also use the notation $\angle_q (p_1, p_2)$, which is just $\angle_q (c_1, c_2)$ such that $q,p_i$ lie on the geodesic $c_i$.
}

$\partial(G/K)$ has a natural metric called the Tits metric\footnote{Here, we follow~\cite{Kapovich_2009aa} to define the Tits metric.
In~\cite{Bridson_1999aa}, the Tits metric defined here is called the angular metric and the Tits metric in this book is defined to be the length metric of the angular metric.} $\tits(-,-)$.
For $\bfa_1, \bfa_2 \in \partial(G/K)$,
\begin{equation*}
  \tits(\bfa_1, \bfa_2) = \sup_{q\in G/K} \angle_{q}(\bfa_1, \bfa_2).
\end{equation*}
The Tits metric is a complete metric,~\cite[p.~281,~Proposition~9.7]{Bridson_1999aa}.
The topology defined by the Tits metric on $\partial(G/K)$ is strictly finer than the cone topology in general, which can be discrete in some cases.
As a result, Tits metric in general is not a Riemannian metric.
However, $(\partial(G/K),\tits)$ is still a $\mathsf{CAT}(1)$ space~\cite[Theorem~9.13]{Bridson_1999aa}.
Especially, it means that for $(\partial(G/K),\tits)$, the geodesic is also well defined in the metric space sense and convex subsets, as well as convex functions, can be defined in the following way.
A subset $A$ of $\partial(G/K)$ is called convex if and only if for any $\bfa, \bfb\in A$ with $\tits(\bfa,\bfb) < \pi$, $\overline{\bfa\bfb} \subseteq A$ holds, where $\overline{\bfa\bfb}$ is the shortest geodesic segment between $\bfa$ and $\bfb$.
A function $h$ on $\partial(G/K)$ is called convex if and only if $h$ is convex along any geodesic of length at most $\pi$.

The Kempf--Ness function $\kn_m$ defines a natural function $\mu_m$ on $\partial(G/K)$, called the slope of $\kn_m$ at infinity.
\begin{equation*}
  \begin{aligned}
    \mu_m: \partial(G/K)\rightarrow & \quad \mathbb{R},\\
    \bfa \mapsto & \lim_{t\rightarrow +\infty} \frac{\kn_m(c(t))}{t},
  \end{aligned}
\end{equation*}
where $c(t)$ a unit-speed geodesic ray asymptotic to $\bfa$.
By Lemma~\ref{lm:kn-lc}, the limit in the definition of $\mu_m$ always exists and is independent of choice of $c(t)$.
Moreover, by~\cite[Lemma~3.2]{Kapovich_2009aa}, $\mu_m$ is a Lipschitz function on $\partial(G/K)$ with respect to the Tits metric.

In~\cite[Lemma~3.2]{Kapovich_2009aa} and~\cite[Theorem~5.4.2]{Woodward_2010aa}, the authors prove the following property of $\mu_m$, which is crucial for our application of the Kempf--Ness function.\footnote{Although we are only interested in $\kn_m$, this theorem is actually valid for any convex Lipschitz function on $\partial(G/K)$.}
\begin{thm}
  \label{thm:kw} Let $\gamma(t)$ be any gradient trajectory of $-\kn_m$ on $G/K$.
  \begin{enumerate}[label=\normalfont(\arabic*)]
  \item $\{\mu_m \le 0\}$ is a convex subset of $\partial(G/K)$.
    $\mu_m$ is a convex function on $\{\mu_m \le 0\}$ and strictly convex function on $\{\mu_m < 0\}$.
  \item $\kn_m$ is proper and bounded below if and only if $\mu_m > 0$ everywhere on $\partial(G/K)$.
  \item If $\{\mu_m < 0\} \ne \emptyset$,
    \begin{enumerate}[label=\normalfont(\arabic{enumi}\alph*)]
    \item $\mu_m$ has a unique minimum point $\minp \in \partial(G/K)$;
    \item $\minp$ is the asymptotic direction of $\gamma(t)$, that is, there exists $\xi \in \kk$ such that $\pi(\exp(it \xi/\norm{\xi}))$ is asymptotic to $\minp$ and
      \begin{equation}
        \lim_{t\rightarrow +\infty} d(\gamma(t),\pi(\exp(it\xi)))/t = 0;\label{eq:lm-kw1}\\
      \end{equation}
      moreover,
      \begin{equation}
        \lim_{t\rightarrow +\infty} {\kn_m(\gamma(t))}/{(\norm{\xi}t)} = \mu_m(\minp); \label{eq:grad-sl} \\
      \end{equation}
     \item the slope of $\kn_m$ at $\minp$ also satisfies,
      \begin{gather}
        \mu_m(\minp)= -\inf_{q\in G/K}\norm{(\grad \kn_m)_q}.\label{eq:lm-kw2}
\end{gather}
    \end{enumerate}
  \item The gradient of $\kn_m$ satisfies
    \begin{equation}
      \label{eq:lm-kw3}
      \inf_{q\in G/K}\norm{(\grad \kn_m)_q} = \lim_{t\rightarrow +\infty} \norm{(\grad \kn_m)_{\gamma(t)}}.
    \end{equation}
   \item The following three properties are all equivalent to each other:
    \begin{enumerate}[label=\normalfont(\arabic{enumi}\alph*)]
    \item $\{\mu_m < 0\} \ne \emptyset$,
    \item $\inf_{q\in G/K} \norm{\grad \kn_m}_q > 0$,
    \item $\lim_{t\rightarrow +\infty} \norm{(\grad \kn_m)_{\gamma(t)}}> 0$.
    \end{enumerate}
  \end{enumerate}
\end{thm}

\begin{proof}
  \begin{inparaenum}
  \item and \item These two assertions are proved in \cite[Lemma~3.2]{Kapovich_2009aa}.

  \item The property (3a) is also proved in~\cite[Lemma~3.2]{Kapovich_2009aa}.
    The proof of (3b) appears in~\cite[Theorem~5.4.2]{Woodward_2010aa}.
    In~\cite[Theorem~3.1]{Hirai_2024gr}, the authors provide detailed proof of the fact that $\gamma(t)$ converges to $\minp$ in the cone topology as $t\rightarrow +\infty$.
    Since $\kn_m$ is a convex function on $G/K$, $\norm{\grad \kn_m}_{\gamma(t)}$ is a monotonically decreasing function.
    Therefore,
    \begin{equation*}
      d(\gamma(t), \gamma(t+1)) \le \int_{t}^{t+1} \norm{\grad \kn_m}_{\gamma(\tau)} \diff \tau \le \norm{\grad \kn_m}_{\gamma(0)}.
    \end{equation*}
    Then, by~\cite[Theorem~2.1]{Kaimanovich_1987ly}, (\ref{eq:lm-kw1}) also follows from~\cite[Theorem~3.1]{Hirai_2024gr} immediately.

    Since $\kn_m$ is a Lipschitz function, (\ref{eq:grad-sl}) is a consequence of (\ref{eq:lm-kw1}).
    
    For the proof of (\ref{eq:lm-kw2}), taking any $q\in G/K$ and $\bfa\in \partial{(G/K)}$, let $c_q(t)$ be a unit-speed geodesic ray asymptotic to $\bfa$.
    By the convexity of $\kn_m$ and L'H\^{o}pital's rule, we have
    \begin{equation*}
      \odv{}{t}\Big|_{t = 0} \kn_m(c_q(t)) \le \lim_{t\rightarrow +\infty} \odv{}{t} \kn_m(c_q(t)) = \lim_{t\rightarrow +\infty} \kn_m(c_q(t))/t = \mu_m(\bfa).
    \end{equation*}
    Therefore,
    \begin{equation*}
      \norm{(\grad \kn_m)_q} \ge -((\grad \kn_m)_q, c'_q(0)) = -\odv{}{t}\Big|_{t = 0} \kn_m(c_q(t))
      \ge -\mu_{m}(\bfa),
    \end{equation*}
    which gives an inequality needed for the proof of (\ref{eq:lm-kw2}),
    \begin{equation}
      \label{eq:ineq-mw}
      \inf_{q\in G/K} \norm{(\grad \kn_m)_q} \ge -\mu_m(\minp).
    \end{equation}
    The above inequality is called the moment-weight inequality in~\cite{Georgoulas_2021mo} or the weak duality in~\cite{Hirai_2024gr}.
    One can find the proof for the inequality in the reverse direction and finish the proof of (\ref{eq:lm-kw2}) in~\cite[Theorem~6.4]{Georgoulas_2021mo} or~\cite[Theorem~3.1]{Hirai_2024gr}.

  \item Since
    \begin{equation*}
      \inf_{q\in G/K}\norm{(\grad \kn_m)_q} \le \lim_{t\rightarrow +\infty} \norm{(\grad \kn_m)_{\gamma(t)}},
    \end{equation*}
    when $\lim_{t\rightarrow +\infty} \norm{(\grad \kn_m)_{\gamma(t)}} = 0$, (\ref{eq:lm-kw3}) holds automatically.
    Meanwhile, if $\lim_{t\rightarrow +\infty} \norm{(\grad \kn_m)_{\gamma(t)}} > 0$,~\cite[Theorem~3.1]{Hirai_2024gr} shows that (\ref{eq:lm-kw3}) also holds.\footnote{More precisely, in~\cite[Theorem~3.1]{Hirai_2024gr}, the authors prove (\ref{eq:lm-kw3}) with the assumption $\inf_{q\in G/K} \norm{\grad \kn_m}_q > 0$. However, the same proof also works with the weaker condition used here.}

  \item The equivalence between (5a) and (5b) is due to (\ref{eq:ineq-mw}) and~\cite[Lemma~3.4]{Kapovich_2009aa}.
    And the equivalence between (5b) and (5c) follows from (\ref{eq:lm-kw3}).
\end{inparaenum}

\end{proof}

For the application of $\kn_m$, the following relation between $\kn_m$ and $\Psi$ is very convenient.
\begin{lem}
  Let $\gamma(t)$ be the gradient trajectory of $-\kn_m$ on $G/K$ starting at $\pi(e)$, where $e\in G$ is the unit element of $G$.
  Then, for any $t \ge 0$,
  \begin{equation*}
    \norm{(\grad \kn_m)_{\gamma(t))}} = \norm{\Psi(\varphi_t(m))}.
  \end{equation*}
  Especially,
  \begin{equation}
    \label{eq:lim-grad-norm}
    \lim_{t\rightarrow +\infty} \norm{(\grad \kn_m)_{\gamma(t)}} = \lim_{t\rightarrow +\infty} \norm{\Psi(\varphi_t(m))}.
  \end{equation}
\end{lem}

\begin{proof}
  By the definition of the lifted Kempf--Ness function (\ref{eq:def-kn}), we have
  \begin{equation}
    \label{eq:grad-kn}
    (\grad \phi_m)_g =  - g \cdot (i \Psi(g^{-1}\cdot m)).
  \end{equation}
  (\ref{eq:grad-kn}) and (\ref{eq:grad-m2}) implies that
  \begin{equation}
    \label{eq:two-grad}
    (\diff \Lambda)_g((\grad \phi_m)_g) = (\grad f)_{\Lambda(g)},\quad g\in G,
  \end{equation}
  where $\Lambda(g) = g^{-1}\cdot m: G \rightarrow M$ is defined in (\ref{eq:def-lg}).
Let $g_t$ be the gradient trajectory of $-\phi_m$ starting at $e$.
  Then, (\ref{eq:two-grad}) gives a relation between the negative gradient flow on $G$ and $M$, which refines (\ref{eq:t-o}).
  \begin{equation}
    \label{eq:two-flow}
    \varphi_t(m) = g_t^{-1}\cdot m.
  \end{equation}

  Note that $\pi(g_t) = \gamma(t)$.
  By (\ref{eq:grad-kn}) and (\ref{eq:two-flow}), we know that
  \begin{equation*}
    \norm{(\grad \kn_m)_{\gamma(t)}} = \norm{(\grad \phi_m)_{g_t}} = \norm{\Psi(\varphi_t(m))},
  \end{equation*}
  from which the lemma follows.
\end{proof}

Theorem~\ref{thm:kw} gives a useful characterization of $W^s_0$ in terms of the slope function.
\begin{prop}
  \label{prop:ws0}
  Let $m\in M$.
  Then, there exists $\lambda \neq 0$ such that $m \in W^s_{\lambda}$ if and only if the slope function satisfies $\{\mu_m < 0\} \ne \emptyset$.
  Equivalently, $m\in W^s_0$ if and only if $\mu_m \ge 0$.
\end{prop}

\begin{proof}
  Combining (\ref{eq:lim-grad-norm}) and Theorem~\ref{thm:kw} (5), the result follows.
\end{proof}

As another application of Theorem~\ref{thm:kw}, we prove Lemma~\ref{lm:s0}, which gives another characterization of $W^s_0$.

\begin{proof}[Proof of {Lemma~\ref{lm:s0}}]
  Since $0$ is the minimal value of $f = \norm{\Psi}^2/2$, $W^s_0$ will contain an open neighborhood of $\Psi^{-1}(0)$.
  Then openness of $W^s_0$ itself follows from the general properties of ODE.

  Since the $G$-invariance of $W^s_0$ follows from (\ref{eq:w0-gcl}), we only need to show (\ref{eq:w0-gcl}) holds.
  Recall that by (\ref{eq:t-o}), $\varphi_t(m) \subseteq G \cdot m$.
  As a result, if $\lim_{t\rightarrow +\infty} \varphi_t(m) \in \Psi^{-1}(0)$, then $\overline{G\cdot m} \cap \Psi^{-1}(0) \neq \emptyset$.
  That is,
  \begin{equation*}
    W^s_0 \subseteq \{m\in M | \overline{G\cdot m} \cap \Psi^{-1}(0) \ne \emptyset\}.
  \end{equation*}
  On the other hand, let $m \in \overline{G\cdot m} \cap \Psi^{-1}(0) \ne \emptyset$.
  By (\ref{eq:grad-kn}), for $g\in G$,
  \begin{equation*}
    \norm{(\grad \kn_m)_{\pi(g)}} =  \norm{\Psi(g^{-1}\cdot m)}.
  \end{equation*}
  Therefore, $m \in \overline{G\cdot m} \cap \Psi^{-1}(0) \ne \emptyset$ and (\ref{eq:lm-kw3}) implies that
  \begin{equation*}
    \lim_{t\rightarrow +\infty} \norm{(\grad \kn_m)_{\gamma(t)}} = \inf_{q\in G/K}\norm{(\grad \kn_m)_q} = 0.
  \end{equation*}
  Then, by (\ref{eq:lim-grad-norm}), we know $\lim_{t\rightarrow +\infty} \Psi(\varphi_t(m)) = 0$, i.e.\ $m\in W^s_0$.
  
\end{proof}

\subsection{A characterization of $Y^{\rss}_{\lambda}$} With the slope function $\mu_m$, we can define a quantity which is first defined by Hesselink in the algebraic case,~\cite{Hesselink_1978un}.
\begin{defn}
  \label{def:hesselink}
  Let $m\in M$ such that there exists $\lambda \neq 0$ satisfying $m \in W^s_{\lambda}$.
  The minimal weight of $m$, $\wm(m)$, is defined to be the unique unit-length element $\wm(m) \in \kk$ such that $\exp(i t \wm(m))$ is asymptotic to the unique minimum point of $\mu_m$.
  The Hesselink weight of $m$, $\wh(m)$, is defined to be $-(\inf_{\partial(G/K)} \mu_m) \wm(m)$.
\end{defn}

\begin{rem}
  By Theorem~\ref{thm:kw}, we see that $\wh(m)$ is just the $\xi\in \kk$ appearing in (\ref{eq:lm-kw1}).
  As a consequence, $\wh: \cup_{\lambda\neq 0} W^s_{\lambda} \rightarrow \kk$ is a $K$-equivariant map.
  To see this, by properties of the lifted Kempf--Ness function (\ref{eq:equi-kn}), for any $k\in K$, one notices that $kg_tk^{-1}$ is a gradient trajectory of $-\phi_{k\cdot m}$ starting at $e$.
  Therefore, since $\pi(g_t)$ and $\pi(\exp(it \wh(m)))$ satisfy the asymptotic relation (\ref{eq:lm-kw1}), so do $\pi(kg_tk^{-1})$ and $\pi(\exp(it \Ad{k}{\wh(m)}))$, which implies that
  \begin{equation*}
    \wh(k\cdot m)  = \Ad{k}{\wh(m)}.
  \end{equation*}
  Moreover, a similar argument also shows that $\inf_{\partial(G/K)} \mu_m$ is a $K$-invariant function with respect to $m$.
  Hence, the minimal weight $\wm$ is also $K$-equivariant.

  Due to the equivariance of weight functions, in the following, we only consider $m\in M$ such that $\wm(m)$ or $\wh(m)$ lies in $\kt_+$.
\end{rem}

With the above definition, points in $Y^{\rss}_{\lambda}$ can be characterized as follows.
\begin{thm}[{=Lemma~\ref{lm:yss}}]
  \label{thm:yss}
  Suppose $\lambda\in \Psi(\mathrm{crit}(f)) \cap \kt^*_+$, $\lambda\neq 0$.
  Then,
  \begin{inparaenum}[(i)]\item $Y^{\rss}_{\lambda} \subseteq \cup_{\rho\neq 0} W^s_{\rho}$;
  \item $m\in Y^{\rss}_{\lambda}$ if and only if $\wh(m) = \lambda$.
  \end{inparaenum}
\end{thm}

To prove this theorem, we need a simple expression of the slope of $\kn_m$ at infinity in terms of the generalized moment map.
\begin{lem}
  \label{lm:sl-at-inf}
  Let $\xi\in \kk$ with $\norm{\xi} = 1$.
  Assuming that $\pi(\exp(it \xi))$ is asymptotic to $\bfa_{\xi} \in \partial(G/K)$, then the following statements hold.
  \begin{enumerate}[label=\normalfont(\arabic*)]
  \item
    \begin{equation*}
      \mu_m(\bfa_{\xi}) =  \lim_{t\rightarrow +\infty} -\langle \Psi(\exp(-it \xi)\cdot m), \xi\rangle.
    \end{equation*}
    As a consequence, $\wm(m) = \eta$ if and only if
    \begin{equation*}
      \lim_{t\rightarrow +\infty} \langle \Psi(\exp(-it \eta)\cdot m), \eta\rangle = \sup_{\xi\in \kk, \norm{\xi} =1}\lim_{t\rightarrow +\infty} \langle \Psi(\exp(-it \xi)\cdot m), \xi\rangle> 0.
    \end{equation*}
  \item If $K$ is an abelian group, then
    \begin{equation*}
      \mu_m(\bfa_{\xi}) = \sup_{g\in G} -\langle \Psi(g\cdot m), \xi\rangle.
    \end{equation*}
  \end{enumerate}
\end{lem}

\begin{proof}
  \begin{inparaenum}
  \item By the direct calculation, we have
    \begin{equation*}
      \begin{aligned}
        \mu_m(\bfa_{\xi})
        &= \lim_{t \rightarrow +\infty} \kn_m(\pi(\exp(it \xi)))/t \\
        &= \lim_{t\rightarrow +\infty}\odv{}{t} \kn_m(\pi(\exp(it \xi)))\\
        &= \lim_{t\rightarrow +\infty}\odv{}{t} \phi_m(\exp(it \xi))\\
        &= \lim_{t\rightarrow +\infty} -\langle \Psi(\exp(-it \xi)\cdot m), \xi\rangle,
      \end{aligned}
    \end{equation*}
    where for the fourth equality we use (\ref{eq:grad-kn}).

    The result about $\wm(m)$ follows from the definition of minimal weight and Theorem~\ref{thm:kw} directly.

  \item Due to the result of (1), we know that
    \begin{equation*}
      \mu_m(\bfa_{\xi}) \le \sup_{g\in G} -\langle \Psi(g\cdot m), \xi\rangle.
    \end{equation*}
    
    By Remark~\ref{rk:equi-geod}, since $K$ is abelian, for any $g\in G$, $\pi(g\exp(it\xi))$ is asymptotic to $\bfa_{\xi}$.
    Using the convexity of $\kn_m$,
    \begin{multline*}
      \odv{}{t}\Big|_{t = 0} \kn_m(\pi(g\exp(it\xi))) \le \lim_{t\rightarrow +\infty} \odv{}{t} \kn_m(\pi(g\exp(it\xi)))\\
      = \lim_{t\rightarrow +\infty} \kn_m(\pi(g\exp(it\xi)))/t = \mu_m(\bfa_{\xi}).
    \end{multline*}
    And by (\ref{eq:grad-kn}),
    \begin{equation*}
      \odv{}{t}\Big|_{t = 0} \kn_m(\pi(g\exp(it\xi))) = \odv{}{t}\Big|_{t = 0} \phi_m(g\exp(it\xi)) = - \langle\Psi(g^{-1}\cdot m),\xi\rangle.
    \end{equation*}
    The above two inequalities give the desired inequality.
    \begin{equation*}
      \sup_{g\in G} - \langle\Psi(g\cdot m),\xi\rangle \le \mu_m(\bfa_{\xi}).
    \end{equation*}
  \end{inparaenum}
\end{proof}

Now, we can prove Theorem~\ref{thm:yss}.

\begin{proof}[Proof of {Theorem~\ref{thm:yss}}]
  We first show that $Y^{\rss}_{\lambda} \subseteq \cup_{\rho\neq 0} W^s_{\rho}$.
  Let $m\in Y^{\rss}_{\lambda}$.
  By the definition of $Y^{\rss}_{\lambda}$, we have
  \begin{equation*}
    x = \lim_{t\rightarrow +\infty} \exp(-it\lambda)\cdot m \in Z^{\rss}_{\lambda}\subseteq Z_{\lambda}.
  \end{equation*}
  Assuming that $\pi(\exp(it \lambda/\norm{\lambda}))$ is asymptotic to $\bfa_{\lambda}$, by Lemma~\ref{lm:sl-at-inf} and the definition of $Z_{\lambda}$, we have
\begin{multline}
    \label{eq:sl-al2}
    \mu_m(\bfa_{\lambda}) = \lim_{t\rightarrow +\infty} -\langle \Psi(\exp(-it \lambda/\norm{\lambda})\cdot m), \lambda/\norm{\lambda} \rangle \\
    = -\langle \Psi(x), \lambda/\norm{\lambda}\rangle = -\norm{\lambda} < 0.
  \end{multline}
  That is, $\{\mu_m < 0\} \ne \emptyset$.
  Hence, by Proposition~\ref{prop:ws0}, $m \in \cup_{\rho \ne 0} W^s_{\rho}$.

  \vspace{\baselineskip}
  ``$m\in Y^{\rss}_{\lambda}$ $\Longrightarrow$ $\wh(m) = \lambda$'' part.
  
  As $Z^{\rss}_{\lambda} \subseteq Y^{\rss}_{\lambda}$, we first show that if $m \in Z^{\rss}_{\lambda}$ then $\wh(m) = \lambda$.
  By (\ref{eq:min-zl}) and (\ref{eq:def-zss}),  if $m\in Z^{\rss}_{\lambda}$,
  \begin{equation*}
    \lim_{t\rightarrow +\infty} \norm{\Psi(\varphi_t(m))} = \norm{\lambda} > 0.
  \end{equation*}
  Therefore, by Theorem~\ref{thm:kw},
  \begin{equation*}
    -\mu_m(\minp) = \norm{\lambda},
  \end{equation*}
  where $\pi(\exp(it \wm(m)))$ is asymptotic to $\minp$.
  On the other hand, by (\ref{eq:sl-al2}), we also have
  \begin{equation*}
    \mu_m(\bfa_{\lambda})
= -\norm{\lambda}.
  \end{equation*}
Then, by the uniqueness of $\minp$ in Theorem~\ref{thm:kw}, $\minp$ and $\bfa_{\lambda}$ are the same.
  Therefore,
  \begin{equation*}
    \wm(m) = \lambda/\norm{\lambda}\text{ and }\wh(m) = \lambda.
  \end{equation*}

  Now, we turn to the general case.
  Let $m\in Y^{\rss}_{\lambda}$ and define $\bfa_{\lambda}$ and $\minp$ as before.
Due to (\ref{eq:sl-al2}), to show that $\wh(m) = \lambda$, in fact, we only need to show that $\wm(m) = \lambda/\norm{\lambda}$.
  We argue by contradiction, that is, we assume that $\wm(m) \neq \lambda/\norm{\lambda}$ or equivalently
  \begin{equation}
    \label{eq:mineql}
    \minp \neq \bfa_{\lambda}.
  \end{equation}
  
Assume that $m\in W^s_{\lambda'}$.
  Note that at the moment, we don't know whether $\lambda'$ and $\lambda$ are the same or not.
Anyway, using the definition of $Z^{\rss}_{\lambda}$, (\ref{eq:def-zss}), we have
  \begin{equation}
    \label{eq:wll}
    x\in Z^{\rss}_{\lambda} \cap \overline{W^s_{\lambda'}} \subseteq W^s_{\lambda} \cap \overline{W^s_{\lambda'}}.
  \end{equation}

Meanwhile, by Theorem~\ref{thm:kw}, (\ref{eq:lim-grad-norm}) and (\ref{eq:mineql}), the following inequalities hold.
  \begin{equation*}
    \norm{\lambda'} = \lim_{t \rightarrow +\infty} \norm{\Psi(\varphi_t(m))} = -\mu_m(\minp) > -\mu_m(\bfa_{\lambda}) = \norm{\lambda}.
  \end{equation*}
  Then, since the stratification $\{W^s_{\lambda}\}$ is obtained from the negative gradient flow,~\cite[Lemma~10.7]{Kirwan84} gives
  \begin{equation*}
    \overline{W^s_{\lambda'}} \subseteq \bigcup_{\|\nu\| > \|\lambda'\|} W^s_{\nu}.
  \end{equation*}
  The above two equations imply that
  \begin{equation*}
    \overline{W^s_{\lambda'}} \cap W^s_{\lambda} = \emptyset. 
  \end{equation*}
  But this contradicts with (\ref{eq:wll}).
  Therefore, $\wm(m) = \lambda/\norm{\lambda}$ must hold.

  \vspace{\baselineskip}
  ``$m\in Y^{\rss}_{\lambda}$ $\Longleftarrow$ $\wh(m) = \lambda$'' part.
  
  Let $Z$ be a connected component of the critical points set of $\Psi^{\lambda}$ satisfying $x\in Z$.
  Since $\wh(m) = \lambda$, by Lemma~\ref{lm:sl-at-inf} and Theorem~\ref{thm:kw},
  \begin{equation*}
    \Psi^{\lambda/\norm{\lambda}}(Z) = \Psi^{\lambda/\norm{\lambda}}(x) = \lim_{t\rightarrow +\infty} \langle \Psi(\exp(-it \lambda)\cdot m), \lambda/\norm{\lambda}\rangle = \norm{\lambda}.
  \end{equation*}
  
  As we have used in the proof of Theorem~\ref{thm:cpx-kn} (1), two generalized moment maps $\Psi$ and $\Psi_{K_{\lambda}}$ coincide on $Z$ and so do the negative gradient flows $\varphi_t$ and $\varphi_{K_{\lambda},t}$ generated by these two generalized moment maps.
  
  For $p \in Z$, by Lemma~\ref{lm:sl-at-inf}, we have
  \begin{align*}
    \inf_{\partial(G/K)} \mu_p
    &< \lim_{t\rightarrow +\infty} -\langle \Psi(\exp(-it \lambda/\norm{\lambda})\cdot p), \lambda/\norm{\lambda}\rangle \\
    &=  -\langle \Psi(p), \lambda/\norm{\lambda}\rangle =  -\langle \Psi(x), \lambda/\norm{\lambda}\rangle \\
&= - \norm{\lambda} < 0.
\end{align*}
Therefore, it is meaningful to define $\wm(p)$ (or $\wh(p)$). 

  At the same time, by Theorem~\ref{thm:kw} and (\ref{eq:psi-l-psi}), the above inequalities imply that
  \begin{equation}
    \label{eq:phi-phik}
    \lim_{t\rightarrow +\infty} \norm{\Psi_{K_{\lambda}}(\varphi_{K_{\lambda},t}(p))} = \lim_{t\rightarrow +\infty} \norm{\Psi(\varphi_t(p))} = - \inf_{\partial{(G/K)}} \mu_p> 0.
  \end{equation}
  Therefore, the minimal weight of $p$ with respect to the $K_{\lambda}$-action can also be defined.
  We denote it by $w_{\min,K_{\lambda}}(p)$.
  
  Nevertheless, by Theorem~\ref{thm:kw} again, together with Lemma~\ref{lm:sl-at-inf} and (\ref{eq:lim-grad-norm}), we further have
  \begin{multline}
    \label{eq:phi-phik2}
\lim_{t\rightarrow +\infty} \norm{\Psi_{K_{\lambda}}(\varphi_{K_{\lambda},t}(p))} 
    = \lim_{t\rightarrow +\infty} \langle \Psi_{K_{\lambda}}(\exp(-it w_{\min,K_{\lambda}}(p) )\cdot p), w_{\min,K_{\lambda}}(p)\rangle\\
    = \lim_{t\rightarrow +\infty} \langle \Psi(\exp(-it w_{\min,K_{\lambda}}(p) )\cdot p), w_{\min,K_{\lambda}}(p)\rangle,
  \end{multline}
  where the last equality is due to (\ref{eq:psi-l-psi}).
  
  Let $\minp$ be the minimum point of $\mu_p$.
  By the uniqueness of $\minp$, Lemma~\ref{lm:sl-at-inf}, (\ref{eq:phi-phik}) and (\ref{eq:phi-phik2}) imply that $\pi(\exp(-it w_{\min,K_{\lambda}}(p) )\cdot p)$ is asymptotic to $\minp$.
  In other words,
  \begin{equation}
    \label{eq:eq-wt}
    \wm(p) = w_{\min,K_{\lambda}}(p) \in \kk_{\lambda}.
  \end{equation}
  Therefore, by (\ref{eq:eq-wt}), to show $\wm(x) = \lambda$, we only need to show $w_{\min,K_{\lambda}}(x) = \lambda/\norm{\lambda}$.
  Let $\zeta\in \kk_{\lambda}$ with $\norm{\zeta} = 1$.
  We are going to prove that
  \begin{equation}
    \label{eq:z-l}
    \lim_{t\rightarrow +\infty} \langle \Psi(\exp(-it \zeta)\cdot x), \zeta\rangle \le \langle \Psi(x), \lambda/\norm{\lambda}\rangle = \norm{\lambda}.
  \end{equation}

  Let $T_{\zeta}\subseteq K$ be the torus generated by $\zeta$ and $\lambda$ and $T^{\mathbb{C}}_{\zeta} \subseteq G$ be its complexification.
  Then the generalized moment map associated with $T_{\zeta}$-action, $\Psi_{T_{\zeta}}$, is the orthogonal projection of $\Psi$ onto $\kt_{\zeta}$.
  Let $Y = T^{\mathbb{C}}_{\zeta} \cdot m$ be the $T^{\mathbb{C}}_{\zeta}$-orbit of $m$.
  By Theorem~\ref{thm:orbit}, $\Psi_{T_{\zeta}}(\overline{Y})$ is a closed convex polytope in $\kt_{\zeta}$.
  Since $\wh(m) = \lambda$, for any $\xi\in \kt_{\zeta}$ and $\norm{\xi} = 1$, by Lemma~\ref{lm:sl-at-inf}, we have
  \begin{multline*}
    \norm{\lambda} = \lim_{t\rightarrow +\infty} \langle \Psi_{T_{\zeta}}(\exp(-it \lambda/\norm{\lambda})\cdot m), \lambda/\norm{\lambda}\rangle \\
    \ge \lim_{t\rightarrow +\infty} \langle \Psi_{T_{\zeta}}(\exp(-it \xi)\cdot m), \xi\rangle.
  \end{multline*}
  Due to Lemma~\ref{lm:sl-at-inf} (2), the above inequality is equivalent to
  \begin{equation}
    \label{eq:conv-proj}
    \norm{\lambda} = \inf_{p\in \overline{Y}} \langle \Psi_{T_{\zeta}}(p), \lambda/\norm{\lambda} \rangle \ge \inf_{p\in \overline{Y}} \langle \Psi_{T_{\zeta}}(p), \xi\rangle.
  \end{equation}

  Because $\Psi_{T_{\zeta}}(\overline{Y})$ is a closed convex polytope, (\ref{eq:conv-proj}) implies that $\lambda\in \Psi_{T_{\zeta}}(\overline{Y})$ and $\lambda$ is the unique point in $\Psi_{T_{\zeta}}(\overline{Y})$ nearest to the origin.
  Choose $z\in \overline{Y}$ such that $\Psi_{T_{\zeta}}(z) = \lambda$.
  Then
  \begin{equation*}
    \langle \Psi_{T_{\zeta}}(z), \lambda/\norm{\lambda}\rangle = \norm{\lambda} = \langle \Psi_{T_{\zeta}}(x), \lambda/\norm{\lambda} \rangle.
  \end{equation*}
  By Theorem~\ref{thm:orbit} (4),
  \begin{equation}
    \label{eq:z-x-cl}
    z\in \overline{T_{\zeta}^{\mathbb{C}} \cdot x}.
  \end{equation}

  On the other hand, let $V = T^{\mathbb{C}}_{\zeta}\cdot x$ be the orbit of $x$.
  $\Psi_{T_{\zeta}}(\overline{V})$ is also a closed convex subset contained in $\Psi_{T_{\zeta}}(\overline{Y})$.
  By (\ref{eq:z-x-cl}), we have $\lambda \in \Psi_{T_{\zeta}}(\overline{V})$.
  Therefore, $\lambda$ is also the unique point in $\Psi_{T_{\zeta}}(\overline{V})$ nearest to the origin.
  As a result,
  \begin{equation*}
    \norm{\lambda} \ge \langle \Psi_{T_{\zeta}}(z), \xi \rangle \ge \inf_{p\in \overline{V}} \langle \Psi_{T_{\zeta}}(p), \xi\rangle.
  \end{equation*}
  Using Lemma~\ref{lm:sl-at-inf} and taking $\xi$ to be $\zeta$, especially, the above inequality implies that
  \begin{equation*}
    \norm{\lambda} \ge \lim_{t\rightarrow +\infty} \langle \Psi_{T_{\zeta}}(\exp(-it \zeta)\cdot x), \zeta\rangle = \lim_{t\rightarrow +\infty} \langle \Psi(\exp(-it \zeta)\cdot x), \zeta\rangle,
  \end{equation*}
  from which (\ref{eq:z-l}) follows.

  Finally, since (\ref{eq:z-l}) holds for any $\zeta\in \kk_{\lambda}$ with $\norm{\zeta} = 1$, by Lemma~\ref{lm:sl-at-inf}, we have
  \begin{equation*}
    \wm(x) = w_{\min,K_{\lambda}}(x) = \lambda/\norm{\lambda},
  \end{equation*}
  and $\wh(x) = \lambda$.

  Then, by Theorem~\ref{thm:kw} and (\ref{eq:lim-grad-norm}),
  \begin{equation*}
    \| \Psi(\lim_{t\rightarrow +\infty} \varphi_t(x)) \| = \lim_{t\rightarrow +\infty} \norm{\Psi(\varphi_t(x))} = \langle \Psi(x), \lambda/\norm{\lambda}\rangle = \norm{\lambda}.
  \end{equation*}
  Due to $\lim_{t\rightarrow +\infty} \varphi_t(x) = \lim_{t\rightarrow +\infty} \varphi_{K_{\lambda},t}(x) \in Z$,
  \begin{equation*}
    \langle \Psi(\lim_{t\rightarrow +\infty} \varphi_t(x)), \lambda/\norm{\lambda} \rangle = \norm{\lambda}.
  \end{equation*}
  Combining the above two equalities, we get
  \begin{equation*}
    \Psi(\lim_{t\rightarrow +\infty} \varphi_t(x)) = \lambda,
  \end{equation*}
  that is, $x\in Z^{\rss}_{\lambda}$.
  And $m\in Y^{\rss}_{\lambda}$ consequently.
\end{proof}

\begin{rem}
  We would like to add some comments about Theorem~\ref{thm:yss}.

  \begin{inparaenum}

  \item Combining Theorem~\ref{thm:yss} and (\ref{eq:ws-ky}), we obtain another characterization of $W^s_{\lambda}$, $\lambda\ne 0$: $m \in W^s_{\lambda}$ if and only if $\wh(m)$ is conjugate to $\lambda$ in $\kk$, which generalizes a result by~\cite{Kirwan84} and~\cite{Ness_1984aa} in the projective case.

  \item As we have seen, the proof of Theorem~\ref{thm:yss} is based on the existence of the Kempf--Ness function ultimately, whose existence only depends on the existence of the generalized moment map.
    Therefore, whether the manifold $M$ is K\"ahler or not makes no difference for this theorem.
    If $M$ is a projective manifold, Theorem~\ref{thm:yss} is a well-established result in the literature, for example~\cite[Lemma~12.24]{Kirwan84}.
If $M$ is only a K\"ahler manifold, Theorem~\ref{thm:yss} still seems to be well known by experts.
    However, when preparing this notes, we can only find a proof for this case in~\cite{Woodward_2010aa}.\footnote{When this note is nearly finished, we notice a recent paper by Paradan and Ressayre, which also contains a proof of Theorem~\ref{thm:yss} for the K\"ahler case,~\cite[Theorem~4.11]{Paradan_2025hk}.
      Nevertheless, our method is different from theirs.}
    More precisely, the author proves the sufficient part of Theorem~\ref{thm:yss} in~\cite[Theorem~7.2.2]{Woodward_2010aa} and the necessary part is contained the proof of~\cite[Theorem~7.1.7]{Woodward_2010aa}.
    Unfortunately, we find the proof in~\cite{Woodward_2010aa} is scratchy and hard to follow.
    \footnote{The main difficulty we encounter when checking the proof of~\cite{Woodward_2010aa}, is that the issue concerning the convergence of paths is not carefully handled in many places, especially in~\cite[Theorem~5.4.2 and Theorem~7.2.2]{Woodward_2010aa}.}
In view of this situation, we provide a detailed proof of Theorem~\ref{thm:cpx-kn} using the Kempf--Ness function in \S~\ref{sec:kn} and \S~\ref{sec:kn-fct} and hope such a proof may be useful for even the K\"ahler case.
  \end{inparaenum}
\end{rem}

\section{Reduction spaces and complex quotients}

The notations in this section are the same with Section~\ref{sec:kn}.
And we assume that $M$ is a compact manifold.

To simplify the exposition, we only concentrate on the case when $0$ is a regular value of $\Psi$.
Then, $M^{\rs}(\Psi)$, the set of stable points with respect to $\Psi$, is defined to be the unique open stratum of the Kirwan--Ness stratification of $\norm{\Psi}^2$, i.e.\ $W^s_0$.
By Lemma~\ref{lm:s0}, $M^{\rs}(\Psi)$ has two equivalent descriptions.
\begin{align*}
  M^{\rs}(\Psi) &= \{m\in M| \lim_{t\rightarrow +\infty} \varphi_t(m) \in \Psi^{-1}(0)\} \\
  &= \{m\in M | \overline{G\cdot m} \cap \Psi^{-1}(0) \ne \emptyset\}.
\end{align*}

\begin{prop}
  \label{prop:ms}
  $M^{\rs}(\Psi)$ is $G$-invariant, open, and its complement is a union of {subanalytic} subsets.
Furthermore, $M^{\rs}(\Psi) = G \Psi^{-1}(0)$,
and the map
  \begin{align*}
    \varphi_{\infty}: M^{\rs}(\Psi) &\rightarrow \Psi^{-1}(0)\\
    m &\mapsto \lim_{t\rightarrow +\infty} \varphi_t(m)
  \end{align*}
  is a continuous retraction from $M^{s}(\Psi)$ onto $\Psi^{-1}(0)$.
\end{prop}

\begin{proof}
  By Theorem~\ref{thm:cpx-kn}, we know that $M^{s}(\Psi)$ is $G$-invariant and open.
  And its complement is the union of $W^{s}_{\lambda}$, $\lambda\neq 0$, each of which is a complex submanifold, thus a subanalytic subset.

  To show that $M^{s}(\Psi) = G \Psi^{-1}(0)$, we note that the $G$-invariance of $M^{s}(\Psi)$ implies that $G \Psi^{-1}(0) \subseteq M^{s}(\Psi)$.
  On the other hand, let $m\in \Psi^{-1}(0)$.
  Since $0$ is a regular value of $\Psi$, by Proposition~\ref{isotropy}, $\kk_m = 0$.
  Moreover, for any $\xi\in \kk$ and $v \in \T_m \Psi^{-1}(0)$,
  \begin{equation*}
    g(J\xi_{M,m}, v) = \omega(\xi_{M,m}, v) = \langle \diff \Psi_m(v), \xi \rangle = 0.
  \end{equation*}
  Therefore, $J(\kk\cdot m) = (i\kk) \cdot m$ is the orthogonal complement of $\T_m \Psi^{-1}(0)$ in $\T_m M$.
  As a result, $G\Psi^{-1}(0)$ contains a neighborhood of $\Psi^{-1}(0)$ in $M$.
  Then, (\ref{eq:t-o}) implies that $M^{s}(\Psi) \subseteq G \Psi^{-1}(0)$.

  At last, since $0$ is a regular value of $\Psi$, in a suitable local coordinate system, one can check that $\norm{\Psi}^2$ is a Morse--Bott function near $\Psi^{-1}(0)$.
  Therefore, the map $\varphi_{\infty}$ gives a continuous retraction from $M^{s}(\Psi)$ onto $\Psi^{-1}(0)$.
\end{proof}

To show that $M^{s}(\Psi)$ defines a categorical quotient, we need the following result about the orbits in $\Psi^{-1}(0)$.

\begin{prop}
\begin{enumerate}[label=\normalfont(\arabic*)]
\item If $m \in \Psi^{-1}(0)$, then $G \cdot m \cap \Psi^{-1}(0) = K \cdot m$.
\item If $z\in M^s(\Psi)$ and $m = \lim_{t\rightarrow +\infty} \varphi_t(z) \in \Psi^{-1}(0)$, then $m\in G \cdot z$.
\item Suppose that $x$ and $y$ lie in $\Psi^{-1}(0)$ and $K \cdot x \ne K \cdot y$.
  Then there exist disjoint $G$-invariant open neighborhoods of $x$ and $y$ in $M^s(\Psi)$.
\end{enumerate}
\end{prop}

\begin{proof}
  \begin{inparaenum}
  \item As usual, the proof for the K\"ahler case~\cite[Lemma~7.2]{Kirwan84} also works here.
    In fact, one only needs to show that $\exp{i \kk} \cdot m \cap \Psi^{-1}(0) = K \cdot m$.
    Let $\xi\in \kk$ and $\Psi(\exp(i\xi)\cdot m) =0$.
    By (\ref{eq:def-mp}),
    \begin{equation*}
      \odv{}{t} \langle \Psi(\exp(it\xi)\cdot m), \xi\rangle = \omega(\xi_{M,\exp(it\xi)\cdot m}, J \xi_{M,\exp(it\xi)\cdot m)}) \ge 0.
    \end{equation*}
    Then $\Psi(m) = \Psi(\exp(i\xi)\cdot m) =0$ forces $\xi_{M,\exp(it\xi)\cdot m} = 0$ holds for any $t\in [0,1]$.
    Especially, $\xi_{M,m} = 0$ which means that $\exp(it\xi)\cdot m = m$.

  \item As in the proof of Proposition~\ref{prop:ms}, since $0$ is a regular value of $\Psi$, there is an $\varepsilon > 0$ such that the following map is a diffeomorphism,
    \begin{equation*}
      \begin{aligned}
        \kk_{\varepsilon} \times \Psi^{-1}(0) & \rightarrow \exp(i\kk_{\varepsilon}) \Psi^{-1}(0) \subseteq M^{s}(\Psi)\\
        (\xi, x) & \mapsto \exp(i\xi)x
      \end{aligned},
    \end{equation*}
    where $\kk_{\varepsilon} = \{\xi\in \kk| \norm{\xi} \le \xi\}$ and $\exp(i\kk_{\varepsilon}) \Psi^{-1}(0)$ is an open neighborhood of $\Psi^{-1}(0)$ in $M^{s}(\Psi)$.

    Since $\lim_{t\rightarrow +\infty} \varphi_t(z) \in \Psi^{-1}(0)$, there exists $C_0> 0$, when $t> C_0$, we have $\varphi_t(z) \in \exp(i\kk_{\varepsilon/2}) \Psi^{-1}(0)$.
    Namely, there exists $\xi_t\in \kk_{\varepsilon/2}$ and $x_t\in \Psi^{-1}(0)$ such that $\varphi_t(z) = \exp(i\xi_t) x_t$.
    Then, the fact that $\varphi_t(z) \subseteq G \cdot z$ and (1) imply that there exists $y \in \Psi^{-1}(0)$ and $k_t \in K$ such that $x_t = k_t y$ when $t > C_0$.
    Now, we can choose a sequence $\{t_i\}$ such that $\lim_{i\rightarrow+\infty} \xi_{t_i} = \xi_0 \in \kk_{\varepsilon}$ and $\lim_{i\rightarrow+\infty} k_{t_i} = k_0 \in K$.
    As a result,
    \begin{equation*}
      m = \lim_{i\rightarrow+\infty} \varphi_{t_i}(z) = \lim_{i\rightarrow+\infty} \exp(i\xi_{t_i}) k_{t_i} y = \exp(i\xi_0)k_0y.
    \end{equation*}
    Hence, by the definition of $y$, by choosing $t> C_0$, we know that $m\in G\cdot y = G \cdot x_t = G \cdot \varphi_t(z) = G \cdot z$.

  \item Let $V_x$ and $V_y$ be disjoint $K$-invariant neighborhoods of $x$ and $y$ in $\Psi^{-1}(0)$, respectively.
    Then, by Proposition~\ref{prop:ms}, $\varphi^{-1}_\infty (V_x)$ and $\varphi^{-1}_\infty (V_y)$ are disjoint open neighborhoods of $x$ and $y$ in $M$.
    Moreover, for any $z\in M^{s}(\Psi)$, (2) implies that $\varphi_{\infty}(z) \in G\cdot z$.
    As a result, for any $g\in G$,
    \begin{equation*}
      \varphi_{\infty}(g z) \in G \cdot (gz) = G\cdot z = G\cdot \varphi_{\infty}(z).
    \end{equation*}
    By (1), we have $\varphi_{\infty}(gz) \in K \cdot \varphi_{\infty}(z)$.
    Hence, if we further assume that $z\in \varphi^{-1}_\infty (V_x)$, $gz$ must also lie in $\varphi^{-1}_{\infty}(V_x)$ due the $K$-invariance of $V_x$.
    Therefore, $\varphi^{-1}_\infty (V_x)$ and $\varphi^{-1}_\infty (V_y)$ are both $G$-invariant.
  \end{inparaenum}
\end{proof}

One then easily deduces the following two results.

\begin{thm}
  The orbit space $M^{s}(\Psi)/G$ is Hausdorff.
  The inclusion $\Psi^{-1} (0) \hookrightarrow M^{s}(\Psi)$ induces a natural map $\Psi^{-1} (0)/K \rightarrow M^{s}(\Psi)/G$ which is the homeomorphic inverse of the map $M^{s}(\Psi)/G \rightarrow \Psi^{-1} (0)/K$ induced by the retraction $M^{s}(\Psi) \rightarrow \Psi^{-1} (0)$.
\end{thm}

\begin{cor}
  \label{cor:two-cpx-str}
  The quotient $M^{s}(\Psi)/G$ has a separated complex analytic structure induced from that of $M^{s}(\Psi)$.
  If $K$ acts on $\Psi^{-1}(0)$ freely, the natural map $\Psi^{-1} (0)/K \rightarrow M^{s}(\Psi)/G$ is a biholomorphic map, where the complex structure on $\Psi^{-1} (0)/K$ is defined by the reduction construction.
\end{cor}

\begin{proof}[Proof of {Corollary~\ref{cor:two-cpx-str}}]
That the quotient space $M^{s}(\Psi)/G$ has a complex analytic structure follows from a theorem of \cite{Hol}.
  
Since we already know that $M^{s}/G \cong \Psi^{-1} (0)/K$ is Hausdorff, the resulting analytic space is separated.

  When $K$ acts on $\Psi^{-1}(0)$ freely, as in the proof of Proposition~\ref{prop:red}, we have an orthogonal decomposition for $m\in \Psi^{-1}(0)$,
  \begin{equation*}
    \T_m M = H_m \oplus \kk \cdot m \oplus (i\kk) \cdot m,
  \end{equation*}
  where $\T_m \Psi^{-1}(0) = H_m \oplus \kk \cdot m$ and $H_m$ is $J$-invariant.
  Therefore, denoting the projection $\Psi^{-1}(0) \rightarrow \Psi^{-1} (0)/K$ by $\pi$ and the projection $M^{s}(\Psi) \rightarrow M^{s}(\Psi)/G$ by $\pi'$, one has
  \begin{equation*}
    \T_{\pi(m)} (\Psi^{-1} (0)/K) \simeq H_m \simeq \T_{\pi'(m)} (M^{s}(\Psi)/G).
  \end{equation*}
  And the above isomorphisms preserve the almost complex structures on these vector spaces.
  Therefore, $\Psi^{-1} (0)/K \rightarrow M^{s}(\Psi)/G$ is a biholomorphic map.
\end{proof}

Based on the results in Section~\ref{sec:kn-fct}, one can use the Kempf--Ness function to discuss the properties of the stable points set as in the K\"ahler case~\cite{Georgoulas_2021mo,Wang_2021th,Woodward_2010aa}.
The details are left to readers.

\section{Complex quotients are reduction spaces}
\label{dynastable}

Mumford proved
that all projective quotients of a projective variety come from linearized ample line bundles~\cite{GIT}.
In the same vein, we will prove an analogous result in the realm of complex manifolds.
In this section, manifolds are always assumed to be complex and we let $G$ be a complex reductive group.

\begin{defn}
  \label{def:ana-loc-gbd}
  A $G$-equivariant map of $G$-manifolds $f: X \rightarrow Y$ is called a principal fiber bundle over $Y$ with group $G$ if the $G$-action on $Y$ is trivial and
  \begin{equation*}
  \begin{aligned}
    G \times X &\rightarrow X \times_Y X \\
    (g, x) &\mapsto (x, gx)
  \end{aligned}
  \end{equation*}
  is an isomorphism.
  $X$ is called an analytic locally trivial principal $G$-bundle if there is an analytic open cover $\cup_i Y_i = Y$, $G$-isomorphisms $g_i: X_i = f^{-1}(Y_i) \cong G \times Y_i$ and transition morphisms $\phi_{ji}: Y_i\cap Y_j \rightarrow G$ such that the following diagrams are commutative
\begin{equation*}
  \begin{tikzcd}
    f^{-1}(Y_i \cap Y_j) \arrow[r,equal] \arrow[d,"g_i"] & f^{-1}(Y_i \cap Y_j) \arrow[d, "g_j"]\\
    G \times (Y_i \cap Y_j) \arrow[r, "l(\phi_{ji})"]& G \times (Y_i \cap Y_j)
  \end{tikzcd}
\end{equation*}
where $l(\phi_{ji})(g,y) = (\phi_{ji}(y)g, y)$ for $(g,y) \in G \times (Y_i \cap Y_j)$.
If $X,Y$ are algebraic manifolds, the Zariski locally trivial principal $G$-bundle can be defined in a similar way by using the Zariski open cover.
\end{defn}

\begin{thm}
  \label{thm:quotient}
  Let $G$ be the complexification group of a compact torus $K$.
  Assume that $G$ acts on $X$ properly such that the orbit space $X/G$ exists as a complex manifold and that $X \rightarrow X/G$ is an analytic locally trivial $G$-bundle.
  Then there exists a $K$-invariant open subset $V$ of $X$,
  a non-degenerate momentumly closed two-form $\omega$ and a generalized moment map $\Psi: X \rightarrow {\kk}^*$
  defined on $V$
  such that the reduction space of $\Psi$ at $0$ is biholomorphic to $X/G$.
  Moreover, if $X/G$ is a K\"ahler manifold, then $\omega$ can also be chosen to be a K\"ahler form.
\end{thm}

\begin{proof}
  Without loss of generality, we assume that $G$ acts on $X$ freely. By Cartan's decomposition, $G= \exp (i{\kk}) K$ and $G$ real-analytically retracts to $K$.
  Since $G$ is abelian, the map from $G$ to $K$ is a group homomorphism.
  Hence the transition functions in Definition~\ref{def:ana-loc-gbd} can be reduced to be analytic functions
\begin{equation*}
\phi_{ji}: Y_i \cap Y_i \rightarrow K \subseteq G.
  \end{equation*}
As a result, there is a (real analytic) principal $K$-bundle $X'$ over $X/G$ such that
  \begin{equation*}
X \cong G \times_K X',
  \end{equation*}
  where $K$ acts on $G$ on the right.

  Identify $X'$ as a (real) submanifold of $X$, $X$ splits in the following way.
  \begin{equation*}
    \begin{aligned}
     F: X =  (\exp (i\kk)K) \times_K X'  &\cong \kk \times X' \\
      (\exp(ia)k, x) & \mapsto (a, kx),
    \end{aligned}
  \end{equation*}
  where $k\in K$ and $a\in \kk$.
  For $k\in K$ and $(a, x) \in \kk \times X'$, we define the $K$-action on $\kk \times X'$ as $k\cdot(a, x) = (\Ad{k}{a}, kx)$.
  Then, $F$ is also $K$-equivariant.
Via the isomorphism $F$, $X$ induces a complex structure $J$ on $\kk \times X'$ and the $K$-action on $\kk \times X'$ is holomorphic with respect to $J$.
  Note that with respect to $J$, the $K$-action on $\kk \times X'$ induces a holomorphic $G$-action on it in a unique way.
  In the meanwhile, there exists another holomorphic $G$-action on $\kk \times X'$ defined by the $G$-action on $X$ and $F$, which also extends the $K$-action on $\kk \times X'$.
  The uniqueness implies that the above two $G$-actions on $\kk \times X'$ must coincide.
  Therefore, we will not distinguish them in the following.

By choosing a $K$-invariant inner product on $\kk$, $X$ is further $K$-isomorphic to $\kk^* \times X'$.
  By Proposition~\ref{prop:min-cp}, on $\kk \times X' \cong \kk^* \times X'$, there exists a momentumly closed two-form $\omega$ and a generalized moment map $\Psi$ such that $\omega$ is non-degenerate on a $K$-invariant neighborhood of $\{0\} \times X'$, which gives the desired open subset $V$ on $X$.

  To show that $X/G = X'/K$ is biholomorphic to a reduction space, we need to check that we can choose $\omega$ such that $\omega$ is compatible with $J$.
  That is, $\omega$ is $J$-invariant and $\omega(-,J -)$ is a Riemmanian metric.
  
  Let $(p, x) \in \kk \times X'$
  Since $K$ and $G$ are abelian, for $a\in \kk$, $i a$ generates the following tangent vector in $\T_{{p,x}}(\kk \times X') \simeq \kk \oplus \T_x X'$,
  \begin{equation}
    \label{eq:vf-ia}
    (i a)_{\kk^* \times X', (p,x)} = \odv{}{ t}\Big|_{t=0} \exp(i a t)(p, x) = \odv{}{t}\Big|_{t=0} (at + p, x) = (a, 0),
  \end{equation}
  where the second equality is due to the isomorphism $F$.
  As a result, we have the equality,
  \begin{equation}
    \label{eq:kgpx}
\kg \cdot (p,x) = \kk \oplus \kk \cdot x,\quad J(\kk \cdot x) = \kk,
\end{equation}
  which means that $\kk \oplus \kk \cdot x$ is a $J$-invariant subspace of $\kk \oplus \T_x X'$.
  Let $H_{p,x} = \T_x X' \cap J_{(p,x)}(\T_x X')$.
  Note that $(\diff(\exp(ip)\cdot))_{0,x}$ gives an $J$-preserving isomorphism between $\T_{{0,x}}(\kk \times X')$ and $\T_{{p,x}}(\kk \times X')$.
  And since $K$ is abelian, the composition
\begin{equation*}
    \kk \times \T_x X' \simeq \T_{{0,x}}(\kk \times X') \xrightarrow{(\diff(\exp(ip)\cdot))_{0,x}} \T_{{p,x}}(\kk \times X') \simeq \kk \times \T_x X'.
  \end{equation*}
is just the identity map.
  Hence, $H_{p,x}$ is independent of $p$ actually and we denote it by $H_{x}$ in the following.

  By the definition of $W_x$ and (\ref{eq:kgpx}),
  \begin{equation}
    \label{eq:kkx-hx}
    \kk \cdot x \cap H_x  = \emptyset.
  \end{equation}
  And the definition of $W_x$ also implies that
  \begin{multline*}
    \dim H_x \ge \dim \T_x X' + \dim J_{(p,x)}(\T_x X') - \dim X \\
    = 2 \dim \T_x X' - (\dim T_x X' + \dim \kk) = \dim T_x X' - \dim \kk
  \end{multline*}
  Since we assume that $K$-action is free, the above two formulas implies that
  \begin{equation*}
    \T_x X' = \kk \cdot x \oplus H_x.
  \end{equation*}
  Note that $H_x$ is $J$-invariant by definition.
  Consequently, $H = \{H_x| x\in X'\}$ is a horizontal distribution for the fiber bundle structure on $X'$.
  Moreover, the $K$-invariance of $J$ implies that $H$ is a $K$-invariant horizontal distribution.
  Therefore, there exists a $\kk$-valued connection one-form $\theta$ on $X'$ such that $\ker \theta = H$.
  
  Let $\sigma$ be a Hermitian form on $X/G = X'/K$.
  With $\sigma$ and $\theta$, recall that the two-form $\omega$ on $\kk \times X'$ is defined as follows,
  \begin{equation*}
    \omega = (\pi\circ \pr_1)^* \sigma - \diff \langle \pr_2, \theta \rangle,
  \end{equation*}
  where $\pr_1: \kk \times X' \rightarrow X'$ and $\pr_2: \kk \times X' \rightarrow \kk$ are projection maps and $\langle-,-\rangle$ denotes the inner product on $\kk$.
  We first check that $\omega$ is $J$-invariant.
  Let $V_1, V_2$ be two smooth vector fields on $\kk\times X'$.
  To calculate the second term in $\omega$, we use the global formula of the exterior differential again, that is,
  \begin{multline}
    \label{eq:dp-v}
    (\diff \langle \pr_2, \theta \rangle)(V_1, V_2) = (V_1) \langle \pr_2, \theta(V_2) \rangle\\ - (V_2) \langle \pr_2, \theta(V_1) \rangle 
    - \langle \pr_2, \theta[V_1,V_2] \rangle.
  \end{multline}

  \paragraph{Case 1.}
  Let $Z_1,Z_2$ be two $K$-invariant smooth sections of $H$.
  As we have explained, $H_{(p,x)}$ is naturally isomorphic to $H_{x}$.
  We can and will view $Z_i$ as smooth sections defined over $\kk \times X'$.
  Therefore,
  \begin{equation*}
    \omega(\sJ Z_1, \sJ Z_2) = \pi^*\sigma (\sJ Z_1, \sJ Z_2) - (\diff \langle \pr_2, \theta \rangle)(\sJ Z_1, \sJ Z_2)
  \end{equation*}
  By (\ref{eq:kgpx}) and (\ref{eq:kkx-hx}), $H_x$ is isomorphic to $\T_{\pi(x)}(X'/K)$ under $\pi_*$.
  Combined with the fact that $\sigma$ is compatible with the complex structure on $X/G = X'/K$, we have
  \begin{equation}
    \label{eq:omega-jz1}
    \pi^*\sigma (J Z_1, J Z_2) = \pi^*\sigma(Z_1, Z_2).
  \end{equation}
  Meanwhile, since $(J Z_i)_{(p,x)} \in H_x$, by (\ref{eq:dp-v}),
  \begin{equation}
    \label{eq:dj-z}
    (\diff \langle \pr_2, \theta \rangle)(J Z_1, J Z_2)
= - \langle \pr_2, \theta[J Z_1,J Z_2] \rangle.
  \end{equation}
However, since $J$ is integrable on $\kk \times X'$, the vanishing of Nijenhuis tensor implies that
  \begin{equation*}
    [J Z_1, J Z_2] - [Z_1, Z_2] = J([J Z_1, Z_2] + [Z_1, J Z_2]).
  \end{equation*}
  Note that $\pr_1^*(\T X')$ is an integrable subbundle of $\T (\kk \times X')$ and $Z_i,J Z_i$ are smooth sections of $\pr_1^*(\T X')$.
  As a result, the above equality implies that
  \begin{equation*}
    [J Z_1, J Z_2]_{(p,x)} - [Z_1, Z_2]_{(p,x)} \in \T_x X' \cap J_{(p,x)}(\T_x X') = H_x.
  \end{equation*}
  Based on this fact and using (\ref{eq:dj-z}), we have
  \begin{multline}
    \label{eq:omega-jz2}
    (\diff \langle \pr_2, \theta \rangle)(J Z_1, J Z_2) = - \langle \pr_2, \theta[J Z_1, J Z_2] \rangle \\
    = - \langle \pr_2, \theta[Z_1,Z_2] \rangle = (\diff \langle \pr_2, \theta \rangle)(Z_1, Z_2).
  \end{multline}
  Combining (\ref{eq:omega-jz1}) and (\ref{eq:omega-jz2}), we have
  \begin{equation*}
    \omega(J Z_1, J Z_2) = \omega(Z_1, Z_2).
  \end{equation*}

  \paragraph{Case 2.} As in Case 1, let $Z_1$ be a $K$-invariant smooth section of $H$.
  Take $\xi_2 \in \kk$.
  Note that the vector field $(\xi_2)_{\kk \times X'}$ coincides with the pull-back of the vector field $(\xi_2)_{X'}$.
  By (\ref{eq:vf-ia}),
\begin{equation*}
    (\pi\circ \pr_1)^* \sigma (J Z_1, J  (\xi_2)_{\kk \times X'}) = (\pi\circ \pr_1)^* \sigma (J Z_1, \xi_2) = \pi^*\sigma(J Z_1, 0) = 0.
  \end{equation*}
  And
  \begin{equation*}
    (\pi\circ \pr_1)^* \sigma (Z_1, (\xi_2)_{\kk \times X'}) = \pi^*\sigma (J Z_1, (\xi_2)_{X'}) = 0.
  \end{equation*}
  The above two equalities give
  \begin{equation}
    \label{eq:wj2-0}
    (\pi\circ \pr_1)^* \sigma (J Z_1, J  (\xi_2)_{\kk \times X'}) = (\pi\circ \pr_1)^* \sigma (Z_1, (\xi_2)_{\kk \times X'}) = 0.
  \end{equation}
  In the meantime, by (\ref{eq:pr2-theta}) and (\ref{eq:vf-ia}),
  \begin{equation}
    \label{eq:wj2-1}
    \begin{gathered}
      (J Z_1) \langle \pr_2, \theta(J (\xi_2)_{\kk \times X'}) \rangle = (J Z_1) \langle \pr_2, \theta( \xi_2 ) \rangle = 0.\\
      (Z_1) \langle \pr_2, \theta((\xi_2)_{\kk \times X'}) \rangle = (Z_1) \langle \pr_2, \xi_2 \rangle = 0,
    \end{gathered}
  \end{equation}
  where $\langle\pr_2, \xi_2\rangle$ is the function defined on $\kk$ (therefore also on $\kk \times X'$) by $\xi_2$ and the inner product on $\kk$.
  And due to $Z_1$ is a section of $H$, we have
  \begin{equation}
    \label{eq:wj2-2}
    (J (\xi_2)_{\kk \times X'}) \langle \pr_2, \theta(J Z_1) \rangle = 0,\quad ((\xi_2)_{\kk \times X'}) \langle \pr_2, \theta(Z_1) \rangle = 0.
  \end{equation}
  Moreover, since $\kk \times \T_x X'$ is a product manifold, by (\ref{eq:kgpx}),
  \begin{equation}
    \label{eq:wj2-3}
    [J Z_1, J (\xi_2)_{\kk \times X'}] = [J Z_1, \xi_2] = 0.
  \end{equation}
  Also note that $Z_1$ is $K$-invariant, which means that
  \begin{equation}
    \label{eq:wj2-4}
    [Z_1, (\xi_2)_{\kk \times X'}] = - \mathcal{L}_{\xi_2} Z_1 = 0.
  \end{equation}
  Now, by (\ref{eq:dp-v}), (\ref{eq:wj2-0})--(\ref{eq:wj2-4}), we have
  \begin{equation}
    \label{eq:omega-zxi}
    \omega(J Z_1, J (\xi_2)_{\kk \times X'}) = \omega(Z_1, (\xi_2)_{\kk \times X'}) = 0.
  \end{equation}
  Moreover, by taking $Z_1$ to be $J Z_1$, we also have
  \begin{multline}
    \label{eq:omega_zjxi}
    \omega(J Z_1, J (i \xi_2)_{\kk \times X'}) = - \omega(J Z_1, (\xi_2)_{\kk \times X'}) \\
    = \omega(Z_1, J (\xi_2)_{\kk \times X'}) = \omega(Z_1, (i\xi_2)_{\kk \times X'}) = 0.
  \end{multline}

  \paragraph{Case 3.}
  Let $\xi_1, \xi_2 \in \kg = \kk \oplus i \kk$.
  It is simple to check that
  \begin{equation*}
    (\pi\circ \pr_1)^* \sigma (J (\xi_1)_{\kk \times X'}, J  (\xi_2)_{\kk \times X'}) = 0, (\pi\circ \pr_1)^* \sigma ((\xi_1)_{\kk \times X'}, (\xi_2)_{\kk \times X'}) = 0.
  \end{equation*}
  As (\ref{eq:wj2-1}), we also have
  \begin{equation*}
    (J (\xi_1)_{\kk \times X'}) \langle \pr_2, \theta(J (\xi_2)_{\kk \times X'}) \rangle = 0,\;
    ((\xi_1)_{\kk \times X'}) \langle \pr_2, \theta((\xi_2)_{\kk \times X'}) \rangle = 0.
  \end{equation*}
  Since $K$ is abelian, similar to (\ref{eq:wj2-3}) and (\ref{eq:wj2-4}), we have
  \begin{equation*}
    [J (\xi_1)_{\kk \times X'}, J (\xi_2)_{\kk \times X'}] = 0,\quad [(\xi_1)_{\kk \times X'}, (\xi_2)_{\kk \times X'}] = 0.
  \end{equation*}
  By the above three groups of formulas,
  \begin{equation}
    \label{eq:omega-xi12}
    \omega(J (\xi_1)_{\kk \times X'}, J (\xi_2)_{\kk \times X'}) = \omega((\xi_1)_{\kk \times X'}, (\xi_2)_{\kk \times X'}) = 0.
  \end{equation}

  Combining all three cases together, we have verified that $\omega$ is $J$-invariant.
  Next, we check that $\omega(-,J-)$ is a Riemannian metric near $\{0\} \times X'$.
  Clearly, it is sufficient to check this on $\{0\} \times X'$.
  Besides, by (\ref{eq:omega-zxi}), (\ref{eq:omega_zjxi}) and (\ref{eq:omega-xi12}), we know that $H_x$, $\kk\cdot x$ and $\kk$ are orthogonal to each other with respect to $\omega(-,J-)$.
  Hence, we only need to show that $\omega(-,J-)$ is nonnegative on $H_x$, $\kk\cdot x$ and $\kk$ respectively.

  Take an orthonormal basis for $\kk$.
  With respect to this basis, let $(p^1,\cdots,p^{\dim K})$ be coordinates for $\kk$ and let $(\theta^1,\cdots,\theta^{\dim K})$ be the components of $\theta$.
  As in (\ref{eq:pr2-theta}), we have
  \begin{equation*}
    \diff \langle\pr_2, \theta \rangle = \sum_{i=1}^{\dim K} (\diff p^i \wedge \pr_1^*\theta^i + p^i \wedge \pr_1^*\diff \theta^i).
  \end{equation*}
  Thus, at $(0,x) \in \{0\} \times X'$,
  \begin{equation*}
    \omega_{(0,x)} = (\pi\circ \pr_1)^* \sigma - \sum_{i=1}^{\dim K} \diff p^i \wedge \pr_1^*\theta^i.
  \end{equation*}
  For $Z\in H_x$, we then have\footnote{Since the curvature of $\theta$ does not vanish in general, this formula only holds near $\{0\} \times X'$.}
  \begin{equation*}
    \omega_{(0,x)}(H_x, J H_x) = \pi^*\sigma(H_x, J H_x) \ge 0,
  \end{equation*}
  because $\sigma$ is a Hermitian form on $X/G = X'/K$.
  For $\xi \in \kk$, by the definition of the connection one-form, we note that
  \begin{equation*}
    \theta^j((\xi)_{X'})= p^i(\xi).
  \end{equation*}
  As a result, together with (\ref{eq:vf-ia}), we have
  \begin{equation*}
    \omega_{(0,x)}((\xi)_{\kk \times X'}, J (\xi)_{\kk \times X'}) = \omega_{(0,x)}((\xi)_{\kk \times X'}, \xi) = \sum_{i=1}^{\dim K} (p^i(\xi))^2 \ge 0.
  \end{equation*}
  And by the $J$-invariance of $\omega$, the above inequality also holds for $\xi\in i\kk$.
  Hence, we have verified that $\omega(-,J-)$ is a Riemannian metric near $\{0\} \times X'$.

  Although $\omega$ is only compatible with $J$ near $\{0\} \times X'$, the proof of Corollary~\ref{cor:two-cpx-str} shows that we can still apply Corollary~\ref{cor:two-cpx-str} to this situation.
  Therefore, we conclude that the reduction space of $\Psi$ at $0$ is biholomorphic to $X/G$.
\end{proof}

\begin{rem} Some comments about Theorem~\ref{thm:quotient}.
  
  \begin{inparaenum}
  \item Since we assume that $K$ is an abelian group, in the Cartan decomposition of $K$, the part $\exp(i\kk)$ is a subgroup of $G$, which is isomorphic to $\kk$ (as an additive group).
    Therefore, $X'$ in the proof of the theorem is diffeomorphic to $X/ \exp(i\kk)$.

  \item One may notice that the Hermitian form $\omega$ in Theorem~\ref{thm:quotient} is not globally defined.
  The same phenomenon happens in the projective cases. See \cite[Converse~1.12]{GIT}.
  \end{inparaenum}
\end{rem}

\begin{exmp}
  Let $X_\Sigma$ be an arbitrary smooth toric variety defined by a fan $\Sigma$.
  Then $X_\Sigma$ is a geometric quotient $U/T$ of ${\bC}^{|\Sigma(1)|}$ by a subtorus $T$ of $({\bC}^*)^{|\Sigma(1)|}$, where $\Sigma(1)$ is the set of one dimensional subcones of $\Sigma$.
The Cartan decomposition of $T= A \times K$ simply arises from ${\bC}^* = {\bR}^{>0} \times S^1$ and the split is a product of groups.
  Our earlier discussions imply that $U \cong U/A \times A \cong U/A \times {\kk}$.
Now Theorem~\ref{thm:quotient} applies to yield that $X_\Sigma$ arises in our way as well.
  For a combinatorial approach to complex quotients of toric varieties, one can see~\cite{Hu98}.
\end{exmp}

It is natural to expect that the generalized moment maps have applications to moduli spaces of connections on vector bundles over arbitrary complex manifolds.
See Li--Yau \cite{LY} and Li--Yau--Zheng \cite{LYZ}.
For example, it is expected that the Hitchin--Kobayashi correspondence for general complex manifolds is an infinitely dimensional case of the correspondence between complex quotients and reduction spaces (see L\"ubke and Teleman's book \cite{LT}).
To get the generalized moment map for the gauge group action on the space of connections, one replaces a K\"ahler form by a non-degenerate $\partial \overline{\partial}$-closed form,~\cite{LY} and \cite{LYZ}.

Finally, it is a natural problem to compare different complex
quotients.

\section{An example about the non-compact group action}
Till now, we only discuss the compact group action in this note.
However, the definitions of momentumly closed forms and generalized moment maps also make sense for the non-compact group action.
In this section, we discuss a particular example about the non-compact group action.

The example we concerned is the famous Calabi--Eckmann manifolds~\cite{Calabi_1953cl}, which is a class of compact non-K\"ahler complex manifold, the second example of this type found in the history.

\subsection{The Calabi--Eckmann manifolds}
We first review some basic properties of the Calabi--Eckmann manifolds.
Fix a parameter $\tau = a + ib \in \bC - \bR$.
We define a $\bC$-action on $Y = (\bC^{n+1}- \{0\}) \times (\bC^{m+1} - \{0\})$ as follows.
\begin{equation}
  \label{eq:c-act}
  \bC \curvearrowright Y: (t, (x,y)) \mapsto (\exp(t) x, \exp(\tau t) y).
\end{equation}
It is straightforward to check that this action is free and proper.
The \emph{Calabi--Eckmann manifold} is defined to be the quotient manifold of $Y$ with respect to the $\bC$-action,\footnote{Usually, the definition of the Calabi--Eckmann manifold requires that $m,n\ge 1$, because when $m= 0$ or $n = 0$, the Calabi--Eckmann manifold is just the Hopf manifold.
  For our discussion in this section, we only need $m + n \ge 1$.
}
\begin{equation*}
  \CE = Y / \bC.
\end{equation*}
Denote the quotient map by $\mathrm{q}$.
Since the $\bC$-action is holomorphic, $\CE$ is a complex manifold.
$\CE$ is also compact.
To see this, let $S^{2n+1}$ and $S^{2m+1}$ be the standard spheres of dimension $2n+1$ and $2m+1$ respectively.
And denote inclusion map from $S^{2n+1} \times S^{2m+1}$ to $Y$ by $j$.
Then
\begin{equation*}
  \mathrm{q}\circ j : S^{2n+1} \times S^{2m+1} \rightarrow \CE
\end{equation*}
is a diffeomorphism.
In other words, the Calabi--Eckmann manifold $\CE$ can be identified with $S^{2n+1} \times S^{2m+1}$ together with a special complex structure, which is denoted by $J_{\tau}$.

In~\cite{Tsukada_1981ei}, Tsukada gives an explicit description of $J_\tau$, also see~\cite{Matsuo_2009as}.
Denote $g_1$ to be the standard metric on $S^{2n+1}$.
Recall that as a subset of $\bC^{n+1}$, $S^1$ can act on $S^{2n+1}$,
\begin{equation*}
  S^1 \curvearrowright S^{2n+1}: (\exp(i\theta), p) \mapsto \exp(i\theta)p.
\end{equation*}
Let $X_1$ be the vector field on $S^{2n+1}$ generated by this $S^1$-action.
And let $\eta_1 \in \Omega^1(S^{2n+1})$ such that $\eta_1(X_1) = 1$ and $\ker \eta_{1,p} \perp X_{1,p}$ for any $p\in S^{2n+1}$ with respect to $g_1$.
Moreover, there is a natural endomorphism $J_1$ on $\T S^{2n+1}$.
$J_1$ satisfies two properties:
\begin{inparaenum}[(i)]
\item $J_1(X_1) = 0$;
\item for any $p\in S^{2n+1}$ and $v\in \ker \eta_{1,p}$, $J_{1,p}v = iv$, where $v$ is identified with a vector in $\bC^{2n+1}$ via $\T_p S^{2n+1} \subseteq \T_p \bC^{n+1} = \bC^{n+1}$.
\end{inparaenum}
In short, $(g_1, X_1, \eta_1, J_1)$ is the standard Sasakian structure on $S^{2n+1}$.
We define $g_2$, $X_2$, $\eta_2$ and $J_2$ in the same way for $S^{2m+1}$.
Now, the complex structure $J_\tau$ has the following expression\footnote{Our sign conventions are different from~\cite{Matsuo_2009as}.}
\begin{equation*}
  J_{\tau} = J_1 + (\frac{a}{b}X_1 + \frac{a^2 + b^2}{b}X_2) \otimes \eta_1 + J_2 - (\frac{1}{b}X_1 + \frac{a}{b}X_2) \otimes \eta_2.
\end{equation*}

Tsukada also defines a Hermitian metric with respect to $J_{\tau}$.
\begin{equation}
  \label{eq:tsu-m}
  g = g_1 + g_2 - a (\eta_1 \otimes \eta_2 + \eta_2 \otimes \eta_1) + (a^2 + b^2 - 1) \eta_1 \otimes \eta_1
\end{equation}
The Hermitian form for $(g, J_{\tau})$ is
\begin{equation*}
  \omega = \Omega_1 + \Omega_2 + b\eta_1 \wedge \eta_2,
\end{equation*}
where $\Omega_i(-,-) = g_i(J_i-,-) = \diff \eta_i(-,-)$, $i=1,2$.
Therefore, the $(g, J_{\tau})$ is a non-K\"ahler metric.

\subsection{A reduction construction}
In this subsection, we will show that the Calabi--Eckmann manifold $\CE$, as well as the metric (\ref{eq:tsu-m}) on it, can be obtained via a reduction construction with respect to a generalized moment map on $Y$.
The group involved is $\bR$, i.e.\ a non-compact group.

Identify $Y$ with $S^{2n+1} \times \bR \times S^{2m+1} \times \bR$ using the diffeomorphism
\begin{equation}
  \label{eq:iso-sy}
  \begin{aligned}
    S^{2n+1} \times \bR \times& S^{2m+1} \times \bR &&\rightarrow Y \\
    (p, s_1,\quad &q, s_2) &&\mapsto (\exp(s_1)p, \exp(s_2)q)
  \end{aligned}.
\end{equation}
Note that under this identification, we will always view the objects defined on $S^{2n+1}$ or $S^{2m+1}$, e.g.\ $X_i$, as objects defined on $Y$.

Via the homomorphism $\bR \hookrightarrow \bC$, the $\bC$-action on $Y$ induces an $\bR$-action.
The vector field on $Y$ generalized by the $\bR$-action is
\begin{equation*}
  V = \pdv{}{s_1} + bX_2 + a \pdv{}{s_2}.
\end{equation*}
We define a two-tensor on $Y$ in the following way,
\begin{multline*}
  h = (b\eta_1 + \diff s_2 - a \diff s_1)^{\otimes 2} + (\diff s_2 - a\diff s_1)^{\otimes 2}\\
  + (\eta_2 - a \eta_1)^{\otimes 2} + (b\diff s_1 -  \eta_2 + a \eta_1)^{\otimes 2} + g'_1 + g'_2,
\end{multline*}
where $g_i' = g_i - \eta_i\otimes \eta_i = \Omega_i(-,J_i - )$.
\begin{prop}
  $h$ is an $\bR$-invariant Hermitian metric on $Y$.
  The generalized moment map of $h$ with respect to the $\bR$-action is
  \begin{equation*}
    \Psi(p,s_1,q,s_2) = {abs_1 - bs_2}.
  \end{equation*}
\end{prop}

\begin{proof}
  To see that $h$ is a metric, one notices that for any $(x,y)\in Y$,
  \begin{equation*}
    \T_{(x,y)} Y = \bigoplus_{i=1,2}(\ker \eta_{i} \oplus \mathrm{span}_{\bR}\{X_i, \pdv{}{s_i}\}).
  \end{equation*}
  By definition of $g'_i$ and $\eta_i$, $g'_i$ is a metric on $\ker \eta_i$.
  Besides, $b\eta_1 + \diff s_2 - a \diff s_1$, $\diff s_2 - a \diff s_1$, $\eta_2 - a \eta_1$ and $b\diff s_1 - \eta_2 + a \eta_1$ is a basis for the annihilator space of $\ker \eta_1 \oplus \ker \eta_2$.
  Therefore, $h$ is a metric on $Y$.

  Let $J$ be the almost complex structure on $Y$.
  By the definition of $J_i$, we have $J_i|_{\ker \eta_i} = J|_{\ker \eta_i}$, which implies that $g_i'$ is $J$-invariant.
  Via the isomorphism (\ref{eq:iso-sy}), we also have
  \begin{equation*}
    J \pdv{}{s_i} = X_i,\text{ equivalently, } J (\diff s_i) = - \eta_i.
  \end{equation*}
  Therefore, $h$ is $J$-invariant.

  Note that $g_i'$ is invariant under the $S^1$-action on $S^{2n+1}$ or $S^{2m+1}$.
  And by (\ref{eq:iso-sy}), the $\bR$-action on $Y$ induces an $S^1$-rotation on the $S^{2n+1}$ or $S^{2m+1}$ component.
  As a result, $g_i'$ is invariant under the $\bR$-action on $Y$.
  Meanwhile,
  \begin{equation*}
    \mathcal{L}_V (b\eta_1 + \diff s_2 - a \diff s_1)  = b \iota_V \diff \eta_1 + \diff \iota_V(b\eta_1 + \diff s_2 - a \diff s_1) = b \iota_V \Omega_1 = 0.
  \end{equation*}
  Similarly, one can check that the remaining terms in $h$ is also $\bR$-invariant.

  The Hermitian form for $h$ is
  \begin{multline}
    \label{eq:hf-h}
    \omega_h = (b\diff s_1 - \eta_2 + a \eta_1) \wedge (b\eta_1 + \diff s_2 - a \diff s_1) \\
    + (\diff s_2 - a \diff s_1) \wedge (\eta_2 - a \eta_1) + \Omega_1 + \Omega_2.
\end{multline}
  Then
  \begin{equation*}
    \iota_{V} \omega_h = \iota_{V} ((\diff s_2 - a \diff s_1) \wedge (\eta_2 - a \eta_1)) = \diff ({abs_1 - bs_2}).
\end{equation*}
  Note that $(as_1 - s_2)/b$ is also invariant under the $\bR$-action.
  Hence, $(as_1 - s_2)/b$ is a generalized moment map for the $\bR$-action on $Y$.
\end{proof}

\begin{rem}
  With the complex coordinates $\bfZ_1 = (z^1_1,\cdots,z^{n+1}_1)$ and $\bfZ_2 = (z^1_2,\cdots,z^{m+1}_2)$ on $\bC^{n+1}-\{0\}$ and $\bC^{m+1} - \{0\}$ respectively, the differential forms appeared $h$ have the following expressions.
  \begin{equation*}
    \begin{gathered}
      \eta_i = \frac{i}{2} \frac{\bfZ_i \cdot \diff \bar{\bfZ}_i - \bar{\bfZ}_i\cdot \diff \bfZ_i}{|\bfZ_i|^2},\quad
      \diff s_i = \frac{1}{2} \frac{\bfZ_i \cdot \diff \bar{\bfZ}_i + \bar{\bfZ}_i\cdot \diff \bfZ_i}{|\bfZ_i|^2},\\
      \begin{multlined}
\Omega_i = \frac{i}{2} \frac{\diff \bfZ_i \wedge \diff \bar{\bfZ}_i}{|\bfZ_i|^2} - \diff s_i \wedge \eta_i \\
        = \frac{i}{2} \bigl(\frac{\diff \bfZ_i \wedge \diff \bar{\bfZ}_i}{|\bfZ_i|^2} - \frac{(\bar{\bfZ}_i\cdot \diff \bfZ_i) \wedge  (\bfZ_i \cdot \diff \bar{\bfZ}_i)}{|\bfZ_i|^4} \bigr).
      \end{multlined}
    \end{gathered}
  \end{equation*}
  Note that $\Omega_i$ is just the pull-back of the Fubini--Study form.
\end{rem}

\begin{prop}
  \label{prop:red-ce}
  For any $c\in \bR$, the reduction of $(Y, J, h)$ at $c$ using the generalized moment map $\Psi$ coincides with $(\CE, J_\tau, g)$.
\end{prop}

\begin{proof}
  First, since $V$ never vanishes on $Y$, the critical value set of $\Psi$ is empty.
  Then, by (\ref{eq:grad-psixi}), for $(x,y)\in Y$, $(\exp(iu)x,\exp(iu\tau)y)$, $u\in \bR$, intersects $\Psi^{-1}(c)$ once and only once.
  As a result, $\CE = Y/\bC$ is diffeomorphic to $\Psi^{-1}(c)/ \bR$, the reduction of $Y$ at $c$.
  
  Let $m\in \Psi^{-1}(c)$ and $H_m \subseteq \T_m \Psi^{-1}(c)$ such that $H_m \perp \mathrm{span}_{\bR}\{V\}$ with respect to $h$.
  By (\ref{eq:grad-psixi}) again,
  \begin{equation*}
    \T_m Y = H_m \oplus \mathrm{span}_{\bR}\{V_m\} \oplus \mathrm{span}_{\bR}\{JV_m\},
  \end{equation*}
  is an orthogonal decomposition.
  Then, as in the proof of Corollary~\ref{cor:two-cpx-str}, we have
  \begin{equation*}
    \T_{\pi(m)} \CE \simeq \T_m Y/ \mathrm{span}_{\bR}\{V_m,JV_m\} \simeq H_m.
  \end{equation*}
  As complex vector spaces, $\T_m Y/ \mathrm{span}_{\bR}\{V,JV\}$ and $H_m$ are also isomorphic.
  Therefore, the complex structure $J_{\tau}$ on $\CE$ coincides with the complex structure obtained from the reduction construction.

  Note that under the isomorphism (\ref{eq:iso-sy}), the inclusion $j$ just identifies $S^{2n+1} \times S^{2m+1}$ with $S^{2n+1} \times \{0\} \times S^{2m+1} \times \{0\}$.
  Fix $u_0 \in \bR$, let $j_{u_0}$ be the canonical isomorphism between $S^{2n+1} \times S^{2m+1}$ and $S^{2n+1} \times \{0\} \times S^{2m+1} \times \{u_0\}$.
  We have the following commutative diagram.
  \begin{equation*}
    \begin{tikzcd}
      S^{2n+1} \times S^{2m+1} \arrow[d, "\simeq"', "j"] \arrow[rr, "U_0"] & & S^{2n+1} \times S^{2m+1} \arrow[d, "\simeq", "j_{u_0}"'] \\
      S^{2n+1} \times \{0\} \times S^{2m+1} \times \{0\} \arrow[rr, "(-iu_0/b)\cdot"] \arrow[rd, "\simeq"', "\mathrm{q}"] & & S^{2n+1} \times \{0\} \times S^{2m+1} \times \{u_0\} \arrow[ld, "\simeq", "\mathrm{q}"'] \\
      & \CE & 
    \end{tikzcd}.
  \end{equation*}
  where
  \begin{equation*}
    U_0(p,q) = (\exp(-iu_0/b)p, \exp(-iau_0/b)q),\quad (p,q) \in S^{2n+1} \times S^{2m+1}.
  \end{equation*}
Clearly, $U_0$ preserves the complex structure $J_{\tau}$ on $S^{2n+1} \times S^{2m+1}$.
  Then the above commutative diagram implies that $\CE$ induces the same complex structure on $S^{2n+1} \times S^{2m+1}$ via either $j$ or $j_{u_0}$.
  On the other hand, $U_0$ also preserves the metric $g$ on $S^{2n+1} \times S^{2m+1}$.
  Therefore, either $j$ or $j_{u_0}$ defines the same Hermitian metric on $\CE$, which is also denoted by $g$ as before.

  Now, we take $u_0 = -c/b$.
  Then, $j_{u_0} (S^{2n+1} \times S^{2m+1}) \subseteq \Psi^{-1}(c)$.
Note that the projection $\pi: \Psi^{-1}(c) \rightarrow \CE$ is just the restriction of $\mathrm{q}$ to $\Psi^{-1}(c)$.
  Let $\omega_c$ be the reduction Hermitian form on $\CE$ and $i_c: \Psi^{-1}(c) \hookrightarrow Y$ be the injection map.
Then,
  \begin{equation*}
(q \circ j_{u_0})^* \omega_c = j_{u_0}^* (q^* \omega_c) = j_{u_0}^*( \pi^* \omega_c) = j_{u_0}^* (i_{c}^* \omega_h) = j_{u_0}^* \omega_h = \omega,
  \end{equation*}
  where we use (\ref{eq:red-sf}) for the third equality and the last equality follows from the direction calculation.
  Therefore, the reduction Hermitian form on $\CE$ coincides with $\omega$, which implies the two metrics are also the same.
\end{proof}

\begin{rem}
  In view of Theorem~\ref{thm:quotient}, we can say that Proposition~\ref{prop:red-ce} provides an instance for this theorem, but with respect to a non-compact group action.
  Moreover, comparing to Theorem~\ref{thm:quotient}, the Hermitian form $\omega_h$ is defined on $Y$ everywhere.
\end{rem}

The Calabi--Eckmann manifold is a special example of a class of non-K\"ahler manifolds called LVMB manifolds,~\cite{Verjovsky_2019in}.
In fact, by reduction with respect to suitable generalized moment maps, one can construct many other examples of LVMB manifolds.
Moreover, we find the Inoue surface and the primary Kodaira surface, two examples of non-K\"ahler complex surfaces, can also be obtained by the reduction construction.
For all these examples, the group involved is non-compact.
In view of this, perhaps compared to problems associated with compact group action, the generalized moment maps perhaps value more for problems associated with the non-compact group action.

\bibliographystyle{amsplain}

\begin{bibdiv}
\begin{biblist}

\bib{Atiyah}{article}{
      author={Atiyah, M.~F.},
       title={Convexity and commuting {H}amiltonians},
        date={1982},
        ISSN={0024-6093},
     journal={Bull. London Math. Soc.},
      volume={14},
      number={1},
       pages={1\ndash 15},
         url={http://mathscinet.ams.org/mathscinet-getitem?mr=642416},
      review={\MR{642416}},
}

\bib{Atiyah_1983ya}{article}{
      author={Atiyah, M.~F.},
      author={Bott, R.},
       title={The {Y}ang-{M}ills equations over {R}iemann surfaces},
        date={1983},
        ISSN={0080-4614},
     journal={Philos. Trans. Roy. Soc. London Ser. A},
      volume={308},
      number={1505},
       pages={523\ndash 615},
         url={https://mathscinet.ams.org/mathscinet-getitem?mr=702806},
      review={\MR{702806}},
}

\bib{Biliotti_2018re}{article}{
      author={Biliotti, Leonardo},
      author={Ghigi, Alessandro},
       title={Remarks on the abelian convexity theorem},
        date={2018},
        ISSN={0002-9939},
     journal={Proc. Amer. Math. Soc.},
      volume={146},
      number={12},
       pages={5409\ndash 5419},
         url={https://mathscinet.ams.org/mathscinet-getitem?mr=3866878},
      review={\MR{3866878}},
}

\bib{Birtea_2009op}{article}{
      author={Birtea, Petre},
      author={Ortega, Juan-Pablo},
      author={Ratiu, Tudor~S.},
       title={Openness and convexity for momentum maps},
        date={2009},
        ISSN={0002-9947,1088-6850},
     journal={Trans. Amer. Math. Soc.},
      volume={361},
      number={2},
       pages={603\ndash 630},
         url={https://mathscinet.ams.org/mathscinet-getitem?mr=2452817},
      review={\MR{2452817}},
}

\bib{Bourbaki_2005li}{book}{
      author={Bourbaki, Nicolas},
       title={Lie groups and {L}ie algebras. {C}hapters 7--9},
      series={Elements of Mathematics (Berlin)},
   publisher={Springer-Verlag, Berlin},
        date={2005},
        ISBN={3-540-43405-4},
         url={https://mathscinet.ams.org/mathscinet-getitem?mr=2109105},
        note={Translated from the 1975 and 1982 French originals by Andrew
  Pressley},
      review={\MR{2109105}},
}

\bib{Bridson_1999aa}{book}{
      author={Bridson, Martin~R.},
      author={Haefliger, Andr{\'e}},
       title={Metric spaces of non-positive curvature},
      series={Grundlehren der Mathematischen Wissenschaften [Fundamental
  Principles of Mathematical Sciences]},
   publisher={Springer-Verlag, Berlin},
        date={1999},
      volume={319},
        ISBN={3-540-64324-9},
         url={http://mathscinet.ams.org/mathscinet-getitem?mr=1744486},
      review={\MR{1744486}},
}

\bib{Calabi_1953cl}{article}{
      author={Calabi, Eugenio},
      author={Eckmann, Beno},
       title={A class of compact, complex manifolds which are not algebraic},
        date={1953},
        ISSN={0003-486X},
     journal={Ann. of Math. (2)},
      volume={58},
       pages={494\ndash 500},
         url={https://mathscinet.ams.org/mathscinet-getitem?mr=57539},
      review={\MR{57539}},
}

\bib{Condevaux_1988ge}{incollection}{
      author={Condevaux, M.},
      author={Dazord, P.},
      author={Molino, P.},
       title={G\'eom\'etrie du moment},
        date={1988},
   booktitle={Travaux du {S}\'eminaire {S}ud-{R}hodanien de {G}\'eom\'etrie,
  {I}},
      series={Publ. D\'ep. Math. Nouvelle S\'er. B},
      volume={88-1},
   publisher={Univ. Claude-Bernard, Lyon},
       pages={131\ndash 160},
         url={https://mathscinet.ams.org/mathscinet-getitem?mr=1040871},
      review={\MR{1040871}},
}

\bib{DiTerlizzi_2007re}{article}{
      author={Di~Terlizzi, Luigia},
      author={Konderak, Jerzy~J.},
       title={Reduction theorems for manifolds with degenerate 2-form},
        date={2007},
        ISSN={0949-5932},
     journal={J. Lie Theory},
      volume={17},
      number={3},
       pages={563\ndash 581},
         url={https://mathscinet.ams.org/mathscinet-getitem?mr=2351999},
      review={\MR{2351999}},
}

\bib{Diez_2024sy}{article}{
      author={Diez, Tobias},
      author={Rudolph, Gerd},
       title={Symplectic reduction in infinite dimensions},
        date={2024},
      eprint={2409.05829v1},
         url={http://arxiv.org/abs/2409.05829v1},
}

\bib{Duistermaat_1982th}{article}{
      author={Duistermaat, J.~J.},
      author={Heckman, G.~J.},
       title={On the variation in the cohomology of the symplectic form of the
  reduced phase space},
        date={1982},
        ISSN={0020-9910},
     journal={Invent. Math.},
      volume={69},
      number={2},
       pages={259\ndash 268},
         url={https://mathscinet.ams.org/mathscinet-getitem?mr=674406},
      review={\MR{674406}},
}

\bib{Fre-Uh}{book}{
      editor={Freed, Daniel~S.},
      editor={Uhlenbeck, Karen~K.},
       title={Geometry and quantum field theory},
      series={IAS/Park City Mathematics Series},
   publisher={American Mathematical Society, Providence, RI; Institute for
  Advanced Study (IAS), Princeton, NJ},
        date={1995},
      volume={1},
        ISBN={0-8218-0400-6},
         url={https://mathscinet.ams.org/mathscinet-getitem?mr=1338390},
      review={\MR{1338390}},
}

\bib{Friedrich_2000di}{book}{
      author={Friedrich, Thomas},
       title={Dirac operators in {R}iemannian geometry},
      series={Graduate Studies in Mathematics},
   publisher={American Mathematical Society, Providence, RI},
        date={2000},
      volume={25},
        ISBN={0-8218-2055-9},
         url={https://mathscinet.ams.org/mathscinet-getitem?mr=1777332},
        note={Translated from the 1997 German original by Andreas Nestke},
      review={\MR{1777332}},
}

\bib{Georgoulas_2021mo}{book}{
      author={Georgoulas, Valentina},
      author={Robbin, Joel~W.},
      author={Salamon, Dietmar~Arno},
       title={The moment-weight inequality and the {H}ilbert-{M}umford
  criterion---{GIT} from the differential geometric viewpoint},
      series={Lecture Notes in Mathematics},
   publisher={Springer, Cham},
        date={2021},
      volume={2297},
        ISBN={978-3-030-89299-9; 978-3-030-89300-2},
         url={https://mathscinet.ams.org/mathscinet-getitem?mr=4387640},
      review={\MR{4387640}},
}

\bib{Guillemin_1982aa}{article}{
      author={Guillemin, V.},
      author={Sternberg, S.},
       title={Convexity properties of the moment mapping},
        date={1982},
        ISSN={0020-9910},
     journal={Invent. Math.},
      volume={67},
      number={3},
       pages={491\ndash 513},
         url={http://mathscinet.ams.org/mathscinet-getitem?mr=664117},
      review={\MR{664117}},
}

\bib{Guillemin_1994aa}{book}{
      author={Guillemin, Victor},
       title={Moment maps and combinatorial invariants of {H}amiltonian
  {$T^n$}-spaces},
      series={Progress in Mathematics},
   publisher={Birkh\H{a}user Boston, Inc., Boston, MA},
        date={1994},
      volume={122},
        ISBN={0-8176-3770-2},
         url={http://mathscinet.ams.org/mathscinet-getitem?mr=1301331},
      review={\MR{1301331}},
}

\bib{Guillemin_2005ab}{book}{
      author={Guillemin, Victor},
      author={Sjamaar, Reyer},
       title={Convexity properties of {H}amiltonian group actions},
      series={CRM Monograph Series},
   publisher={American Mathematical Society, Providence, RI},
        date={2005},
      volume={26},
        ISBN={0-8218-3918-7},
         url={http://mathscinet.ams.org/mathscinet-getitem?mr=2175783},
      review={\MR{2175783}},
}

\bib{Guillemin_1984no}{incollection}{
      author={Guillemin, Victor},
      author={Sternberg, Shlomo},
       title={A normal form for the moment map},
        date={1984},
   booktitle={Differential geometric methods in mathematical physics
  ({J}erusalem, 1982)},
      series={Math. Phys. Stud.},
      volume={6},
   publisher={Reidel, Dordrecht},
       pages={161\ndash 175},
      review={\MR{767835}},
}

\bib{Hesselink_1978un}{article}{
      author={Hesselink, Wim~H.},
       title={Uniform instability in reductive groups},
        date={1978},
        ISSN={0075-4102,1435-5345},
     journal={J. Reine Angew. Math.},
      volume={303/304},
       pages={74\ndash 96},
         url={https://mathscinet.ams.org/mathscinet-getitem?mr=514673},
      review={\MR{514673}},
}

\bib{Hilgert_1994aa}{article}{
      author={Hilgert, Joachim},
      author={Neeb, Karl-Hermann},
      author={Plank, Werner},
       title={Symplectic convexity theorems and coadjoint orbits},
        date={1994},
        ISSN={0010-437X},
     journal={Compositio Math.},
      volume={94},
      number={2},
       pages={129\ndash 180},
         url={https://mathscinet.ams.org/mathscinet-getitem?mr=1302314},
      review={\MR{1302314}},
}

\bib{Hirai_2024gr}{incollection}{
      author={Hirai, Hiroshi},
      author={Sakabe, Keiya},
       title={Gradient descent for unbounded convex functions on {H}adamard
  manifolds and its applications to scaling problems},
        date={2024},
   booktitle={2024 {IEEE} 65th {A}nnual {S}ymposium on {F}oundations of
  {C}omputer {S}cience---{FOCS} 2024},
   publisher={IEEE Computer Soc., Los Alamitos, CA},
       pages={2387\ndash 2402},
         url={https://mathscinet.ams.org/mathscinet-getitem?mr=4849336},
      review={\MR{4849336}},
}

\bib{Hirzebruch_1988el}{incollection}{
      author={Hirzebruch, Friedrich},
       title={Elliptic genera of level {$N$} for complex manifolds},
        date={1988},
   booktitle={Differential geometrical methods in theoretical physics ({C}omo,
  1987)},
      series={NATO Adv. Sci. Inst. Ser. C: Math. Phys. Sci.},
      volume={250},
   publisher={Kluwer Acad. Publ., Dordrecht},
       pages={37\ndash 63},
         url={https://mathscinet.ams.org/mathscinet-getitem?mr=981372},
      review={\MR{981372}},
}

\bib{Hol}{article}{
      author={Holmann, Harald},
       title={Komplexe {R}\"aume mit komplexen {T}ransformations-gruppen},
        date={1963},
        ISSN={0025-5831,1432-1807},
     journal={Math. Ann.},
      volume={150},
       pages={327\ndash 360},
         url={https://mathscinet.ams.org/mathscinet-getitem?mr=150789},
      review={\MR{150789}},
}

\bib{Hu98}{article}{
      author={Hu, Yi},
       title={Combinatorics and quotients of toric varieties},
        date={2002},
        ISSN={0179-5376,1432-0444},
     journal={Discrete Comput. Geom.},
      volume={28},
      number={2},
       pages={151\ndash 174},
         url={https://mathscinet.ams.org/mathscinet-getitem?mr=1920137},
      review={\MR{1920137}},
}

\bib{Kapovich_2009aa}{article}{
      author={Kapovich, Michael},
      author={Leeb, Bernhard},
      author={Millson, John},
       title={Convex functions on symmetric spaces, side lengths of polygons
  and the stability inequalities for weighted configurations at infinity},
        date={2009},
        ISSN={0022-040X},
     journal={J. Differential Geom.},
      volume={81},
      number={2},
       pages={297\ndash 354},
         url={https://mathscinet.ams.org/mathscinet-getitem?mr=2472176},
      review={\MR{2472176}},
}

\bib{Karshon_1993mo}{article}{
      author={Karshon, Yael},
      author={Tolman, Susan},
       title={The moment map and line bundles over presymplectic toric
  manifolds},
        date={1993},
        ISSN={0022-040X,1945-743X},
     journal={J. Differential Geom.},
      volume={38},
      number={3},
       pages={465\ndash 484},
         url={https://mathscinet.ams.org/mathscinet-getitem?mr=1243782},
      review={\MR{1243782}},
}

\bib{Kaimanovich_1987ly}{article}{
      author={Ka\u{\i}manovich, V.~A.},
       title={Lyapunov exponents, symmetric spaces and a multiplicative ergodic
  theorem for semisimple {L}ie groups},
        date={1987},
        ISSN={0373-2703},
     journal={Zap. Nauchn. Sem. Leningrad. Otdel. Mat. Inst. Steklov. (LOMI)},
      volume={164},
      number={Differentsial\cprime naya Geom. Gruppy Li i Mekh. IX},
       pages={29\ndash 46, 196\ndash 197},
         url={https://mathscinet.ams.org/mathscinet-getitem?mr=947327},
      review={\MR{947327}},
}

\bib{Kawasaki_1981in}{article}{
      author={Kawasaki, Tetsuro},
       title={The index of elliptic operators over {$V$}-manifolds},
        date={1981},
        ISSN={0027-7630,2152-6842},
     journal={Nagoya Math. J.},
      volume={84},
       pages={135\ndash 157},
         url={https://mathscinet.ams.org/mathscinet-getitem?mr=641150},
      review={\MR{641150}},
}

\bib{KM}{article}{
      author={Keel, Se{\'a}n},
      author={Mori, Shigefumi},
       title={Quotients by groupoids},
        date={1997},
        ISSN={0003-486X},
     journal={Ann. of Math. (2)},
      volume={145},
      number={1},
       pages={193\ndash 213},
         url={http://mathscinet.ams.org/mathscinet-getitem?mr=1432041},
      review={\MR{1432041}},
}

\bib{Kempf_1979aa}{incollection}{
      author={Kempf, George},
      author={Ness, Linda},
       title={The length of vectors in representation spaces},
        date={1979},
   booktitle={Algebraic geometry ({P}roc. {S}ummer {M}eeting, {U}niv.
  {C}openhagen, {C}openhagen, 1978)},
      series={Lecture Notes in Math.},
      volume={732},
   publisher={Springer, Berlin},
       pages={233\ndash 243},
         url={https://mathscinet.ams.org/mathscinet-getitem?mr=555701},
      review={\MR{555701}},
}

\bib{Kirwan84b}{article}{
      author={Kirwan, Frances},
       title={Convexity properties of the moment mapping. {III}},
        date={1984},
        ISSN={0020-9910},
     journal={Invent. Math.},
      volume={77},
      number={3},
       pages={547\ndash 552},
         url={http://mathscinet.ams.org/mathscinet-getitem?mr=759257},
      review={\MR{759257}},
}

\bib{Kirwan84}{book}{
      author={Kirwan, Frances~Clare},
       title={Cohomology of quotients in symplectic and algebraic geometry},
      series={Mathematical Notes},
   publisher={Princeton University Press, Princeton, NJ},
        date={1984},
      volume={31},
        ISBN={0-691-08370-3},
         url={https://mathscinet.ams.org/mathscinet-getitem?mr=766741},
      review={\MR{766741}},
}

\bib{Knop_2002co}{article}{
      author={Knop, Friedrich},
       title={Convexity of {H}amiltonian manifolds},
        date={2002},
        ISSN={0949-5932},
     journal={J. Lie Theory},
      volume={12},
      number={2},
       pages={571\ndash 582},
         url={https://mathscinet.ams.org/mathscinet-getitem?mr=1923787},
      review={\MR{1923787}},
}

\bib{Kobayashi_1958aa}{article}{
      author={Kobayashi, Shoshichi},
       title={Fixed points of isometries},
        date={1958},
        ISSN={0027-7630},
     journal={Nagoya Math. J.},
      volume={13},
       pages={63\ndash 68},
         url={http://mathscinet.ams.org/mathscinet-getitem?mr=0103508},
      review={\MR{0103508}},
}

\bib{kollar97}{article}{
      author={Koll{\'a}r, J{\'a}nos},
       title={Quotient spaces modulo algebraic groups},
        date={1997},
        ISSN={0003-486X},
     journal={Ann. of Math. (2)},
      volume={145},
      number={1},
       pages={33\ndash 79},
         url={https://mathscinet.ams.org/mathscinet-getitem?mr=1432036},
      review={\MR{1432036}},
}

\bib{Kollar_2006no}{article}{
      author={Koll\'ar, J\'anos},
       title={Non-quasi-projective moduli spaces},
        date={2006},
        ISSN={0003-486X,1939-8980},
     journal={Ann. of Math. (2)},
      volume={164},
      number={3},
       pages={1077\ndash 1096},
         url={https://mathscinet.ams.org/mathscinet-getitem?mr=2259254},
      review={\MR{2259254}},
}

\bib{Lerman_2005aa}{article}{
      author={Lerman, Eugene},
       title={Gradient flow of the norm squared of a moment map},
        date={2005},
        ISSN={0013-8584},
     journal={Enseign. Math. (2)},
      volume={51},
      number={1-2},
       pages={117\ndash 127},
         url={http://mathscinet.ams.org/mathscinet-getitem?mr=2154623},
      review={\MR{2154623}},
}

\bib{LYZ}{article}{
      author={Li, J.},
      author={Yau, S.-T.},
      author={Zheng, F.},
       title={A simple proof of {B}ogomolov's theorem on class {${\rm VII}_0$}
  surfaces with {$b_2=0$}},
        date={1990},
        ISSN={0019-2082,1945-6581},
     journal={Illinois J. Math.},
      volume={34},
      number={2},
       pages={217\ndash 220},
         url={https://mathscinet.ams.org/mathscinet-getitem?mr=1046563},
      review={\MR{1046563}},
}

\bib{LY}{incollection}{
      author={Li, Jun},
      author={Yau, Shing-Tung},
       title={Hermitian-{Y}ang-{M}ills connection on non-{K}\"ahler manifolds},
        date={1987},
   booktitle={Mathematical aspects of string theory ({S}an {D}iego, {C}alif.,
  1986)},
      series={Adv. Ser. Math. Phys.},
      volume={1},
   publisher={World Sci. Publishing, Singapore},
       pages={560\ndash 573},
      review={\MR{915839}},
}

\bib{LT}{book}{
      author={L{\"u}bke, Martin},
      author={Teleman, Andrei},
       title={The {K}obayashi-{H}itchin correspondence},
   publisher={World Scientific Publishing Co., Inc., River Edge, NJ},
        date={1995},
        ISBN={981-02-2168-1},
         url={https://mathscinet.ams.org/mathscinet-getitem?mr=1370660},
      review={\MR{1370660}},
}

\bib{Marle_1984le}{incollection}{
      author={Marle, C.-M.},
       title={Le voisinage d'une orbite d'une action hamiltonienne d'un groupe
  de {L}ie},
        date={1984},
   booktitle={South {R}hone seminar on geometry, {II} ({L}yon, 1983)},
      series={Travaux en Cours},
   publisher={Hermann, Paris},
       pages={19\ndash 35},
      review={\MR{753857}},
}

\bib{Matsuo_2009as}{article}{
      author={Matsuo, Koji},
       title={Astheno-{K}\"ahler structures on {C}alabi-{E}ckmann manifolds},
        date={2009},
        ISSN={0010-1354,1730-6302},
     journal={Colloq. Math.},
      volume={115},
      number={1},
       pages={33\ndash 39},
         url={https://mathscinet.ams.org/mathscinet-getitem?mr=2475838},
      review={\MR{2475838}},
}

\bib{Meinrenken_1999ab}{article}{
      author={Meinrenken, Eckhard},
      author={Sjamaar, Reyer},
       title={Singular reduction and quantization},
        date={1999},
        ISSN={0040-9383},
     journal={Topology},
      volume={38},
      number={4},
       pages={699\ndash 762},
         url={http://mathscinet.ams.org/mathscinet-getitem?mr=1679797},
      review={\MR{1679797}},
}

\bib{GIT}{book}{
      author={Mumford, D.},
      author={Fogarty, J.},
      author={Kirwan, F.},
       title={Geometric invariant theory},
     edition={Third},
      series={Ergebnisse der Mathematik und ihrer Grenzgebiete (2) [Results in
  Mathematics and Related Areas (2)]},
   publisher={Springer-Verlag, Berlin},
        date={1994},
      volume={34},
        ISBN={3-540-56963-4},
         url={http://mathscinet.ams.org/mathscinet-getitem?mr=1304906},
      review={\MR{1304906}},
}

\bib{Ness_1984aa}{article}{
      author={Ness, Linda},
       title={A stratification of the null cone via the moment map},
        date={1984},
        ISSN={0002-9327},
     journal={Amer. J. Math.},
      volume={106},
      number={6},
       pages={1281\ndash 1329},
         url={http://mathscinet.ams.org/mathscinet-getitem?mr=765581},
        note={With an appendix by David Mumford},
      review={\MR{765581}},
}

\bib{Ortega_2004aa}{book}{
      author={Ortega, Juan-Pablo},
      author={Ratiu, Tudor~S.},
       title={Momentum maps and {H}amiltonian reduction},
      series={Progress in Mathematics},
   publisher={Birkh\H{a}user Boston, Inc., Boston, MA},
        date={2004},
      volume={222},
        ISBN={0-8176-4307-9},
         url={http://mathscinet.ams.org/mathscinet-getitem?mr=2021152},
      review={\MR{2021152}},
}

\bib{Paradan_2025hk}{article}{
      author={Paradan, Paul-Emile},
      author={Ressayre, Nicolas},
       title={{HKKN}-stratifications in a non-compact framework},
        date={2025},
      eprint={2505.03259v3},
         url={http://arxiv.org/abs/2505.03259v3},
}

\bib{Sjamaar_1998aa}{article}{
      author={Sjamaar, Reyer},
       title={Convexity properties of the moment mapping re-examined},
        date={1998},
        ISSN={0001-8708},
     journal={Adv. Math.},
      volume={138},
      number={1},
       pages={46\ndash 91},
         url={http://mathscinet.ams.org/mathscinet-getitem?mr=1645052},
      review={\MR{1645052}},
}

\bib{SL}{article}{
      author={Sjamaar, Reyer},
      author={Lerman, Eugene},
       title={Stratified symplectic spaces and reduction},
        date={1991},
        ISSN={0003-486X},
     journal={Ann. of Math. (2)},
      volume={134},
      number={2},
       pages={375\ndash 422},
         url={http://mathscinet.ams.org/mathscinet-getitem?mr=1127479},
      review={\MR{1127479}},
}

\bib{Smale}{article}{
      author={Smale, Stephen},
       title={On gradient dynamical systems},
        date={1961},
        ISSN={0003-486X},
     journal={Ann. of Math. (2)},
      volume={74},
       pages={199\ndash 206},
         url={https://mathscinet.ams.org/mathscinet-getitem?mr=133139},
      review={\MR{133139}},
}

\bib{Sternberg_1977mi}{article}{
      author={Sternberg, Shlomo},
       title={Minimal coupling and the symplectic mechanics of a classical
  particle in the presence of a {Y}ang-{M}ills field},
        date={1977},
        ISSN={0027-8424},
     journal={Proc. Nat. Acad. Sci. U.S.A.},
      volume={74},
      number={12},
       pages={5253\ndash 5254},
         url={https://mathscinet.ams.org/mathscinet-getitem?mr=458486},
      review={\MR{458486}},
}

\bib{Tian_1998aa}{article}{
      author={Tian, Youliang},
      author={Zhang, Weiping},
       title={An analytic proof of the geometric quantization conjecture of
  {G}uillemin-{S}ternberg},
        date={1998},
        ISSN={0020-9910},
     journal={Invent. Math.},
      volume={132},
      number={2},
       pages={229\ndash 259},
         url={https://mathscinet.ams.org/mathscinet-getitem?mr=1621428},
      review={\MR{1621428}},
}

\bib{Tsukada_1981ei}{article}{
      author={Tsukada, Kazumi},
       title={Eigenvalues of the {L}aplacian on {C}alabi-{E}ckmann manifolds},
        date={1981},
        ISSN={0025-5645,1881-1167},
     journal={J. Math. Soc. Japan},
      volume={33},
      number={4},
       pages={673\ndash 691},
         url={https://mathscinet.ams.org/mathscinet-getitem?mr=630631},
      review={\MR{630631}},
}

\bib{Verjovsky_2019in}{incollection}{
      author={Verjovsky, Alberto},
       title={Intersection of quadrics in {$\Bbb C^n$}, moment-angle manifolds,
  complex manifolds and convex manifolds},
        date={2019},
   booktitle={Complex non-{K}\"ahler geometry},
      series={Lecture Notes in Math.},
      volume={2246},
   publisher={Springer, Cham},
       pages={163\ndash 240},
      review={\MR{3972002}},
}

\bib{Wang_2021th}{article}{
      author={Wang, Xiangsheng},
       title={On the complex structure of symplectic quotients},
        date={2021},
        ISSN={1674-7283},
     journal={Sci. China Math.},
      volume={64},
      number={12},
       pages={2719\ndash 2742},
         url={https://mathscinet.ams.org/mathscinet-getitem?mr=4343064},
      review={\MR{4343064}},
}

\bib{Weinstein_1978un}{article}{
      author={Weinstein, A.},
       title={A universal phase space for particles in {Y}ang-{M}ills fields},
        date={1978},
        ISSN={0377-9017},
     journal={Lett. Math. Phys.},
      volume={2},
      number={5},
       pages={417\ndash 420},
         url={https://mathscinet.ams.org/mathscinet-getitem?mr=507025},
      review={\MR{507025}},
}

\bib{Woodward_2010aa}{article}{
      author={Woodward, Chris},
       title={Moment maps and geometric invariant theory},
        date={2010},
     journal={Les cours du CIRM},
      volume={1},
      number={1},
       pages={55\ndash 98},
         url={http://eudml.org/doc/116365},
}

\end{biblist}
\end{bibdiv}

\vskip .4cm

\end{document}